\documentclass[a4paper,10pt,english,leqno]{amsart}
\usepackage{dsfont}
\usepackage{mathrsfs}
\usepackage{mathtools}
\usepackage{amssymb}
\usepackage{eufrak}
\usepackage{comment}

\usepackage{color}
\usepackage[pagebackref=true,colorlinks=false]{hyperref}

\definecolor{darkgreen}{rgb}{0.5,0.25,0}
\definecolor{darkblue}{rgb}{0,0,0.56}
\definecolor{answerblue}{rgb}{0,0,0.75}

\hypersetup{colorlinks,breaklinks,
            linkcolor=darkblue,urlcolor=darkblue,
            anchorcolor=darkblue,citecolor=darkblue}

\numberwithin{equation}{section}
\allowdisplaybreaks[1]

\usepackage{environ}
\makeatletter
\NewEnviron{Lalign}{\tagsleft@true\begin{align}\BODY\end{align}}
\makeatother

\usepackage[shortcuts,cyremdash]{extdash}

\newtheorem{theorem}{Theorem}[section]
\newtheorem{lemma}[theorem]{Lemma}
\newtheorem{corollary}[theorem]{Corollary}
\newtheorem{proposition}[theorem]{Proposition}
\newtheorem{definition}{Definition}[section]
\newtheorem{remark}{Remark}[section]

\newcommand{\R}{\mathbb{R}}

\newcommand{\dy}{\,dy}

\newcommand{\seq}[1]{\left\{#1\right\}}
\newcommand{\abs}[1]{\left|#1\right|}

\newcommand{\norm}[1]{\left\| #1\right\|}

\DeclareMathOperator*{\esssup}{ess\,sup}

\newcommand{\Z}{\mathbb Z}

\renewcommand{\P}{\mathbb P}   
\newcommand{\EE}{\mathbb E}    

\DeclareMathOperator{\Div}{div}
\newcommand{\Divh}{\Div_h}

\newcommand{\Grad}{\mbox{grad}_h\, }
\newcommand{\action}[2]{\left\langle #1,#2 \right\rangle} 

\newcommand{\uk}{\mathcal{U}_\kappa}
\newcommand{\tml}{\mathcal{T}^m_l(\tilde{X}_\kappa)}
\DeclareMathOperator{\supp}{supp}
\DeclareMathOperator{\dive}{div}

\theoremstyle{example}

\title[Stochastic continuity equations on manifolds]
{Renormalization of stochastic continuity equations on Riemannian manifolds}

\author[L. Galimberti]{Luca Galimberti}
\address[Luca Galimberti]
{\newline Department of Mathematical Sciences
\newline NTNU Norwegian University of Science and Technology
\newline NO-7491 Trondheim, Norway} 
\email[]{luca.galimberti@ntnu.no}

\author[K. H. Karlsen]{Kenneth H. Karlsen}
\address[Kenneth Hvistendahl Karlsen]
{\newline Department of mathematics
\newline University of Oslo
\newline P.O. Box 1053,  Blindern
\newline N--0316 Oslo, Norway} 
\email[]{kennethk@math.uio.no}

\date{\today}



\thanks{This work was supported by the Research Council 
of Norway through the project Stochastic Conservation Laws (250674/F20). To appear in \textit{Stochastic Processes and their Applications}.}

\begin{document}

\begin{abstract}

We consider the initial-value problem for 
stochastic continuity equations of the form
$$
\partial_t \rho + \Divh \left[\rho \left(u(t,x) 
+ \sum_{i=1}^N a_i(x)\circ \frac{dW^i}{dt}\right)\right] = 0,
$$
defined on a smooth closed Riemannian manifold $M$ 
with metric $h$, where the Sobolev regular velocity 
field $u$ is perturbed by Gaussian noise 
terms $\dot{W}_i(t)$ driven by smooth spatially dependent vector fields $a_i(x)$ 
on $M$. Our main result is that weak ($L^2$) solutions 
are renormalized solutions, that is, if $\rho$ is a weak solution, then the 
nonlinear composition $S(\rho)$ is a weak solution as well, for any 
``reasonable" function $S:\R\to\R$. The proof consists of a 
systematic procedure for regularizing tensor fields 
on a manifold, a convenient choice of atlas to simplify technical 
computations linked to the Christoffel symbols, and several DiPerna-Lions type 
commutators $\mathcal{C}_\varepsilon (\rho,D)$ between (first/second order) 
geometric differential operators $D$ and the regularization device 
($\varepsilon$ is the scaling parameter).   This work, which is related to 
the ``Euclidean" result in \cite{Punshon-Smith:2017aa}, 
reveals some structural effects that noise and 
nonlinear domains have on the dynamics of weak solutions. 
\end{abstract}

\maketitle

{\small \tableofcontents}


\section{Introduction}\label{sec:intro}

For a number of years many researchers appended new effects and features 
to partial differential equations (PDEs) in fluid mechanics in order to better 
account for various physical phenomena. An interesting example arises when a 
hyperbolic PDE is posed on a curved manifold instead of 
a flat Euclidean domain, in which case the curvature 
of the domain makes nontrivial alterations to 
the solution dynamics \cite{Amorim:2005aa,Ben-Artzi:2007aa,Rossmanith:2004aa}. 
Relevant applications include geophysical flows and general relativity. 
Another example is the rapid rise in the use of stochastic processes to 
extend the scope of hyperbolic PDEs (on Euclidean domains) 
in an attempt to achieve better understanding of turbulence. 
Randomness can enter the PDEs in different ways, such as 
through stochastic forcing or in uncertain system parameters (fluxes). 
Generally speaking, the mathematical literature 
for stochastic partial differential equations (SPDEs) 
on manifolds is at the moment in short supply 
\cite{Elliott:2012aa,Galimberti:2018aa,Gyongy:1993aa,Gyongy:1997aa}. 
In this paper we consider stochastic continuity equations with a non-regular 
velocity field that is perturbed by Gaussian noise terms powered by spatially 
dependent vector fields. In contrast to the existing literature, 
the main novelty is indeed that we pose 
these equations on a curved manifold, being specifically 
interested in the combined effect of noise and nonlinear domains on the 
dynamics of weak solutions. 

Fix a $d$-dimensional ($d\geq 1$) smooth Riemannian 
manifold $M$, endowed with a metric $h$. 
We assume $M$ to be compact, connected, oriented, and without 
boundary. We are interested in the initial-value problem for 
the stochastic continuity equation
\begin{equation}\label{eq:target}
	d\rho + \Divh(\rho\, u)\,dt + \sum_{i=1}^N \Divh( \rho \,a_i)\circ dW^i(t) = 0
	\quad \mbox{on } \, [0,T]\times M, 
\end{equation}
where $W^1,\ldots,W^N$ are independent Wiener processes, 
$a_1,\ldots ,a_N$ are smooth vector fields on $M$ (i.e., 
first order differential operators on $M$), the symbol 
$\circ$ refers to the Stratonovich interpretation of stochastic integrals, 
$u:[0,T]\times M\to TM$ is a time-dependent $W^{1,2}$ vector field on $M$ 
(a rough velocity field), $\Div_h$ is the divergence 
operator linked to the manifold $(M,h)$, and 
$\rho=\rho(\omega,t,x)$ is the unknown (density of a mass distribution) 
that is sought up to a fixed final time $T>0$. 
The equation \eqref{eq:target} is supplemented 
with initial data $\rho(0)=\rho_0\in L^2$ on $M$.

In the deterministic case ($a_i\equiv 0$, $M=\R^d$), the well-posedness of 
weak solutions follows from the theory of renormalized solutions 
due to DiPerna and Lions \cite{DL89}. A key step in this theory relies 
on showing that weak solutions are renormalized solutions, i.e., 
if $\rho$ is a weak solution, then $S(\rho)$ is a weak 
solution as well, for any ``reasonable" 
nonlinear function $S:\R\to\R$. The validity of this chain rule property 
depends on the regularity of the velocity field $u$. 
DiPerna and Lions proved it in the case that $u$ is $W^{1,p}$-regular 
in the spatial variable, while Ambrosio \cite{Ambrosio:2004aa} 
proved it for $BV$ velocity fields. An extension of the 
DiPerna-Lions theory to a class of Riemannian manifolds can be 
found in \cite{Fang:2011aa} (we will return to this paper below).

The well-posedness of stochastic transport/continuity equations with ``Lipschitz" 
coefficients (defined on Euclidean domains) is classical in the literature 
and has been deeply analyzed in Kunita's works 
\cite{Chow:2015aa,Kunita}. In \cite{Attanasio:2011fj} the renormalization property 
is established for stochastic transport equations with irregular ($BV$) velocity field $u$ 
and ``constant" noise coefficients ($a_i\equiv 1$). Moreover, they proved that the 
renormalization property implies uniqueness without the usual 
$L^\infty$ assumption on the divergence of $u$, thereby providing an example 
of the so-called ``regularization by noise" phenomenon. 
In recent years ``regularization by noise`` has been a recurring theme 
in many papers on the analysis of stochastic transport/continuity 
equations, a significant part of it motivated 
by \cite{Flandoli-Gubinelli-Priola}, see e.g.~\cite{Beck:2014aa,Duboscq:2016aa,Fedrizzi:2013aa,Flandoli:2011aa,
Neves:2015aa,Neves:2016aa,Mohammed:2015aa,Zhang:2010aa}.

Recently \cite{Punshon-Smith:2018aa,Rossmanith:2004aa} the 
renormalization property was established for stochastic continuity equations with 
spatially dependent noise coefficients, written 
in It\^{o} form and defined on an Euclidean domain. 
In the one-dimensional case and without a ``deterministic'' 
drift term, the equations analyzed 
in \cite{Punshon-Smith:2018aa} take the form
\begin{equation}\label{eq:PunSmi}
	\partial_t \rho  +\partial_{x}\left(\sigma \rho\right) \frac{dW(t)}{dt} 
	=  \partial_{xx}^2\left(\frac{\sigma^2}{2} \rho\right),
	 \qquad (t,x)\in [0,T]\times \R,
\end{equation}
where $\sigma=\sigma(x)$ is an irregular coefficient that 
belongs to $W^{1,\frac{2p}{p-2}}_{\text{loc}}$, while $\rho$ is 
an $L^p$ weak solution ($p\ge 2$). The derivation of the (renormalized) 
equation satisfied by $F(\rho)$, for any sufficiently smooth $F:\R\to\R$, 
is based on regularizing (in $x$) the weak solution $\rho$ 
by convolution with a standard mollifier 
sequence $\seq{J_\varepsilon(x)}_{\varepsilon>0}$, 
$\rho_\varepsilon:=J_\varepsilon \star \rho$, using the It\^{o} (temporal) and 
classical (spatial) chain rules to compute $F(\rho_\varepsilon)$, and 
deriving commutator estimates to control the regularization error.  
A key insight in \cite{Punshon-Smith:2018aa}, also 
needed in one of the steps in our renormalization 
proof for \eqref{eq:target}, is the identification 
of a ``second order'' commutator, which is crucial 
to conclude that the regularization error converges to zero, without 
having to assume some kind of ``parabolic'' regularity 
like $\sigma\partial_x \rho\in L^2$ --- the 
nature of the SPDE \eqref{eq:PunSmi} is hyperbolic not parabolic, 
so this regularity is not available (at variance 
with \cite{Le-Bris:2008pb}). To be a bit more precise, 
the ``second order" commutator in \cite{Punshon-Smith:2018aa} takes the form
\begin{align*}
	\mathcal{C}_2(\varepsilon;\varrho,\sigma) & :=
	\frac{\sigma^2}{2}\partial_{xx}^2 \varrho_\varepsilon
	-\sigma \partial_{xx}^2 \left( \sigma \varrho \right)_\varepsilon
	+\partial_{xx}^2
	\left(\frac{\sigma^2}{2}\varrho\right)_\varepsilon
	\\ & = \frac12 \int_{\R}
	\partial_{xx}^2 J_\varepsilon(x-y)
	\left(\sigma(x)-\sigma(y)\right)^2
	\varrho(y)\dy,
\end{align*}
where $\varrho\in L^p_{\text{loc}}(\R)$ and 
$\sigma=\sigma(x)\in W^{1,q}_{\text{loc}}(\R)$, $p,q\in [1,\infty]$. 
It is proved in \cite{Punshon-Smith:2018aa} that, as $\varepsilon\to 0$, 
$\mathcal{C}_2(\varepsilon;\varrho,\sigma)\to (\partial_x \sigma)^2 \varrho$ in 
$L^r_{\text{loc}}(\R)$ with $\frac{1}{r}=\frac{1}{p}+\frac{2}{q}$. 

Modulo a deterministic drift term (which we do not include), 
the equation \eqref{eq:PunSmi} can also be written in the form
\begin{equation}\label{eq:GalKar}
	\partial_t \rho  +\partial_{x}\left(\sigma \rho\right) \frac{dW(t)}{dt} 
	=  \partial_x\left(\frac{\sigma^2}{2} \partial_x \rho \right).
\end{equation}
This particular equation is similar to the equation 
studied in \cite{Galimberti:2018aa}, which arises in  the kinetic formulation 
of stochastically forced hyperbolic conservation laws (on manifolds). 
The uniqueness proof in \cite{Galimberti:2018aa} relies on 
writing the equation satisfied by $F(\rho)=\rho^2$. 
In the Euclidean setting, one is lead to control the following error term, 
linked to the second order differential operator 
in \eqref{eq:GalKar} and the ``It\^{o} correction":
$$
\mathcal{R}(\varepsilon):=\abs{\, \int  \partial_x \varrho_{\varepsilon}
\left(\sigma^2 \partial_x \varrho \right)_\varepsilon
-\Bigl(\left(\sigma \partial_\xi  \varrho \right)_\varepsilon \Bigr)^2\, dx},
$$    
again without imposing a condition like $\sigma\partial_x \varrho\in L^2$. 
Nevertheless, in the kinetic formulation of conservation laws one has access to 
additional structural information, namely that $\partial_x \rho$ is 
a bounded measure. In \cite{Galimberti:2018aa} we use this, and the observation
\begin{align*}
	\mathcal{R}(\varepsilon)
	& = \frac12 \int \left(\sigma(y)-\sigma(\bar y)\right)^2 
	(\partial_x  \varrho)(y)(\partial_x  \varrho)(\bar y)
	J_\varepsilon(x-y) J_\varepsilon(x-\bar y)
	\, dy \, d\bar y \, dx,
\end{align*}
to establish that $\mathcal{R}(\varepsilon)\to 0$ as $\varepsilon\to 0$. 
The detailed handling of error terms like $\mathcal{R}(\varepsilon)$ becomes 
significantly more complicated on a curved 
manifold, cf.~\cite{Galimberti:2018aa} for details.

Let us return to the equation \eqref{eq:target}. 
Our main result is the renormalization 
property for weak $L^2$ solutions, roughly speaking 
under the assumption that $u(t,\cdot)$ is a $W^{1,2}$ vector field 
on $M$, whereas $a_1,\ldots ,a_N$ are smooth vector fields on $M$.
As corollaries, we deduce the uniqueness of weak solutions 
and an a priori estimate,  under the additional (usual) 
condition that $\Divh u\in L^1_tL^\infty$. 

The complete renormalization proof is long and 
technical, with the ``Euclidean'' discussion above shedding 
some light on one part of the argument in a simplified situation.  
A key technical part of the proof concerns the regularization of 
functions via convolution using a mollifier. 
In the Euclidean case mollification commutes with differential operators 
and the regularization error (linked to a commutator between the derivative and the 
convolution operator) converges as the mollification radius tends to zero. 
These properties are not easy to engineer if the function 
in question is defined on a manifold. 
On a Riemannian manifold there exist different approaches for 
smoothing functions, including (\textit{i}) the use of partition of unity 
combined with Euclidean convolution in local charts (see~e.g.~\cite{Dumas:1994aa}), 
(\textit{ii}) the so-called ``Riemannian convolution smoothing'' 
\cite{Greene:1979aa} that is better at preserving geometric properties, 
and (\textit{iii}) the heat semigroup method 
(see~e.g.~\cite{Fang:2011aa}). In \cite{Fang:2011aa}, the authors employ the heat 
semigroup to regularize functions as well as vector fields on manifolds. 
As an application, they extend the DiPerna-Lions theory 
(deterministic equations) to a class of Riemannian manifolds. 
One of the results in \cite{Fang:2011aa} says that the 
DiPerna-Lions commutator converges in $L^1$. 
It is not clear to us how to improve this to $L^2$ convergence, which is 
required by our argument to handle the regularization 
error coming from the second order differential operators (arising when 
passing from Stratonovich to It\^{o} integrals), cf.~the discussion above. 

In the present work we need to regularize functions 
as well as \textit{tensor fields}. 
We will make use of an approach based 
on ``pullback, Euclidean smoothing, and then extension'', in the spirit 
of \cite{Galimberti:2018aa}. When applied to functions 
our approach reduces to (\textit{i}).  
Our regularizing procedure consists of three main steps: 
(\textit{I}) a localization step based on a partition of unity; 
(\textit{II}) transportation of tensor fields from 
$M$ to $\R^d$ and vice versa via pushforwards and pullbacks 
to produce ``intrinsic'' geometric objects; 
(\textit{III}) a convenient choice of atlas that 
allows us to work (locally) with the standard $d$-dimensional 
volume element $dx$ instead of the Riemannian volume element 
$dV_h$, which in local coordinates 
equals $\abs{h}^{\frac12}dx^1\cdots dx^d$ 
(presumably not essential, but it dramatically 
simplifies some computations).  Although our approach shares 
some similarities with the mollifier smoothing 
method found in Nash's celebrated work \cite{Nash:1956aa} 
on embeddings of manifolds into Euclidean spaces, there are essential 
differences. The most important one is that Nash regularizes 
tensor fields on Riemannian manifolds by embedding the manifold 
into an Euclidean space and then convolve the tensor field with a 
mollifier defined on the ambient space. Since the mollifier lives in the larger 
Euclidean space, we cannot easily use it as a test function 
in the weak formulation of \eqref{eq:target} to derive a similar 
SPDE for $\rho_\varepsilon$, the 
regularized version of the weak solution $\rho$.

Roughly speaking, our proof starts off from the following 
It\^{o} form of \eqref{eq:target} (cf.~Section \ref{sec:results} 
for details):
\begin{equation}\label{eq:target-ito}
	d\rho+\Divh (\rho u) \, dt+\sum_{i=1}^N \Divh(\rho a_i) \, dW^i(t)
	=\frac12\sum_{i=1}^N\Lambda_i( \rho).
\end{equation}
Recall that for a vector field $X$ (locally given by $X^j \partial_j$), 
the divergence of $X$ is given by $\Divh X=\partial_j X^j + \Gamma^j_{ij} X^i$, 
where $\Gamma_{ij}^k$ are the Christoffel symbols associated with 
the Levi-Civita connection $\nabla$ of the metric $h$ (the Einstein summation 
convention over repeated indices is used throughout the paper). 
For a smooth function $f:M\to \R$, we have $X(f)=(X,\Grad f)_h$ (which locally 
becomes $X^j\partial_j f$). Moreover, $X\bigl(X(f)\bigr)= (\nabla^2 f)(X,X) 
+ (\nabla_XX)(f)$, where $\nabla^2 f$ is the covariant Hessian of $f$ 
and $\nabla_XX$ is the covariant derivative of $X$ in the direction $X$. In the It\^{o} SPDE 
\eqref{eq:target-ito} we denote by $\Lambda_i(\cdot) 
:=\Divh \left(\Divh(\rho a_i)a_i \right)$ the formal adjoint of 
$a_i\bigl(a_i(\cdot)\bigr)$. Later we prove that the 
second order differential operator $\Lambda_i(f)$ may be recast into the form 
$\Divh^{2}\left(f \hat{a}_i\right) - \Divh \left(f \nabla_{a_i}a_i\right)$, 
where $\Divh^{2}(S)$ is defined by $\Divh\bigl(\Divh(S) \bigr)$ 
for any symmetric $(0,2)$-tensor field $S$. 
Further, $\hat{a}_i$ is the symmetric $(0,2)$-tensor field whose 
components are locally given by $\hat{a}_i^{kl}=a_i^k a_i^l$. 
We refer to an upcoming section for relevant background 
material in differential geometry.

Fixing a smooth partition of the unity 
$\seq{\mathcal{U}_\kappa}_{\kappa\in\mathcal{A}}$ 
subordinate to a conveniently chosen atlas $\mathcal{A}$, cf.~(\textit{III}) above, 
we utilize our regularization device to derive a rather involved equation 
for each piece $\bigl( \rho(t) \mathcal{U}_\kappa\bigr)_\varepsilon$. A 
global SPDE for $\rho_\varepsilon:=\sum_\kappa 
\bigl( \rho(t) \mathcal{U}_\kappa\bigr)_\varepsilon$  is then obtained 
by summing up the local equations. We subsequently 
use the It\^{o} and classical chain rules to 
arrive at an equation for $F(\rho_\varepsilon)$, $F\in C^2$ with 
$F,F',F''$ bounded, which contains numerous remainder terms coming from 
the regularization procedure, some of which can be analyzed in 
terms of first order commutators related to the differential 
operators $\Divh(\cdot u)$, $\Divh \left(\cdot \nabla_{a_i}a_i\right)$ and 
second order commutators related to $\Divh^{2}\left(\cdot \hat{a}_i\right)$. 
In addition, we must exploit specific cancellations coming from 
some quadratic terms linked to the covariation of the martingale part of 
the equation \eqref{eq:target-ito} and the second order operators $\Lambda_i$.
The localization part of the regularization procedure generates a number 
of error terms as well, some of which are easy to control whereas others rely on 
the identification of specific cancellations. At long last, after  
sending the regularization parameter $\varepsilon$ to zero, we 
arrive at the renormalized equation
\begin{align*}
	\partial_t &F(\rho)  + \Divh \bigl(F(\rho) u \bigr) 
	- \sum_{i=1}^N \Divh \bigl(G_F(\rho) \bar{a}_i\bigr)+ G_F(\rho)\Divh u
	\\ & 
	+\sum_{i=1}^N \Divh \bigl(F(\rho) a_i\bigr)\, \dot{W}^i
	+\sum_{i=1}^N G_F(\rho)\Divh a_i \, \dot{W}^i(t)
	\\ & \quad
	= \frac12\sum_{i=1}^N\Lambda_i(F(\rho))
	-\frac12 \sum_{i=1}^N\Lambda_i(1) G_F(\rho) 
	+ \frac12\sum_{i=1}^N F''(\rho)\,\bigl(\rho\Divh a_i\bigr)^2,
\end{align*}
where $G_F(\rho)=\rho F'(\rho)-F(\rho)$, $\bar{a}_i= (\Divh a_i) \, a_i$, 
$\Lambda_i(1) = \Divh^{2} \left(\hat{a}_i\right) -\Divh \left(\nabla_{a_i}a_i \right)$.

The remaining part of this paper is organized as follows: In Section \ref{sec:background} 
we collect the assumptions that are imposed on the ``data" of the problem, and 
present background material from differential geometry and stochastic analysis.
The definitions of solution and the main results are stated in Section \ref{sec:results}. 
Section \ref{sec:informal-proof} is dedicated to an informal outline of the proof 
of the renormalization property, while a rigorous proof is 
developed in Section \ref{sec:proof-main-result}.  Corollaries of the main result 
(uniqueness and a priori estimate) are proved in Section \ref{sec:uniqueness}. 
Finally, in Section \ref{sec:appendix} we bring together a few basic 
results used throughout the paper.


\section{Background material and hypotheses}\label{sec:background}

In an attempt to  make this paper more self-contained and to fix relevant notation, we 
briefly review some basic aspects of differential geometry and stochastic analysis.
Furthermore, we collect the precise assumptions imposed on the 
coefficients $u, a_i$ appearing in the stochastic continuity equation \eqref{eq:target}. 


\subsection{Geometric framework}\label{sec:geometric framework}

We refer to \cite{Aubin,LeeSmooth} for basic definitions and facts 
concerning manifolds. Consider a $d$-dimensional 
smooth Riemannian manifold $M$, which is closed, connected, and 
oriented (for instance, the $d$-dimensional sphere). 
Moreover, $M$ is endowed with a smooth (Riemannian) metric $h$. 
By this we mean that $h$ is a positive-definite 2-covariant tensor field, which thus 
determines for every $x\in M$ an inner product $h_x$ on $T_xM$.  Here, $T_xM$ 
denotes the tangent space at $x$, whereas $TM=\coprod_{x\in M} T_xM$ 
denotes the tangent bundle. For two arbitrary vectors 
$X_1,X_2\in T_xM$, we will henceforth write 
$h_x(X_1,X_2)=:\left(X_1,X_2\right)_{h_x}$ 
or even $\left(X_1,X_2\right)_h$ if the context is clear. 
We set $\abs{X}_h:=\left( X, X\right)_h^{1/2}$. 
Recall that in local coordinates $x=(x^i)$, the 
partial derivatives $\partial_i:=\frac{\partial}{\partial x^i}$ 
form a basis for $T_xM$, while the differential forms $dx^i$ 
determine a basis for the cotangent space $T_x^\ast M$. 
Therefore, in local coordinates, $h$ reads 
$$
h = h_{ij}\, dx^i dx^j,  \quad 
h_{ij}=\left(\partial_i,\partial_j\right)_h. 
$$
We will denote by $(h^{ij})$ the inverse of the matrix $(h_{ij})$.

We denote by $dV_h$ the Riemannian density associated to $h$, which 
in local coordinates takes the form
$$
dV_h = \abs{h}^{1/2}\, dx^1\cdots dx^d, 
$$
where $\abs{h}$ is the determinant of $h$. Integration 
with respect to $dV_h$ is done in the following way: 
if $f\in C^0(M)$ has support contained in the domain of a 
single chart $\Phi:U\subset M\to \Phi(U)\subset \R^d$, then
$$
\int_M f(x) \,dV_h(x)
=\int_{\Phi(U)}\left(\abs{h}^{1/2}f\right) 
\circ \Phi^{-1} \,dx^1\cdots dx^d, 
$$ 
where $(x^i)$ are the coordinates associated 
to $\Phi$. If $\supp f$ is not contained in 
a single chart domain, then the integral is defined as
$$
\int_M f(x)\,dV_h(x)
=\sum_{i\in\mathcal I} \int_M \left(\alpha_i f\right) (x) \,dV_h(x), 
$$
where $(\alpha_i)_{i\in\mathcal{I}}$ is a partition of 
unity subordinate to some atlas $\mathcal{A}$. 
Throughout the paper, we will assume for convenience that 
$$
\mathrm{Vol}(M,h):=\int_MdV_h=1. 
$$
For $p\in [1,\infty]$, we denote by $L^p(M)$ the usual Lebesgue spaces on $(M,h)$. 
Always in local coordinates, the gradient of a function $f:M\to \R$ 
is the vector field given by the following expression
$$
\Grad f := h^{ij} \partial_if\,\partial_j. 
$$

A smooth $k$-dimensional real vector bundle is a pair of smooth manifolds 
$E$ (the total space) and $V$ (the base), together with a 
surjective map $\pi: E \to V$ (the projection), satisfying the following three 
conditions: (i) each set $E_x:=\pi^{-1}(x)$ (called the fiber of $E$ over $x$) is endowed  
with the structure of a real vector space; (ii) for each $x\in V$, there exists a 
neighborhood $U$ of $x$ and a diffeomorphism $\phi:\pi^{-1}(U)\to U\times \R^k$, 
called a local trivialization of $E$, such that 
$\pi_1\circ\phi = \pi$ on $\pi^{-1}(U)$, where $\pi_1$ is the projection onto 
the first factor; (iii) the restriction of $\phi$ to each fiber, 
$\phi:E_x\to\{x\}\times\R^k$, is a linear isomorphism.  

Given a smooth vector bundle $\pi: E \to V$ over a smooth manifold $V$, a 
section of $E$ is a section of the map $\pi$, i.e., a map $\sigma:V\to E$ 
satisfying $\pi\circ \sigma =\text{Id}_V$.

For an arbitrary finite-dimensional real vector space $H$, we use 
$\mathcal{T}^m(H)$, $\mathcal{T}_l(H)$, and $\mathcal{T}^m_l(H)$ to denote 
the spaces of covariant $m$-tensors, contravariant $l$-tensors, and 
mixed tensors of type $(m,l)$ on $H$, respectively. For an 
arbitrary smooth manifold $V$, we define the bundles of covariant 
$m$-tensors, contravariant $l$-tensors, and mixed tensors 
of type $(m,l)$ on $V$ respectively by 
$$
\mathcal{T}^m(V)=\coprod_{x\in V} \mathcal{T}^m(T_xV), 
\quad 
\mathcal{T}_l(V)= \coprod_{x\in V} \mathcal{T}_l(T_xV), 
\quad 
\mathcal{T}^m_l(V)= \coprod_{x\in V} \mathcal{T}^m_l (T_xV).
$$ 
Note the natural identifications 
$\mathcal{T}^1(V)=T^\ast V$ and $\mathcal{T}_1(V)=TV$.

Let $F:V\to \bar V$ be a \textit{diffeomorphism} 
between two smooth manifolds $V$, $\bar V$. The 
symbols $F_\ast$, $F^\ast$ denote the smooth bundle isomorphisms 
$F_\ast:\mathcal{T}^m_l(V)\to \mathcal{T}^m_l\left(\bar V\right)$ 
and $F^\ast:\mathcal{T}^m_l\left(\bar V\right)\to \mathcal{T}^m_l(V)$ satisfying 
$$
F_\ast S\left(X_1,\ldots ,X_m,\omega^1,\ldots ,\omega^l\right)
= S\left(F^{-1}_\ast X_1,\ldots ,F^{-1}_\ast X_m, 
F^\ast\omega^1,\ldots , F^\ast\omega^l\right),
$$
for $S\in\mathcal{T}^m_l(V)$, $X_i \in T \bar V$, 
$\omega^j\in T^\ast \bar V$, and 
$$
F^\ast S\left(X_1,\ldots ,X_m,\omega^1,\ldots ,\omega^l\right)
= S\left(F_\ast X_1,\ldots ,F_\ast X_m, F^{-1\ast}\omega^1,
\ldots,F^{-1\ast}\omega^l\right),
$$
for $S\in\mathcal{T}^m_l\left(\bar V\right)$, $X_i \in TV$, 
$\omega^j\in T^\ast V$ (for further details see \cite[Chapter 11]{LeeSmooth}).

The symbol $\nabla$ refers to the Levi-Civita 
connection of $h$, namely the unique linear connection on $M$ that is 
compatible with $h$ and is symmetric. The Christoffel symbols 
associated to $\nabla$ are given by
$$
\Gamma^k_{ij} = \frac12 h^{kl} \left( \partial_ih_{jl}
+\partial_jh_{il}-\partial_lh_{ij} \right). 
$$
In particular, the covariant derivative of a vector field 
$X=X^\alpha\partial_\alpha$ is the $(1,1)$-tensor field which 
in local coordinates reads
$$
(\nabla X)_j^\alpha  := 
\partial_j X^\alpha + \Gamma^\alpha_{kj}X^k.
$$
The divergence of a vector field $X=X^j\partial_j$ is the function defined by
$$
\Divh X := \partial_jX^j + \Gamma^j_{kj}X^k. 
$$

For any vector field $X$ and $f\in C^1(M)$, we have 
$X(f)= (X,\Grad f)_h$, which locally takes the form $X^j \partial_jf$. 
We recall that for a (smooth) vector field $X$, the following 
integration by parts formula holds:
$$
\int_M X(f) \,dV_h =\int_M \left(\Grad f,X\right)_h\,dV_h 
=-\int_M f \,\Divh X \, dV_h,
$$
recalling that $M$ is closed (all functions are compactly supported).

Given a smooth vector field $X$ on $M$, we consider the  norms
$$
\norm{X}_{\overrightarrow{L^p(M)}}^p
:= \int_M \abs{X}^p_h \,dV_h,
\quad p\in [1,\infty), \quad 
 \norm{X}_{\overrightarrow{L^\infty(M)}} 
 :=\norm{\abs{X}_h}_{L^\infty(M)}.
$$
The closure of the space of the smooth vector fields on $M$ 
with respect to the norm $\norm{\cdot}_{\overrightarrow{L^p(M)}}$ is 
denoted by $\overrightarrow{L^p(M)}$. We define the Sobolev space 
$\overrightarrow{W^{1,p}(M)}$ in a similar fashion. 
Indeed, consider the norm
\begin{align*}
	&\norm{X}_{\overrightarrow{W^{1,p}(M)}}^p
	:= \int_M \abs{X}^p_h + \abs{\nabla X}^p_h \,dV_h, 
	\quad p\in [1,\infty),
	\\ & 
	\norm{X}_{\overrightarrow{W^{1,\infty}(M)}}
	:=\norm{\abs{X}_h + \abs{\nabla X}_h}_{L^\infty(M)},
\end{align*}
where locally $\abs{\nabla X}_h^2= (\nabla X)^i_j \, h_{ik}h^{jm} \, (\nabla X)^k_m$. 
The closure of the space of the smooth vector fields with respect  
to this norm is $\overrightarrow{W^{1,p}(M)}$. 
For more operative definitions, $\overrightarrow{L^p(M)}$ and 
$\overrightarrow{W^{1,p}(M)}$ can be seen as the spaces of vector fields whose 
components in any arbitrary chart belong to the corresponding Euclidean space.

Given a smooth vector field $X$, consider the second 
order differential operator $X(X(\cdot))$. We have

\begin{lemma}[geometric identity]\label{lem:ai-ai}
For any smooth vector field $X$ and $\psi\in C^2(M)$,
$$
X(X(\psi)) = (\nabla^2 \psi)(X,X) + (\nabla_XX)(\psi),
$$
where $\nabla^2 \psi$ denotes the covariant Hessian of 
$\psi$ and $\nabla_XX$ denotes the covariant derivative of $X$ in the direction $X$.
\end{lemma}

\begin{proof}
In any coordinate system, we have
\begin{align*}
	& (\nabla^2 \psi)(X,X) = \partial_{lm} \psi X^l X^m
	-\Gamma_{lm}^j \partial_j \psi X^l X^m,
	\\ & (\nabla_XX)(\psi) = X^m \partial_m X^l \partial_l\psi 
	+ \Gamma_{lm}^j\partial_j\psi X^l X^m.
\end{align*}
On the other hand, $X(X(\psi)) = \partial_{lm}\psi X^lX^m 
+ X^m\partial_m X^l\partial_l\psi$.
\end{proof}

In the following, we will consistently write 
$(\nabla^2 \cdot)(a_i,a_i) + (\nabla_{a_i}a_i)(\cdot)$ instead of $a_i(a_i(\cdot))$, 
thereby highlighting the presence of the Hessian.

Let us introduce the following second order differential 
operators associated to the vector fields $\{a_i\}_{i=1}^N$:
\begin{equation}\label{eq:def-Lambda-i}
	\Lambda_i(\psi):= 
	\Divh(\Divh(\psi a_i)a_i), \;\;\; \psi\in C^2(M), \;\; i=1,\ldots, N.
\end{equation}
We will need to write these operators in a more appropriate form. 
To this end, we will first make a short digression into some 
concepts from differential geometry.

Given a smooth symmetric $(0,2)$-tensor field $S$ on $M$, 
we can compute $\Divh S$, which is the smooth vector 
field whose local expression is given by 
\begin{equation}\label{eq:divh-S}
	\Divh S := \nabla_jS^{ij}\, \partial_i 
	= \left\{\partial_j S^{ij}+ \Gamma^i_{lj}S^{lj} 
	+ \Gamma^j_{lj}S^{il}\right\}\partial_i,
\end{equation}
where, obviously, $S= S^{ij}\,\partial_i\otimes\partial_j$ (since $S$ is 
symmetric, it is irrelevant which index we contract). 
Because $\Divh S$ is a vector field, it can operate on functions by 
differentiation. Moreover, we can compute its divergence. 
Henceforth, we set
\begin{equation}\label{eq:def-divh2}
	\Divh^{2}(S):=\Divh \bigl(\Divh(S)\bigr).
\end{equation}

Given any vector field $X$ on $M$, we can 
canonically construct a symmetric $(0,2)$-tensor field on $M$ in 
the following fashion: we consider the 
endomorphism induced by $X$ on the tangent bundle $TM$,  
$$
Y_p \mapsto \left(X_p,Y_p\right)_h \,X_p, \quad p\in M, 
\quad Y=\text{vector field}. 
$$
This endomorphism can be canonically identified 
with a $(1,1)$-tensor field. Besides, rising an index via the 
metric $h$ produces a symmetric $(0,2)$-tensor 
field $\hat{X}$, whose components are locally given by
$$
\hat{X}^{jk}= X^j X^k.
$$

\begin{remark}\label{rem:def-hat-a}
In what follows, we use the symbols $\hat{a}_1,\ldots,\hat{a}_N$ 
to denote the smooth symmetric $(0,2)$-tensor fields 
obtained by applying the procedure defined above to 
the vector fields $a_1,\ldots,a_N$.
\end{remark}

We may now state
\begin{lemma}[alternative expression for $\Lambda_i$]\label{lem:alt-Lambdai}
For $\psi\in C^2(M)$,  
\begin{equation}\label{eq:alt-Lambdai}
	\Lambda_i(\psi) = 
	\Divh^{2}\bigl(\psi \hat{a}_i\bigr) 
	- \Divh \bigl(\psi \nabla_{a_i}a_i\bigr),
	\quad i=1,\ldots,N.
\end{equation}
\end{lemma}

\begin{proof}
In any coordinates, from the definition of the 
divergence of a vector field,
\begin{align*}
	&\left(\Divh(\psi a_i)a^\beta_i\right)\partial_\beta 
	=\left[\partial_\ell \left(\psi a^\ell_i\right)a^\beta_i 
	+ \Gamma^k_{\ell k}\, a^\ell_ia^\beta_i\psi\right]\partial_\beta 
	\\ & \qquad
	= \left[\partial_\ell \left(\psi \hat{a}^{\ell\beta}_i\right) 
	+ \Gamma^k_{\ell k}\,\hat{a}^{\ell\beta}_i\psi 
	- \psi a^\ell_i\partial_\ell a^\beta_i\right]\partial_\beta 
	\\ & \qquad
	= \left[\partial_\ell \left(\psi \hat{a}^{\ell\beta}_i\right) 
	+ \Gamma^k_{\ell k}\,\hat{a}^{\ell\beta}_i\psi 
	- \psi a^\ell_i\partial_\ell a^\beta_i - \psi\Gamma^\beta_{jk}\hat{a}^{jk}_i 
	+ \psi\Gamma^\beta_{jk}\hat{a}^{jk}_i\right]\partial_\beta.
\end{align*}
Therefore, recalling that locally \eqref{eq:divh-S}
and $\nabla_{a_i}a_i=\left[a^\ell_i \partial_\ell a^\beta_i 
+ \Gamma^\beta_{jk}\hat{a}^{jk}_i\right]\partial_\beta$ hold, we 
obtain the following identity between vector fields:
$$
\Divh (\psi a_i )\, a_i 
= \Divh \bigl(\psi \hat{a}_i\bigr) - \psi \nabla_{a_i}a_i.
$$
We apply $\Div$ to this equation to 
obtain \eqref{eq:alt-Lambdai}.
\end{proof}

\begin{remark}[adjoint of $\Lambda_i$]
The adjoint of $\Lambda_i(\cdot)$ is $a_i(a_i(\cdot))$, 
i.e.~$\forall \psi,\phi\in C^2(M)$,
$$
\int_M \Lambda_i(\psi)\phi\,dV_h 
= \int_M \psi \, a_i(a_i(\phi)) \,dV_h 
=\int_M \psi\,\Bigl((\nabla^2\phi)(a_i,a_i) 
+(\nabla_{a_i}a_i)(\psi)\Bigr)\,dV_h.
$$
\end{remark}

The following lemma turns out to be an extremely useful 
instrument in the proof of Theorem \ref{thm:main-result}. 
It allows us to introduce a special atlas on $M$, in whose 
charts the determinant of the metric $h$ will be constant. 
It turns out that this atlas significantly simplifies several 
terms in some already long computations; in broad strokes, 
the underlying reason is we can work locally with the 
standard $d$-dimensional Lebesgue measure $dz$ 
instead of the Riemannian volume element $dV_h$.

\begin{lemma}[convenient choice of atlas]\label{lem:det-metric-const}
On the manifold $M$ there exists a finite atlas 
$\mathcal{A}=\{\kappa:X_\kappa\subset M \to \tilde{X}_\kappa\subset \R^d\}$ 
such that, for any $\kappa\in\mathcal{A}$, the determinant 
of the metric written in that chart is equal to one: 
$\abs{h_\kappa}\equiv 1$. In particular, we have 
\begin{equation}\label{eq:special-atlas-prop}
	\text{$\Gamma^m_{mj}=0$ on $X_\kappa$, 
	for any $j=1,\cdots,d$.}
\end{equation}
\end{lemma}

\begin{proof}
Fix $x\in M$ and consider a chart $\Phi$ around $x$, 
whose induced coordinates are named $(u^i)$ 
and whose range is the open unit cube in $\R^d$, $(0,1)^d$. 
Then, $(\Phi^{-1})^\ast\, dV_h = f \, du^1\wedge\cdots\wedge du^d$, 
where $f=\abs{h_\Phi}^{1/2}$, where $(\Phi^{-1})^\ast$ is defined in 
Section \ref{sec:geometric framework}, and $\wedge$ 
denotes the wedge product between forms. 
Without loss of generality, we can assume from the 
beginning that $f\in C^\infty([0,1]^d)$. Consider the following map 
from $(0,1)^d$ to $\R^d$:
$$
\Psi: 
\begin{cases}
	z^1= 
	\int_0^{u^1} f(\zeta,u^2,\ldots,u^d)
	\, d\zeta \\
	z^2=u^2 \\
	\vdots \\
	z^d=u^d
\end{cases}.
$$
One can check that $\Psi$ is smooth and invertible onto its image (recall $f>0$). 
Moreover, $\abs{\Psi'}=f(u^1,\ldots,u^d)>0$. By the inverse function 
theorem and the fact that $\Psi$ admits a global inverse, we 
infer that $\Psi$ is a diffeomorphism of $(0,1)^d$ onto its image, and
$$
\bigl((\Psi\circ\Phi)^{-1}\bigr)^\ast\, dV_h 
= \bigl(\Psi^{-1}\bigr)^\ast \bigl(\Phi^{-1}\bigr)^\ast \, dV_h 
=  dz^1\wedge\cdots\wedge dz^d.
$$
We set $\kappa_x:= \Psi\circ\Phi$. We repeat this procedure for 
any $x\in M$, and by compactness of $M$ we end up with a finite 
atlas $\mathcal{A}=\{\kappa:X_\kappa\subset M 
\to \tilde{X}_\kappa\subset \R^d\}$ with the desired property. 
In general, $\Gamma^m_{mj}= \partial_j\log \abs{h_\kappa}^{\frac12}$ 
\cite[page 106]{Aubin}. Hence, \eqref{eq:special-atlas-prop} follows.
\end{proof}

\begin{remark}
A different proof of Lemma \ref{lem:det-metric-const}, 
which requires much more baggage, can be found in \cite{Banyaga}.
\end{remark}

Finally, we discuss the conditions imposed on the vector field $u$. Firstly, 
\begin{equation}\label{eq:u-ass-1}
	u\in L^1\left([0,T]; \overrightarrow{W^{1,2}(M)}\right).
\end{equation}
In particular, we have $u\in L^1\left([0,T]; \overrightarrow{L^2(M)}\right)$, which 
is sufficient to conclude that for $\rho\in L^\infty_tL^2_{\omega,x}$ and 
$\psi\in C^\infty(M)$, $t\mapsto \int_0^t \int_M \rho(s) u(s)(\psi) \,dV_h\,ds$ 
is absolutely continuous, $\P$-a.s., and hence is not contributing to 
cross-variations against $W^i$. These cross-variations appear when passing from 
Stratonovich to It\^o integrals in the SPDE \eqref{eq:target}, 
consult the upcoming Lemma \ref{lem:L2-weak-sol-Ito}. 

For the uniqueness result (cf.~Corollary \ref{cor:uniq-result}), we 
must also assume
\begin{equation}\label{eq:u-ass-2}
	\Divh u\in L^1\left([0,T]; L^\infty(M)\right).
\end{equation}

\begin{remark}
In the following, for a function $f:M\to \R$ and a vector 
field $X$, we will freely jump between the different notations
$$
f(x)X(x),\quad f(x)X, \quad (fX)(x), 
\qquad x\in M,
$$
for the vector field obtained by pointwise scalar 
multiplication of $f$ and $X$.
\end{remark}

\subsection{Stochastic framework}

We refer to \cite{Protter,Revuz:1999aa} for relevant notation, 
concepts, and basic results in stochastic analysis. From beginning to end, we fix a 
complete probability space $(\Omega, \mathcal{F}, \P)$ and a 
complete right-continuous filtration $\seq{\mathcal{F}_t}_{t\in [0,T]}$. 
Without loss of generality, we assume that the $\sigma$-algebra 
$\mathcal{F}$ is countably generated.  Let $W=\seq{W_i}_{i=1}^N$ 
be a finite sequence of independent one-dimensional 
Brownian motions adapted to the filtration
$\seq{\mathcal{F}_t}_{t\in [0,T]}$. We refer to 
$\bigl(\Omega, \mathcal{F}, \seq{\mathcal{F}_t}_{t\in [0,T]},\P, W\bigr)$ 
as a (Brownian) \textit{stochastic basis}. 

Consider two real-valued stochastic processes $Y,\tilde Y$. 
We call $\tilde Y$ a \textit{modification} of $Y$ 
if, for each $t\in [0,T]$, $\P\bigl(\bigl\{\omega\in \Omega:Y(\omega,t)
=\tilde Y(\omega,t)\bigr\}\bigr)=1$.  It is important to pick 
good modifications of stochastic processes. 
Right (or left) continuous modifications are often 
used (they are known to exist 
for rather general processes), since any two such modifications 
of the same process are indistinguishable (with 
probability one they have the same sample paths). 
Besides, they necessarily have left-limits everywhere. 
Right-continuous processes with left-limits 
are referred to as \textit{c{\`a}dl{\`a}g}.

An $\left\{\mathcal{F}_t\right\}_{t\in [0,T]}$-adapted, c{\`a}dl{\`a}g process 
$Y$ is an $\left\{\mathcal{F}_t\right\}_{t\in [0,T]}$-semimartingale 
if there exist processes $F,M$ with $F_0=M_0=0$ such that
$$
Y_t = Y_0 + F_t+ M_t,
$$
where $F$ is a finite variation process and $M$ is a local martingale. 
In this paper we will only be concerned with \textit{continuous} semimartingales.
The quantifier ``local'' refers to the existence of a sequence $\seq{\tau_n}_{n\ge1}$ of 
stopping times increasing to infinity such that the stopped 
processes $\mathbf{1}_{\seq{\tau_n>0}} M_{t\wedge \tau_n}$ are martingales. 

Given two continuous semimartingales $Y$ and $Z$, we can define 
the Fisk-Stratonovich integral of $Y$ with respect to $Z$ by
$$
\int_0^tY(s)\circ dZ(s) = \int_0^tY(s) \, dZ(s) + \frac12 \action{Y}{Z}_t,
$$
where $\int_0^tY(s) dZ(s)$ is the It\^o integral of $Y$ 
with respect to $Z$ and $\action{Y}{Z}$ denotes the quadratic cross-variation process of $Y$  
and $Z$.  Let us recall It\^o's formula for a continuous semimartingale $Y$. 
Let $F\in C^2(\R)$. Then $F(Y)$ is again a continuous semimartingale and 
the following chain rule formula holds:
$$
F(Y(t))-F(Y(0))= \int_{0}^t F'(Y(s))dY(s) 
+ \frac12\int_0^t F''(Y(s)) \, d\action{Y}{Y}_s.
$$

Martingale inequalities are generally important for several reasons. For 
us they will be used to bound It\^{o} stochastic integrals in terms 
of their quadratic variation (which is easy to compute). One of the most important 
martingale inequalities is the Burkholder-Davis-Gundy inequality. 
Let $Y=\seq{Y_t}_{t\in [0,T]}$ be a continuous local martingale with $Y_0=0$. 
Then, for any stopping time $\tau \leq T$,
\begin{equation}\label{eq:BDG}
	\EE \left(\sup_{t\in[0,\tau]}\abs{Y_t}\right)^p \leq 
	C_p\, \EE \sqrt[\leftroot{-2}\uproot{-1}p]{\action{Y}{Y}_\tau},
	\qquad p\in (0,\infty),
\end{equation}
where $C_p$ is a universal constant. We use \eqref{eq:BDG} 
with $p=1$, in which case $C_p=3$. 

Finally, the vector fields driving the noise in \eqref{eq:target} satisfy
\begin{equation}\label{eq:a-ass}
	a_1,\ldots,a_N\in C^\infty(M).
\end{equation}


\section{Weak solutions and main results}\label{sec:results}
Inspired by \cite{Flandoli-Gubinelli-Priola}, we work with the 
following concept of solution for \eqref{eq:target}.

\begin{definition}[weak $L^2$ solution]\label{def:L2-weak-sol}
Given $\rho_0\in L^2(M)$, a weak $L^2$ solution of \eqref{eq:target} 
with initial datum $\rho|_{t=0}=\rho_0$ is a function $\rho\in 
L^\infty([0,T];L^2(\Omega\times M))$ such that for all $\psi\in C^\infty(M)$ the 
stochastic process $(\omega,t)\mapsto\int_M\rho(t)\psi\,dV_h$ has a continuous modification 
which is an $\left\{\mathcal{F}_t\right\}_{t\in [0,T]}$-semimartingale and 
$\P$-a.s., for all $t\in [0,T]$, 
\begin{equation}\label{eq:L2-weak-sol-strat}
	\begin{split}
		\int_M\rho(t)\psi\,dV_h &= \int_M\rho_0\psi\,dV_h
		+\int_0^t\int_M\rho(s)\,u(\psi)\,dV_h\,ds\\
		&\qquad+\sum_{i=1}^N\int_0^t\int_M\rho(s)\,a_i(\psi)\,dV_h\circ dW^i(s).
	\end{split}
\end{equation}
\end{definition}

\begin{remark}
Since each vector field $a_i$ is smooth, cf.~\eqref{eq:a-ass}, 
the corresponding stochastic process 
$(\omega,t)\mapsto\int_M \rho(s) \, a_i(\psi) \, dV_h$ 
has a continuous modification that is an 
$\left\{\mathcal{F}_t\right\}_{t\in [0,T]}$-semimartingale.
\end{remark}

The first result brings \eqref{eq:target} 
into its equivalent It\^o form. The result is 
analogous to Lemma 13 in \cite{Flandoli-Gubinelli-Priola}.

\begin{lemma}[Stratonovich-It\^{o} conversion]\label{lem:L2-weak-sol-Ito}
Let $\rho$ be a weak $L^2$ solution 
of \eqref{eq:target}, according to Definition \ref{def:L2-weak-sol}. 
Then equation \eqref{eq:L2-weak-sol-strat} is equivalent to 
\begin{equation}\label{eq:L2-weak-Ito}
	\begin{split}
		& \int_M \rho(t)\psi\, dV_h = \int_M \rho_0\psi\, dV_h 
		+\int_0^t\int_M\rho(s)\,u(\psi)\,dV_h\,ds
		\\ & \quad 
		+ \sum_{i=1}^N\int_0^t\int_M \rho(s) \, a_i(\psi) \, dV_h\, dW^i(s)
		+ \frac12\sum_{i=1}^N \int_0^t \int_M \rho(s) \,a_i(a_i(\psi)) \, dV_h \,ds.
	\end{split}
\end{equation}
\end{lemma}

\begin{proof}
Let us commence from \eqref{eq:L2-weak-sol-strat}. 
The Stratonovich integrals can be written as
\begin{align*}
	& \sum_{i=1}^N \int_0^t\int_M \rho(s) \, a_i(\psi) \, dV_h\circ dW^i(s)
	\\& \quad 
	=\sum_{i=1}^N\int_0^t\int_M \rho(s)\, a_i(\psi) \, dV_h\, dW^i(s)
	+\frac12 \sum_{i=1}^N 
	\left\langle\int_M \rho(s) \,a_i(\psi) \, dV_h, W^i \right\rangle_t.
\end{align*}
where $\action{\cdot}{\cdot}$ denotes the cross-variation between stochastic 
processes. Using \eqref{eq:L2-weak-sol-strat} with $a_i(\psi)\in C^\infty(M)$ as 
test function, we infer
\begin{align*}
	& \left\langle \int_M \rho \,a_i(\psi) \, dV_h, W^i\right\rangle_t 
	= \sum_{j=1}^N \left\langle \int_0^\cdot \int_M \rho \, a_j(a_i(\psi)) 
	\, dV_h\circ dW^j, W^i\right\rangle_t
	\\ & \qquad 
	= \sum_{j=1}^N \left\langle \int_0^\cdot \int_M \rho \, a_j(a_i(\psi)) 
	\, dV_h\, dW^j, W^i\right\rangle_t
	\\ & \qquad \qquad 
	+ \frac12 \sum_{j=1}^N\left\langle \left\langle  \int_M \rho\,a_j(a_i(\psi)) 
	\, dV_h, W^j   \right\rangle , W^i \right \rangle_t,
\end{align*}
where we have exploited that the time-integral is absolutely continuous 
and thus not contributing to the cross-variation against $W^i$, 
which follows from \eqref{eq:u-ass-1} and the fact that $\rho$ 
belongs $\P$-a.s.~to $L^2([0,T]\times M)$. 

Since $a_j(a_i(\psi))\in C^\infty(M)$, the stochastic process 
$(\omega,t)\mapsto \int_M \rho\, a_j(a_i(\psi)) \, dV_h$ is 
a continuous semimartingale by assumption. It follows from 
\cite[Theorem 2.2.14]{Kunita} that the variation process $\left\langle  \int_M \rho \,a_j(a_i(\psi) )
\, dV_h, W^j  \right\rangle$ is continuous and of bounded variation. Hence, 
$\bigl\langle\left\langle\cdot,\cdot\right\rangle,\cdot\bigr\rangle=0$. Therefore,
$$
\left\langle\int_M \rho \,a_i(\psi) \, dV_h, W^i\right\rangle_t  
= \sum_{j=1}^N \left\langle \int_0^\cdot \int_M \rho \, a_j(a_i(\psi)) 
\, dV_h\, dW^j, W^i \right\rangle_t.
$$
Since $\rho\in L^\infty([0,T];L^2(\Omega\times M))$, we clearly 
have $\int_M \rho \, a_j(a_i(\psi)) \, dV_h\in L^2([0,T])$, $\P$-a.s., and 
so by \cite[Theorem 2.3.2]{Kunita} we obtain
\begin{align*}
	&\left\langle\int_M \rho\, a_i(\psi) \, dV_h, W^i\right\rangle_t 
 	= \sum_{j=1}^N \int_0^t \int_M \rho \,a_j(a_i(\psi)) 
	\, dV_h \,d\langle W^j, W^i\rangle_s
	\\ & \quad 
	= \sum_{j=1}^N \int_0^t \int_M \rho \,a_j(a_i(\psi)) 
	\, dV_h \,\delta^{ji} \, ds
	= \int_0^t \int_M \rho \,a_i(a_i(\psi)) \, dV_h \,ds,
\end{align*}
and the sought equation \eqref{eq:L2-weak-Ito} follows.  
Finally, we can repeat this argument, starting with \eqref{eq:L2-weak-Ito} 
and working our way back to \eqref{eq:L2-weak-sol-strat}. This 
concludes the proof.
\end{proof}

In view of Lemma \ref{lem:L2-weak-sol-Ito}, we have an 
equivalent concept of solution.

\begin{definition}[weak $L^2$ solution, It\^{o} formulation]\label{def:L2-weak-sol-Ito}
A weak $L^2$-solution of \eqref{eq:target} 
with initial datum $\rho|_{t=0}=\rho_0\in L^2(M)$ 
is a function $\rho\in L^\infty([0,T];L^2(\Omega\times M))$ such 
that for any $\psi\in C^\infty(M)$ the stochastic process 
$(\omega,t)\mapsto\int_M\rho(t)\psi\,dV_h$ 
has a continuous modification which is $\left\{\mathcal{F}_t\right\}_{t\in [0,T]}$-adapted 
and satisfies the following equation $\P$-a.s., for all $t\in [0,T]$:
\begin{align*}
	& \int_M \rho(t)\psi\, dV_h = \int_M \rho_0\psi\, dV_h 
	+ \int_0^t\int_M\rho(s)\,u(\psi)\,dV_h\,ds
	\\ &\quad 
	+ \sum_{i=1}^N\int_0^t\int_M \rho(s) \,a_i(\psi) \, dV_h\, dW^i(s)
	+ \frac12\sum_{i=1}^N \int_0^t \int_M \rho(s) \, a_i(a_i(\psi)) \, dV_h \,ds.
\end{align*}
\end{definition}

\begin{definition}[renormalization property]\label{def:main-result}
Let $\rho$ be a weak $L^2$ solution of 
\eqref{eq:target} with initial datum 
$\rho|_{t=0}=\rho_0\in L^2(M)$.
We say that $\rho$ is renormalizable if, for 
any $F\in C^2(\R)$ with $F,F',F''$ bounded, and for any 
$\psi\in C^\infty(M)$, the stochastic process 
$(\omega,t)\mapsto\int_M F(\rho(t))\psi\,dV_h$ 
has a continuous modification that is 
$\left\{\mathcal{F}_t\right\}_{t\in [0,T]}$-adapted and 
satisfies the following SPDE weakly (in $x$) $\P$-a.s.:
\begin{equation}\label{eq:main-result}
	\begin{split}
		& dF(\rho)  + \Divh \bigl(F(\rho) u \bigr)\, dt 
		+ G_F(\rho)\Divh u \, dt
		\\ & \qquad\qquad 
		+\sum_{i=1}^N \Divh\bigl(F(\rho) a_i\bigr)\, dW^i(t)
		+\sum_{i=1}^N G_F(\rho)\Divh a_i \, dW^i(t)
		\\ & \qquad
		= \frac12\sum_{i=1}^N\Lambda_i(F(\rho))\, dt
		-\frac12 \sum_{i=1}^N\Lambda_i(1) G_F(\rho)\,dt 
		\\ & \qquad \qquad 
		+ \frac12\sum_{i=1}^N F''(\rho)\,\bigl(\rho\Divh a_i\bigr)^2\, dt
		+\sum_{i=1}^N\Divh \bigl(G_F(\rho)\bar{a}_i\bigr ) \,dt,
	\end{split}
\end{equation}
where the second order differential operator $\Lambda_i$ 
is defined in \eqref{eq:alt-Lambdai},  
\begin{equation}\label{eq:bar-ai-Lambda-i(1)}
	\bar{a}_i :=\left(\Divh a_i\right)a_i,
	\quad
	\Lambda_i(1) = \Divh^{2} \left(\hat{a}_i\right) 
	-\Divh \left(\nabla_{a_i}a_i \right),
\end{equation}
and 
\begin{equation}\label{eq:GF-def}
	G_F(\xi)=\xi F'(\xi)-F(\xi), 
	\quad \xi\in\R.
\end{equation}

The equation \eqref{eq:main-result} 
is understood in the space-weak sense, that is, 
for all test functions $\psi\in C^\infty(M)$ 
and for all $t\in [0,T]$, $\P$-a.s.,
{\small
\begin{equation}\label{eq:main-result-weak-form}
	\begin{split}
		&\int_M F(\rho(t))\psi \, dV_h 
		= \int_M F(\rho_0)\psi\, dV_h 
		+\int_0^t\int_M F(\rho(s))\,u(\psi)\,dV_h\,ds
		\\ &\quad 
		+\sum_{i=1}^N\int_0^t\int_M 
		F(\rho(s)) \,a_i(\psi) \, dV_h\, dW^i(s)
		+ \frac12 \sum_{i=1}^N \int_0^t \int_M 
		F(\rho(s))\, a_i(a_i(\psi)) \, dV_h \,ds
		\\ & \quad 
		-\int_0^t\int_M G_F(\rho(s))\Divh u \,\psi\, dV_h\, ds
		-\sum_{i=1}^N\int_0^t\int_MG_F(\rho(s))\Divh a_i 
		\,\psi\,dV_h \, dW^i(s) 
		\\ & \quad 
		-\frac12 \sum_{i=1}^N\int_0^t\int_M 
		\Lambda_i(1)\,G_F(\rho(s))\,\psi\,dV_h\,ds
		\\ & \quad 
		+\frac12\sum_{i=1}^N\int_0^t\int_M 
		F''(\rho(s))\bigl(\rho(s)\Divh a_i\bigr)^2\,\psi\,dV_h\,ds
		\\ & \quad
		-\sum_{i=1}^N\int_0^t\int_M G_F(\rho(s)) 
		\bar{a}_i(\psi)\,dV_h \,ds.
	\end{split}
\end{equation}
}
\end{definition}

\begin{remark}\label{rem:modification-of-some-terms}
The quantity $\mathcal{J}:=-\frac12 \sum_i \Lambda_i(1) G_F(\rho)\,dt
+\sum_i \Divh \bigl(G_F(\rho)\bar{a}_i\bigr) \,dt$ 
in \eqref{eq:main-result} takes the equivalent form
$$
\mathcal{J}
=\frac12 \sum_{i=1}^N\Lambda_i(1) G_F(\rho)\,dt
+\sum_{i=1}^N \bar{a}_i\bigl( G_F(\rho)   \bigr)\,dt,
$$
if we apply the product rule to the divergence 
of the scalar $G_F(\rho)$ times the vector field $\bar{a}_i$, 
remembering that $\Lambda_i(1)=\Divh\bar{a}_i$, cf.~\eqref{eq:def-Lambda-i} 
and \eqref{eq:bar-ai-Lambda-i(1)}. We will make use of this expression 
for $\mathcal{J}$ in the upcoming computations.
\end{remark}

We can now state the main result of this paper.

\begin{theorem}[renormalization property]\label{thm:main-result}
Suppose conditions \eqref{eq:u-ass-1} and \eqref{eq:a-ass} hold. 
Consider a weak $L^2$ solution $\rho$ of \eqref{eq:target} 
with initial datum $\rho_0\in L^2(M)$, according 
to Definition \ref{def:L2-weak-sol-Ito}. 
Then $\rho$ is renormalizable in the sense of 
Definition \ref{def:main-result}. 
\end{theorem}

As an application of this result, we obtain a 
uniqueness result for \eqref{eq:target}, if we further assume 
that $\Divh u\in L^1_tL^\infty_x$, cf.~\eqref{eq:u-ass-2}. 
More precisely, we have

\begin{corollary}[uniqueness]\label{cor:uniq-result}
Suppose conditions \eqref{eq:u-ass-1}, \eqref{eq:u-ass-2}, 
and \eqref{eq:a-ass} hold. Then the initial-value 
problem for \eqref{eq:target} possesses 
at most one weak $L^2$ solution $\rho$ in the sense 
of Definition \ref{def:L2-weak-sol-Ito}. 
\end{corollary}

According to Definition \ref{def:L2-weak-sol-Ito}, a weak solution 
$\rho$ belongs to the space $L^\infty_tL^2_{\omega,x}$. Combining 
the proof of Corollary \ref{cor:uniq-result} and a standard 
martingale argument, we can strengthen this 
through ``shifting'' $\esssup_{t}$ inside the expectation operator $\EE[\cdot]$, 
so that  $\rho \in L^2_\omega L^\infty_tL^2_x$ and 
consequently, $\P$-a.s., $\rho \in L^\infty_tL^2_x$.

\begin{corollary}[a priori estimate]\label{cor:apriori-est}
Suppose the assumptions of Corollary \ref{cor:uniq-result} 
are satisfied. Consider a weak $L^2$ solution $\rho$ of \eqref{eq:target} 
with initial datum $\rho_0\in L^2(M)$. Then 
$\rho\in L^2\left(\Omega; L^\infty([0,T];L^2(M))\right)$ and 
\begin{equation}\label{eq:apriori-est}
	\EE \esssup_{t\in [0,T]} \norm{\rho(t)}_{L^2(M)}^2 
	\leq \exp(Ct) \norm{\rho_0}_{L^2(M)}^2,
\end{equation}
where the constant $C$ depends on $\norm{\Divh u}_{L^1_tL^\infty_x}$ 
and $\max_i\norm{a_i}_{C^2}$.
\end{corollary}

\begin{remark}
Throughout the paper, we assume that the vector fields 
driving the noise are smooth, $a_i\in C^\infty$. 
In the Euclidian setting \cite{Punshon-Smith:2017aa}, the renormalization property holds under 
appropriate Sobolev smoothness, say 
$a_i\in W^{1,p}$ with $p\ge 4$. 
A conceivable but quite technical extension of our work
would allow for $a_i\in \overrightarrow{W^{1,p}(M)}$. 
We leave this extension for future work. We refer to \cite{Galimberti:2021aa} for proof of 
the existence of weak solutions. 
Beyond the existence result, in that paper, we identify 
a delicate “regularization by noise” effect 
for carefully chosen noise vector fields 
(these vector fields must be linked to 
the geometry of the underlying domain). 
Consequently, we obtain existence without 
an $L^\infty$ assumption on the divergence 
of the velocity $u$.  
\end{remark}


\section{Informal proof of Theorem \ref{thm:main-result}}
\label{sec:informal-proof}

In this section, we give a motivational account of the proof of our main result, 
assuming simply that all considered functions have the necessary 
smoothness for the operations we perform on them. 
To this end, consider a solution $\rho$ of \eqref{eq:target}, 
which in It\^o form reads \eqref{eq:target-ito}, 
cf.~Lemma \ref{lem:L2-weak-sol-Ito}.
An application of It\^o's formula with $F\in C^2(\R)$ gives 
\begin{equation}\label{eq:ito-strong-sol}
	\begin{split}
		& dF(\rho)+ F'(\rho)\Divh(\rho u)\, dt 
		+ \sum_{i=1}^N F'(\rho)\Divh (\rho a_i) \, dW^i(t)
		\\ &\qquad \quad
		= \frac12 \sum_{i=1}^N F'(\rho) \Lambda_i(\rho)\,dt 
		+\frac12\sum_{i=1}^N F''(\rho) 
		\bigl(\Divh(\rho a_i)\bigr)^2dt.
	\end{split}
\end{equation}

By the product and chain rules,
$$
F'(\rho)\Divh(\rho V)
=\Divh \left(F(\rho) V\right) +G_F(\rho)\Divh V, 
\qquad V=u,a_i.
$$

To take care of the term $F'(\rho) \Lambda_i(\rho)$, we need

\begin{lemma}\label{lem:div(F(f)S)}
Let $S$ be a smooth symmetric $(0,2)$-tensor field on $M$, 
$f\in C^1(M)$, and $F\in C^1(\R)$. Then, as vector fields,
$$
\Divh \left(F(f)S\right) = F(f)\Divh(S) + F'(f)\,S(df,\cdot).
$$

\end{lemma}

\begin{proof}
In any coordinates, by the product and chain rules,
\begin{align*}
	& \Divh(F(f)S) 
	\overset{\eqref{eq:divh-S}}{=} 
	\Bigl[\partial_j \left(F(f)S^{ij}\right) 
	+ \Gamma^i_{lj} F(f)S^{lj} 
	+ \Gamma^j_{lj}F(f)S^{il} \Bigr]\partial_i
	\\ & \quad = F(f)\Divh(S) + S^{ij}F'(f)\partial_j f\partial_i 
	=F(f)\Divh(S) + F'(f)\,S(df,\cdot).
\end{align*}
\end{proof}

In view of Lemmas \ref{lem:alt-Lambdai} and \ref{lem:div(F(f)S)},
\begin{equation}\label{eq:expand-Lambda_i(F)}
	\begin{split}
		\Lambda_i(F(\rho)) & = 
		\Divh^{2} \left(F(\rho) \hat{a}_i\right) 
		- \Divh \left(F(\rho) \nabla_{a_i}a_i\right)
		\\ & = \Divh \Bigl( F(\rho)\Divh(\hat{a}_i) 
		+ F'(\rho)\,\hat{a}_i(d\rho,\cdot)\Bigr)
		- \Divh \left(F(\rho) \nabla_{a_i}a_i\right)
		\\ & = F'(\rho)\Divh(\hat{a}_i)(\rho) 
		+ F(\rho) \Divh^{2}(\hat{a}_i) 
		+ F''(\rho)\, \hat{a}_i(d\rho,\cdot)(\rho)
		\\ & \qquad 
		+ F'(\rho)\Divh \left(\hat{a}_i(d\rho,\cdot)\right)
		- F'(\rho)\nabla_{a_i}a_i(\rho) 
		- F(\rho)\Divh(\nabla_{a_i}a_i).
	\end{split}
\end{equation}
where, cf.~Remark \ref{rem:def-hat-a}, $\hat{a}_i(d\rho,\cdot)(\rho)
=\hat{a}_i(d\rho,d\rho)= \bigl(a_i(\rho)\bigr)^2$. 
On the other hand,
\begin{equation}\label{eq:F'Lambda_i}
	\begin{split}
		F'(\rho)\Lambda_i(\rho) 
		& = F'(\rho)\Divh^{2}\left(\rho \hat{a}_i\right) 
		- F'(\rho)\Divh \left(\rho\nabla_{a_i}a_i\right)
		\\ & 
		= F'(\rho)\Divh\Bigl( \rho \Divh(\hat{a}_i) 
		+\hat{a}_i(d\rho,\cdot)\Bigr) 
		- F'(\rho)\Divh \left(\rho\nabla_{a_i}a_i\right)
		\\ & 
		= F'(\rho)\Divh(\hat{a}_i)(\rho) 
		+ F'(\rho)\rho\,\Divh^{2}(\hat{a}_i) 
		+ F'(\rho)\Divh \left(\hat{a}_i(d\rho,\cdot)\right)
		\\ &\qquad 
		- F'(\rho)\rho\,\Divh(\nabla_{a_i}a_i)
		-F'(\rho)\nabla_{a_i}a_i(\rho).
	\end{split}
\end{equation}
Therefore, subtracting \eqref{eq:expand-Lambda_i(F)} from \eqref{eq:F'Lambda_i},
\begin{align*}
	& F'(\rho)\Lambda_i(\rho)-\Lambda_i(F(\rho)) 
	\\ & \quad
	=G_F(\rho) \Divh^{2}(\hat{a}_i) 
	- G_F(\rho) \Divh(\nabla_{a_i}a_i) 
	- F''(\rho)\bigl(a_i(\rho)\bigr)^2
	\\ & \quad
	=G_F(\rho)\Lambda_i(1) - F''(\rho)\bigl(a_i(\rho)\bigr)^2,
\end{align*}
where $G_F$ is defined in \eqref{eq:GF-def} and $\Lambda_i(1)$ 
is defined in \eqref{eq:bar-ai-Lambda-i(1)}.

In view of the above computations, we can write 
\eqref{eq:ito-strong-sol} as
\begin{align*}
	& dF(\rho)+\Divh \bigl(F(\rho) u\bigr)\, dt
	+G_F(\rho)\Divh u \, dt
	\\ & \qquad\qquad
	+\sum_{i=1}^N \Divh \bigl(F(\rho) a_i\bigr)\, dW^i(t)
	+\sum_{i=1}^N G_F(\rho)\Divh a_i \, dW^i(t)
	\\ & \qquad 
	= \frac12\sum_{i=1}^N\Lambda_i(F(\rho))\, dt
	+\frac12 \sum_{i=1}^N\Lambda_i(1) G_F(\rho)\,dt
	\\ & \qquad \qquad
	+\underbrace{\frac12\sum_{i=1}^N F''(\rho)\bigl(\Divh(\rho a_i) \bigr)^2\, dt
	-\frac12 \sum_{i=1}^N F''(\rho)\bigl(a_i(\rho)\bigr)^2\,dt}_{=:\mathcal{Q}},
\end{align*}
where we need to take a closer look 
at the (potentially) problematic term $\mathcal{Q}$, 
which contains the difference between some 
quadratic terms linked to the covariation of the martingale part of 
the equation \eqref{eq:target-ito} 
and the second order operators $\Lambda_i$.

We apply the product rule to write $\Divh (\rho a_i )
=\rho \Divh a_i+a_i(\rho)$, and then expand the 
square $\bigl(\Divh(\rho a_i)\bigr)^2$, yielding
\begin{align*}
	&\bigl(\Divh(\rho \,a_i) \bigr)^2 - \bigl(a_i(\rho)\bigr)^2 
	\\ & \quad 
	= \bigl(\rho \Divh a_i\bigr)^2 + 2\rho \,a_i(\rho)\Divh a_i 
	= \bigl(\rho \Divh a_i\bigr)^2 + 2\rho \,\bar{a}_i(\rho),
\end{align*}
where $\bar{a}$ is defined in $\eqref{eq:bar-ai-Lambda-i(1)}$. 
As a result of this and the chain/product rules for $\bar{a}_i$,
\begin{align*}
	& F''(\rho) 
	\Bigl (\bigl(\Divh(\rho \,a_i)\bigr)^2 
	- \bigl(a_i(\rho)\bigr)^2\Bigr)
	\\ & \quad 
	= F''(\rho)\bigl(\rho \Divh a_i \bigr)^2 
	+ 2\rho \,\bar{a}_i\bigl(F'(\rho)\bigr)
	\\ & \quad 
	= F''(\rho)\bigl(\rho \Divh a_i\bigr)^2 
	+ 2\,\bar{a}_i\bigl(\rho\,F'(\rho)\bigr) 
	- 2 F'(\rho)\bar{a}_i(\rho) 
	\\ & \quad 
	= F''(\rho)\bigl(\rho \Divh a_i\bigr)^2 
	+ 2 \bar{a}_i\bigl(G_F(\rho)\bigr),
\end{align*}
and so $\mathcal{Q}$ becomes
$$
\mathcal{Q}=\frac12\sum_{i=1}^N F''(\rho)
\bigl(\rho \Divh a_i\bigr)^2\, dt
+\sum_{i=1}^N\overline{a}_i\bigl(G_F(\rho)\bigr)\,dt.
$$
We note here that the problematic term $\bigl(a(\rho)\bigr)^2$ 
has cancelled out in the final expression for $\mathcal{Q}$. 
This is similar to what happens in the Euclidean setting 
\cite{Punshon-Smith:2017aa}. On a curved manifold we must 
in addition exploit ``cancellations" to control some error terms coming 
from the localization part of our regularization procedure, i.e., 
terms related to the geometry of the underlying domain (cf.~Section 
\ref{sec:proof-main-result} for details).

This concludes the informal argument for \eqref{eq:main-result}.


\section{Rigorous proof of Theorem \ref{thm:main-result}}
\label{sec:proof-main-result}
The aim of this section is to develop a rigorous proof 
of Theorem \ref{thm:main-result}. The proof will involve 
a series of long computations, which will be 
scattered over seven subsections. We begin with the
procedure for regularizing tensor fields on a manifold, along with 
several commutator estimates for controlling the regularization error.


\subsection{Pullback and extension of tensor fields}

We first recall and extend some concepts from Section \ref{sec:geometric framework}. 
Let $V$ be an arbitrary (boundaryless) smooth manifold of dimension $d$. Consider 
an arbitrary chart $\kappa: X_\kappa\to \tilde{X}_\kappa$ for $V$, where $X_\kappa$ 
and $\tilde{X}_\kappa$ are open subsets of $V$ and $\R^d$ respectively. 
Let $\text{RS}(\tml)$ denote the space of $m$ covariant 
and $l$ contravariant tensor fields on $\tilde{X}_\kappa\subset\R^d$, 
and define similarly $\text{RS}(\mathcal{T}^m_l(X_\kappa))$ 
and $\text{RS}(\mathcal{T}^m_l(V))$. 
Observe that we do not impose any assumptions on the 
regularity of the coefficients of the tensor fields; RS is 
an acronym for Rough Sections. Let $\text{SS}(\tml)$ be the subspace of smooth 
sections, and define similarly $\text{SS}(\mathcal{T}^m_l(X_\kappa))$, $
\text{SS}(\mathcal{T}^m_l(V))$.

We are going to define a procedure for pulling an element of 
$\text{RS}(\tml)$ back to $V$. Indeed, given $\sigma\in\text{RS}(\tml)$, 
we may transport it on $X_\kappa\subset V$ via the \textit{diffeomorphism} 
$\kappa$, and we call the result $\kappa^\ast\sigma$, which will belong to 
$\text{RS}(\mathcal{T}^m_l(X_\kappa))$. We refer to 
\cite[Exercise 11-6]{LeeSmooth} and Section \ref{sec:geometric framework} 
for details (the fact that $\kappa$ is a diffeomorphism is crucial). 
Moreover, we may trivially extend it to 
the whole of $V$, by simply declaring that it is the 
null $(m,l)$-tensor field outside $X_\kappa$. Let us name the resulting 
object $(\kappa^\ast\sigma)^{\mathrm{ext}}$. 
Let us give a name to the entire procedure:
\begin{equation}\label{eq:extension-op}
	\mathcal{L}_\kappa:\text{RS}(\tml)\ni\sigma\mapsto 
	(\kappa^\ast\sigma)^{\mathrm{ext}}
	\in \text{RS}(\mathcal{T}^m_l(V)),
\end{equation}
where $\mathcal{L}_\kappa$ will be referred to 
as a ``pullback-extension" operator.

Assuming in addition that
$$
\supp \sigma \subset \tilde{X}_\kappa, \qquad
\sigma \in \text{SS}(\tml),
$$
it is trivial to see that $\mathcal{L}_\kappa\sigma\in \text{SS}(\mathcal{T}^m_l(V))$ 
and $\supp\mathcal{L}_\kappa\sigma \subset X_\kappa$.

In the following, starting from objects defined on open 
Euclidean subsets, we are going to use this procedure 
repeatedly to build global objects on $M$.


\subsection{Regularization \& commutator estimates}\label{sec:local-smoothing}

From now on we are going to use the 
atlas $\mathcal{A}$ provided by Lemma \ref{lem:det-metric-const}. 
For fixed $\kappa\in\mathcal{A}$, the induced coordinates 
will be typically denoted by $z$ or $\bar{z}$. 
We need a smooth partition of the unity 
$\seq{\mathcal{U}_\kappa}_{\kappa\in\mathcal{A}}$ 
subordinate to $\mathcal{A}$, i.e.,
\begin{enumerate}

\item $\mathcal{U}_\kappa\geq 0$, 
$\sum_{\kappa\in\mathcal{A}}\mathcal{U}_\kappa=1$,

\item $\mathcal{U}_\kappa\in C^\infty(M)$, and

\item $\mbox{supp }\uk\subset X_\kappa$ (and compact). 

\end{enumerate}
 
Let $\rho$ be a weak $L^2$-solution of \eqref{eq:target} with 
initial datum $\rho_0\in L^2(M)$. 
In what follows, we introduce a series of local objects that 
appear later in a localized version of \eqref{eq:target}, and 
establish their main properties. For $\kappa\in\mathcal{A}$, 
fix a standard mollifier $\phi$ on $\R^d$ with 
support in $\overline{B_1(0)}$, and define 
the rescaled mollifier
\begin{equation}\label{eq:mollifier}
	\phi_\varepsilon(z) := 
	\varepsilon^{-d}\phi\left(\frac{z}{\varepsilon}\right), 
	\quad z\in\R^d, 
\end{equation}
whose support is contained in $\overline{B_\varepsilon(0)}$, 
$z=\kappa(x)$.

\subsubsection{Localization \& smoothing of $\rho$, $\rho_0$}
Set
\begin{equation}\label{eq:def-rhok-rho0k}
	\rho_\kappa(\omega,t,z):=\uk(z)\,\rho(\omega,t,z),
	\quad 
	\rho_{0,\kappa}(\omega,z):=\uk(z)\,\rho_0(\omega,z), 
\end{equation}
for $\omega\in\Omega$, $t\in[0,T]$, 
$z\in\tilde{X}_\kappa \subset\R^d$.

\begin{remark}
As is customary in differential geometry, we 
will use the convention of not explicitly 
writing the chart, in order to alleviate the notation. 
For example, if $f:M\to\R$, then we write $f(z)$ 
instead of $f(\kappa^{-1}(z))$.
\end{remark}

We observe that for fixed 
$\omega\in\Omega,t\in[0,T]$, 
\begin{align*}
	&\supp \rho_\kappa(\omega,t,\cdot)
	\subset \kappa\left(\supp\uk\right)\subset\subset 
	\tilde{X}_\kappa 
	\subset \R^d,
	\\ &
	\supp \rho_{0,\kappa}(\omega,\cdot)\subset 
	\kappa\left(\supp\uk\right)\subset\subset \tilde{X}_\kappa 
	\subset \R^d, 
\end{align*}
and thus $\rho_\kappa(\omega,t,\cdot)$ and 
$\rho_{0,\kappa}(\omega,\cdot)$ may be viewed 
as global functions on $\R^d$.  Next we define spatial 
regularizations of $\rho_{\kappa}$ and $\rho_{0,\kappa}$. 
For $\omega\in\Omega$, $t\in[0,T]$, $z\in\R^d$,
\begin{align*}
	& (\rho_{\kappa})_\varepsilon (\omega,t,z) 
	\\ & \qquad :=\int_{\R^d}\rho_{\kappa}(\omega,t,\bar{z})
	\, \phi_\varepsilon\left(z-\bar{z}\right)
	\,d\bar{z} =\int_{\R^d}\uk(\bar{z})\rho(\omega,t,\bar{z})\,
	\phi_\varepsilon\left(z-\bar{z}\right)\,d\bar{z},
	\\ & (\rho_{0,\kappa})_\varepsilon(\omega,z) 
	\\ & \qquad := \int_{\R^d}\rho_{0,\kappa}(\omega,\bar{z})\,
	\phi_\varepsilon\left(z-\bar{z}\right)\,d\bar{z} 
	= \int_{\R^d}\uk({\bar{z}})\rho_{0}(\omega,\bar{z})
	\,\phi_\varepsilon\left(z-\bar{z}\right)\,d\bar{z}.
\end{align*}
For later use, set 
\begin{equation}\label{eq:def-ek-enull}
	\varepsilon_\kappa :=
	\mbox{dist}\left(\kappa\left(\supp\uk\right),\partial\tilde{X}_\kappa \right)>0,
	\qquad 
	\varepsilon_0 :=\frac14\min_{\kappa}\{\varepsilon_\kappa\}>0. 
\end{equation}
The main properties of $(\rho_{\kappa})_\varepsilon$ and 
$(\rho_{0,\kappa})_\varepsilon$ are collected in 

\begin{lemma}\label{lem:prop-rho-kappa-eps}
Fix $\kappa\in\mathcal{A}$, cf.~Lemma \ref{lem:det-metric-const}. Then
\begin{enumerate}

\item $(\rho_{\kappa})_\varepsilon(\omega,t,\cdot)\in 
C^\infty(\R^d)$, for all $(\omega,t)\in\Omega \times [0,T]$.

\item For $\varepsilon < \varepsilon_\kappa$ and for any 
$\omega\in\Omega,t\in[0,T]$,
$$
\supp (\rho_{\kappa})_\varepsilon(\omega,t,\cdot) \subset 
\kappa\left(\supp\uk\right) +\overline{B_\varepsilon(0)}
\subset\subset \tilde{X}_\kappa. 
$$ 
This implies in particular that for any $(\omega,t)\in\Omega \times [0,T]$, 
the function $(\rho_{\kappa})_\varepsilon(\omega,t,\cdot)$ can 
be seen as an element of $C^\infty(M)$, provided 
that we set it equal to zero outside of $X_\kappa$.

\item For any $p\in [1,2]$ and $(\omega,t)\in\Omega \times [0,T]$,
$$
(\rho_{\kappa})_\varepsilon(\omega,t,\cdot) 
\overset{\varepsilon\downarrow 0}{\longrightarrow} 
\uk(\cdot) \rho(\omega,t,\cdot) 
\quad \mbox{in $L^p(M)$}.  
$$
Therefore, for any $q\in [1,\infty)$, 
$(\rho_{\kappa})_\varepsilon 
\overset{\varepsilon\downarrow 0}{\longrightarrow} \uk \rho$ 
in $L^q\left([0,T]; L^2(\Omega\times M)\right)$.
\end{enumerate}

The listed properties hold true for 
$(\rho_{0,\kappa})_\varepsilon$ as well.
\end{lemma}

\begin{proof}
Claims (1) and (2) follow from standard properties of convolution. 
To prove claim (3), recall Lemma \ref{lem:det-metric-const}. Indeed, on 
$X_\kappa$ we use the coordinates given by $\kappa$, for which 
$\abs{h_\kappa(z)}^{1/2}=1$, to compute as follows:
\begin{align*}
	& \int_{M}\abs{(\rho_{\kappa})_\varepsilon(t,x) 
	-\uk(x) \rho(t,x)}^p\,dV_h(x)
	\\ & \quad =\int_{X_\kappa}
	\abs{(\rho_{\kappa})_\varepsilon(t,x) 
	-\uk(x) \rho(t,x)}^p\,dV_h(x)
	\\ & \quad =\int_{\tilde{X}_\kappa}
	\abs{(\rho_{\kappa})_\varepsilon(t,z) 
	-\uk(z) \rho(t,z)}^p\,dz
	\\ & \quad = \int_{\tilde{X}_\kappa}
	\abs{(\rho_{\kappa})_\varepsilon(t,z) 
	-\rho_\kappa(t,z)}^p\,dz.
\end{align*}
The last integral converges to zero as 
$\varepsilon$ goes to zero by standard properties of mollifiers, 
since $\rho_\kappa(\omega,t,\cdot)$ is in $L^2(\R^d)$ 
and has compact support. This follows easily from our 
assumption $\rho\in L^\infty([0,T];L^2(\Omega\times M))$). 

Claims (1), (2), and (3) can be proved in a similar way 
for $(\rho_{0,\kappa})_\varepsilon$.
\end{proof}

We define the global counterparts of $(\rho_{\kappa})_\varepsilon$ 
and $(\rho_{0,\kappa})_\varepsilon$ as follows:
$$
\rho_\varepsilon(\omega,t,x):=
\sum_\kappa (\rho_\kappa)_\varepsilon(\omega,t,x), 
\quad 
\rho_{0,\varepsilon}(x):=
\sum_\kappa (\rho_{0,\kappa})_\varepsilon(x),
$$
for $\omega\in\Omega$, $t\in [0,T]$, $x\in M$, 
and $\varepsilon<\varepsilon_0$.

\subsubsection{Localization \& smoothing of $\rho a_i$}

Consider for $\omega\in\Omega$, $t\in[0,T]$, and $z\in\tilde{X}_\kappa$, 
the object $\seq{\rho_\kappa(\omega,t,z)a^l_i(z)}_l$, which 
belongs to $\text{RS}\left(\mathcal{T}^0_1(\tilde{X}_\kappa)\right)$ and 
is compactly supported in $\kappa\left(\supp\uk\right) 
\subset \tilde{X}_\kappa$, uniformly in $\omega,t$. 

\begin{remark}
By writing $a_i(z)$ we mean the vector field 
evaluated at the point $z$, not differentiation. 
It is a minor abuse of notation that should not 
provoke too much confusion. We will use this convention in the following 
also for other objects.
\end{remark}

We regularize $\seq{\rho_\kappa(\omega,t,z)a^l_i(z)}_l$ 
componentwise via the mollifier $\phi_\varepsilon$ (as above). 
The result is an object in $\text{SS}(\mathcal{T}^0_1(\tilde{X}_\kappa))$ 
that is compactly supported, uniformly in $\varepsilon$. 
We denote this regularized object by
$$
(\rho_\kappa(t)a_i)_\varepsilon
=(\rho_\kappa(t)a_i)_\varepsilon(x)
=(\rho_\kappa(\omega,t)a_i)_\varepsilon(x).
$$ 
We apply the pullback-extension operator 
$\mathcal{L}_\kappa$ defined in \eqref{eq:extension-op}, yielding
$$
\mathcal{L}_\kappa \left(\rho_\kappa(t)a_i\right)_\varepsilon
\in \text{SS}\left(\mathcal{T}^0_1(M)\right).
$$
Finally, we define the companion vector 
field $(\rho_\kappa)_\varepsilon(t)a_i\in \text{SS}\left(\mathcal{T}^0_1(M)\right)$. 

We need the following version of a well-known result found in \cite{DL89}. 

\begin{lemma}[DiPerna-Lions commutator; ``$\Divh(\rho a_i)$"]\label{lem:com-term-1}
Fix $\kappa\in\mathcal{A}$, cf.~Lemma \ref{lem:det-metric-const}, 
and define for $(\omega,t,x)\in\Omega\times[0,T]\times M$ 
the smooth (in $x$) functions
\begin{equation}\label{eq:def-r-kappa}
	r_{\kappa,\varepsilon,i}(\omega,t,x)
	:=\Divh\mathcal{L}_\kappa 
	\bigl(\rho_\kappa(t) a_i\bigr)_\varepsilon(x)
	-\Divh\bigl( (\rho_\kappa)_\varepsilon(t) a_i\bigr)(x),
\end{equation}
for $i=1,\ldots,N$. Then $r_{\kappa,\varepsilon,i}
\overset{\varepsilon\downarrow0}{\longrightarrow} 0$ 
in $L^q([0,T];L^2(\Omega\times M))$, for any $q\in [1,\infty)$.

Furthermore, for $x\in X_\kappa$ (the only relevant case), in the 
coordinates induced by $\kappa$, we have the representations
\begin{align*}
	a_i\left(r_{\kappa,\varepsilon,i}\right) 
	& = a^m_i\partial_{ml} 
	\left(\rho_\kappa a^l_i\right)_\varepsilon 
	-a^m_i a^l_i \partial_{ml}(\rho_\kappa)_\varepsilon
	-a^m_i\partial_m a^l_i\, 
	\partial_{l}(\rho_\kappa)_\varepsilon
	\\ & \qquad
	- a^m_i\partial_la^l_i\,\partial_{m}(\rho_\kappa)_\varepsilon 
	- (\rho_\kappa)_\varepsilon\,a^m_i\partial_{ml}a^l_i,
\end{align*}
and 
$$
\Divh \mathcal{L}_\kappa \left(\rho_\kappa(t)a_i\right)_\varepsilon
=\partial_l\left(\rho_\kappa(t)a^l_i\right)_\varepsilon.
$$
\end{lemma}

\begin{proof}
Let $x\in X_\kappa$. Then, in the coordinates 
induced by $\kappa$, we have by definition that
$\bigl((\rho_\kappa a_i)_\varepsilon\bigr)^l 
= \left(\rho_\kappa a_i^l\right)_\varepsilon$ ($l=1,\ldots,d$) 
and $\Gamma^a_{ab}=0$; therefore, $\Divh$ 
coincides with the Euclidean divergence and thus 
$$
\Divh \mathcal{L}_\kappa \left(\rho_\kappa(s)a_i\right)_\varepsilon
=\partial_l\left(\rho_\kappa(s)a^l_i\right)_\varepsilon.
$$
Moreover,
$$
r_{\kappa,\varepsilon,i} 
= \partial_l\left(\rho_\kappa a_i^l\right)_\varepsilon 
- \partial_l (\rho_\kappa)_\varepsilon a_i^l 
- (\rho_\kappa)_\varepsilon \partial_l a_i^l.
$$
Differentiating this expression according to $a_i$ leads 
to the claimed representation 
for $a_i\left(r_{\kappa,\varepsilon,i}\right)$. 
The convergence claim is also clear. Indeed, with 
``$\rho a_i \in L^\infty_tL^2_x$", repeated applications of 
Lemma \ref{lemma:gen-diperna-lions} 
lead to $r_{\kappa,\varepsilon,i}\to 0$ in 
$L^q([0,T]; L^2(\Omega\times \tilde{X}_\kappa))$, 
for any $q\in [1,\infty)$. Because the support of 
$r_{\kappa,\varepsilon,i}$ is compactly contained in $X_\kappa$ 
and $\abs{h_\kappa}^\frac{1}{2}=1$ therein, we 
immediately deduce the result.
\end{proof}

For $(\omega,t,x)\in\Omega\times [0,T]\times M$, 
we define the function
\begin{equation}\label{eq:commutator-r}
	r_{\varepsilon,i}(\omega,t,x)
	:=\sum_{\kappa\in\mathcal{A}}r_{\kappa,\varepsilon,i}(\omega,t,x)
\end{equation}
and the vector field
$$
(\rho(t)a_i)_\varepsilon(x):=
\sum_\kappa\mathcal{L}_\kappa (\rho_\kappa(t)a_i)_\varepsilon(x),
$$
which both are smooth in $x$, $i=1,\ldots,N$. 
Clearly, as $\varepsilon\to 0$, the global remainder 
function $r_{\varepsilon,i}$ converges to zero 
in $L^q([0,T];L^2(\Omega\times M))$, for any $q\in [1,\infty)$.

\subsubsection{Localization \& smoothing of $\rho\nabla_{a_i}a_i$.}

Consider for $\omega\in\Omega$, $t\in[0,T]$, 
and $z\in\tilde{X}_\kappa$, the vector field
$$
\seq{(\rho_\kappa(\omega,t,z) (\nabla_{a_i}a_i)^l(z)}_l,
$$
which is an object in $\text{RS}\left(\mathcal{T}^0_1(\tilde{X}_\kappa)\right)$ 
that is compactly supported in $\kappa\left(\supp\uk\right)\subset \tilde{X}_\kappa$, 
uniformly in $\omega,t$. Observe that 
$$
\rho_\kappa (\nabla_{a_i}a_i)^l
= \rho_\kappa a^m_i\partial_m a^l_i 
+ \rho_\kappa\Gamma^l_{mb} a^m_i a^b_i, 
\qquad l=1,\ldots ,d.
$$
We regularize $\rho_\kappa\nabla_{a_i}a_i$ componentwise 
using the mollifier $\phi_\varepsilon$, denoting the result by 
$\left(\rho_\kappa(t)\nabla_{a_i}a_i\right)_\varepsilon$. 
By definition, $\bigl(\left(\rho_\kappa(t)
\nabla_{a_i}a_i\right)_\varepsilon\bigr)^l
=\left(\left(\rho_\kappa(t)\nabla_{a_i}a_i\right)^l\right)_\varepsilon$. 
We apply the pullback-extension operator $\mathcal{L}_\kappa$, arriving at 
the compactly supported vector field
$$
\mathcal{L}_\kappa\left(\rho_\kappa(s)\nabla_{a_i}a_i\right)_\varepsilon
\in \text{SS}\left(\mathcal{T}^0_1(M)\right).
$$
Also in this case, we define the companion vector field 
$(\rho_\kappa)_\varepsilon(t)\nabla_{a_i}a_i
\in \text{SS}\left(\mathcal{T}^0_1(M)\right)$.
 
\begin{lemma}[DiPerna-Lions commutator; ``$\Divh(\rho \nabla_{a_i}a_i)$"]\label{lem:com-term-2}
Fix $\kappa\in\mathcal{A}$, cf.~Lemma \ref{lem:det-metric-const}, 
and define for $(\omega,t,x)\in\Omega\times [0,T]\times M$ the smooth (in $x$) functions 
$$
\tilde{r}_{\kappa,\varepsilon,i}(\omega,t,x)
:=\Divh\mathcal{L}_\kappa
\bigl(\rho_\kappa(t) \nabla_{a_i}a_i\bigr)_\varepsilon(x)
-\Divh\bigl((\rho_\kappa)_\varepsilon(t) \nabla_{a_i}a_i\bigr)(x),
$$
for $i=1,\ldots ,N$. Then $\tilde{r}_{\kappa,\varepsilon,i}
\overset{\varepsilon\downarrow0}{\longrightarrow} 0$ 
in $L^q([0,T];L^2(\Omega\times M))$, for any $q\in [1,\infty)$. 

Furthermore, for $x\in X_\kappa$ (the only relevant case), in the 
coordinates induced by $\kappa$, we have the representation
$$
\Divh\mathcal{L}_\kappa 
\left(\rho_\kappa(t)\nabla_{a_i}a_i\right)_\varepsilon
=\partial_l\left(a^m_i\partial_m a^l_i\rho_\kappa(t)\right)_\varepsilon(z)
+\partial_l\left(a^m_i a^r_i\Gamma_{mr}^l
\rho_\kappa(t)\right)_\varepsilon(z).
$$
\end{lemma}

\begin{proof}
The proof is identical to the one of Lemma \ref{lem:com-term-1}, since 
the vector fields $\nabla_{a_1}a_1,\ldots,\nabla_{a_N}a_N$ 
are smooth. 
\end{proof}

We also introduce the global function
\begin{equation}\label{eq:commutator-tilde}
	\tilde{r}_{\varepsilon,i}(\omega,t,x)
	:=\sum_{\kappa\in\mathcal{A}}
	\tilde{r}_{\kappa,\varepsilon,i}(\omega,t,x)
\end{equation}
and the global vector field
$$
\left(\rho(t)\nabla_{a_i}a_i\right)_\varepsilon(x)
:=\sum_{\kappa\in\mathcal{A}}\mathcal{L}_\kappa 
\left(\rho_\kappa(t)\nabla_{a_i}a_i\right)_\varepsilon(x),
$$
which both are smooth in $x$ and defined for 
$(\omega,t,x)\in \Omega\times[0,T]\times M$, $i=1,\ldots,N$. 
Clearly, as $\varepsilon\to 0$, we have 
$\tilde{r}_{\varepsilon,i}\to 0$ in 
$L^q([0,T];L^2(\Omega\times M))$, for any $q\in [1,\infty)$.

\subsubsection{Localization \& smoothing of $\rho\hat{a}_i$.}
Recalling Remark \ref{rem:def-hat-a} ($\hat{a}_i^{ml}=a_i^ma_i^l$), 
let us consider $\seq{\rho_\kappa(\omega,t,z)\hat{a}_i^{ml}(z)}_{m,l}$, 
for $\omega\in\Omega$, $t\in [0,T]$, and $z\in\tilde{X}_\kappa$, 
which defines a symmetric object in 
$\text{RS}\left(\mathcal{T}^0_2(\tilde{X}_\kappa)\right)$ 
that is compactly supported in $\kappa\left(\supp\uk\right)
\subset \tilde{X}_\kappa$, uniformly in $\omega,t$.  
We regularize this object componentwise using the mollifier 
$\phi_\varepsilon$, thereby obtaining a symmetric element in
$\text{SS}\left(\mathcal{T}^0_2(\tilde{X}_\kappa)\right)$, whose support 
is contained in $\tilde{X}_\kappa$, uniformly in $\varepsilon$. 
We denote this smooth $(0,2)$-tensor field by 
$\left(\rho_\kappa(t) \hat{a}_i\right)_\varepsilon$; clearly, 
by definition, $\bigl(\left(\rho_\kappa(t) \hat{a}_i\right)_\varepsilon\bigr)^{ml}
=\left(\rho_\kappa(s) \hat{a}_i^{ml}\right)_\varepsilon$. 
Applying the pullback-extension operator 
$\mathcal{L}_\kappa$, we obtain
$$
\mathcal{L}_\kappa 
\left(\left(\rho_\kappa(t) \hat{a}_i\right)_\varepsilon\right)
\in  \text{SS}\left(\mathcal{T}^0_2(M)\right),
$$
with support in $X_\kappa$, uniformly in $\varepsilon$, 
and symmetric. We also need the globally defined object
$$
\left(\rho(t) \hat{a}_i\right)_\varepsilon(x)
:=\sum_\kappa\mathcal{L}_\kappa
\left(\rho_\kappa(t) \hat{a}_i\right)_\varepsilon(x),
$$
for $(\omega,t,x)\in\Omega\times [0,T]\times M$ and, 
cf.~\eqref{eq:def-ek-enull}, $\varepsilon<\varepsilon_0$.

Let us compute $\Divh^{2}\mathcal{L}_\kappa
\left(\left(\rho_\kappa(t) \hat{a}_i\right)_\varepsilon\right)$ 
in the local coordinates given by Lemma \ref{lem:det-metric-const}. 
The only relevant case is $x\in X_\kappa$, where we use 
the coordinates induced by $\kappa$. Recall that in these coordinates we 
have $\Gamma_{j\alpha}^\alpha=0$ for all $j$. Hence,
\begin{equation}\label{eq:div-square-our-coord}
	\begin{split}
		& \Divh^{2} \mathcal{L}_\kappa
		\left(\left(\rho_\kappa(t) \hat{a}_i\right)_\varepsilon\right)
		\\ & \qquad 
		=\partial_{ml} \left( \left(\rho_\kappa(t) 
		\hat{a}_i\right)_\varepsilon\right)^{ml}(z) 
		+ \partial_l \left[\Gamma_{mj}^l \left(\left(\rho_\kappa(s)
		\hat{a}_i\right)_\varepsilon\right)^{mj}\right](z)
		\\ &\qquad 
		=\partial_{ml}\left(\rho_\kappa(t) 
		\hat{a}_i^{ml}\right)_\varepsilon 
		+ \partial_l \left[\Gamma_{mj}^l
		\left(\rho_\kappa(t)\hat{a}^{mj}_i\right)_\varepsilon
		\right](z).
	\end{split}
\end{equation}

To be able to control the regularization error linked 
to $\Divh^{2}(\rho \hat{a}_i)$, we need first to consider 
some additional terms appearing in the definition of $\Divh^{2}$ 
that is related to the Christoffel symbols $\Gamma$ 
of the Levi Civita connection. To this end, consider 
\begin{equation}\label{eq:def-Vk}
	V_{\kappa,i}:=\seq{\rho_\kappa(\omega,t,z)
	\Gamma^l_{mj}(z)\hat{a}^{mj}_i(z)}_l,
\end{equation} 
which belongs to $\text{RS}\left(\mathcal{T}^0_1(\tilde{X}_\kappa)\right)$ and is 
compactly supported in $\kappa\left(\supp\uk\right)\subset \tilde{X}_\kappa$, 
uniformly in $\omega,t$. The regularized counterpart of $V_{\kappa,i}$ 
is denoted by $V_{\kappa,i,\varepsilon}$. Clearly, $\left(V_{\kappa,i,\varepsilon}\right)^l
=\left(\rho_\kappa(t)\Gamma^l_{mj}\hat{a}^{mj}_i\right)_\varepsilon$. 
Applying the pullback-extension operator $\mathcal{L}_\kappa$ yields
$$
\mathcal{L}_\kappa V_{\kappa,i,\varepsilon}
\in \text{SS}\left(\mathcal{T}^0_1(M)\right).
$$
We multiply the components of 
$\left(\rho_\kappa(t) \hat{a}_i\right)_\varepsilon$ 
by the Christoffel symbols $\Gamma$ (written in 
the coordinates induced by $\kappa$) 
and then add them together. The result is 
\begin{equation}\label{eq:def-barVk}
	\bar{V}_{\kappa,i,\varepsilon} := 
	\seq{\Gamma_{mj}^l\left(\rho_\kappa(t)
	\hat{a}^{mj}_i\right)_\varepsilon}_l, 
\end{equation}
an object in $\text{SS}(\mathcal{T}^0_1(\tilde{X}_\kappa))$ that is 
compactly supported, uniformly in $\varepsilon$. 
Pushing forward $\bar{V}_{\kappa,i,\varepsilon}$ to $M$ 
via $\mathcal{L}_\kappa$, we obtain
$$
\mathcal{L}_\kappa \bar{V}_{\kappa,i,\varepsilon}
\in \text{SS}\left(\mathcal{T}^0_1(M)\right).
$$
For $x\in X_\kappa$, in the coordinates 
given by $\kappa$ ($z=\kappa(x)$), we have
\begin{equation}\label{eq:div-V-barV-our-coord}
	\begin{split}
		&\Divh \mathcal{L}_\kappa V_{\kappa,i,\varepsilon} 
		- \Divh \mathcal{L}_\kappa \bar{V}_{\kappa,i,\varepsilon}
		\\ & \qquad
		= \partial_l\left[ \left(\rho_\kappa(t)\Gamma^l_{mj}
		\hat{a}^{mj}_i\right)_\varepsilon(z) 
		- \Gamma_{mj}^l \left(\rho_\kappa(t)
		\hat{a}^{mj}_i\right)_\varepsilon(z)\right].
	\end{split}
\end{equation}
We also need the globally defined objects
$$
V_{i,\varepsilon}(\omega, t, x):=
\sum_{\kappa\in\mathcal{A}}
\mathcal{L}_\kappa V_{\kappa,i,\varepsilon}(\omega, t, x),
\quad 
\bar{V}_{i,\varepsilon}(\omega, t, x):=
\sum_{\kappa\in\mathcal{A}}\mathcal{L}_\kappa 
\bar{V}_{\kappa,i,\varepsilon}(\omega, t, x),
$$
for $(\omega,t,x)\in\Omega\times[0,T]\times M$ 
and $\varepsilon<\varepsilon_0$.

Finally, consider the smooth (in $x$) and 
compactly supported (in $\tilde{X}_\kappa$) vector 
field $\seq{(\rho_\kappa)_\varepsilon(t) \Gamma^l_{mj}\hat{a}^{mj}_i}_l$, 
denoted by $(\rho_\kappa)_\varepsilon \gamma_i$. 
Applying $\mathcal{L}_\kappa$, we obtain
$$
\mathcal{L}_\kappa \left[(\rho_\kappa)_\varepsilon \gamma_i\right]
\in \text{SS}\left(\mathcal{T}^0_1(M)\right),
$$
which is compactly supported in $X_\kappa$, 
uniformly in $\varepsilon$.

\begin{lemma}[DiPerna-Lions commutator; ``$\Divh (\rho \Gamma \hat{a})$'']\label{lem:com-term-3}
Fix $\kappa\in\mathcal{A}$, cf.~Lemma \ref{lem:det-metric-const}, 
and define for $(\omega,t,x)\in\Omega\times [0,T]\times M$ the 
smooth (in $x$) functions 
\begin{equation}\label{eq:def-rbar-kappa}
	\bar{r}_{\kappa,\varepsilon,i}(\omega,t,x)
	:=\Divh\mathcal{L}_\kappa V_{\kappa,i,\varepsilon}(x)
	-\Divh\mathcal{L}_\kappa 
	\left[(\rho_\kappa(t))_\varepsilon \gamma_i\right](x)
\end{equation}
for $i=1,\ldots,N$. Then, $\bar{r}_{\kappa,\varepsilon,i}
\overset{\varepsilon\downarrow0}{\longrightarrow} 0$ 
in $L^q([0,T];L^2(\Omega\times M))$, for any $q\in [1,\infty)$. 
Moreover, in $X_\kappa$ (the only relevant case) in the 
coordinates induced by $\kappa$ ($z=\kappa(x)$), 
$\bar{r}_{\kappa,\varepsilon,i}$ takes the form
$$
\bar{r}_{\kappa,\varepsilon,i}(\omega,t,z)
=\partial_l \left(\rho_\kappa(t) \Gamma^l_{mj}\hat{a}^{mj}_i\right)_\varepsilon(z)
- \partial_l\left((\rho_\kappa(t))_\varepsilon 
\Gamma^l_{mj}\hat{a}^{mj}_i\right)(z).
$$
\end{lemma}

\begin{proof}
See the proofs of Lemmas \ref{lem:com-term-1} and \ref{lem:com-term-2}. 
\end{proof}

As before, we introduce the global function
\begin{equation}\label{eq:commutator-bar}
	\bar{r}_{\varepsilon,i}(\omega,t,x):
	=\sum_{\kappa\in\mathcal{A}}
	\bar{r}_{\kappa,\varepsilon,i}(\omega,t,x),
\end{equation}
with $(\omega,t,x)\in\Omega\times [0,T]\times M$ 
and, cf.~\eqref{eq:def-ek-enull}, $\varepsilon<\varepsilon_0$, $i=1,\ldots,N$. 
Obviously, as $\varepsilon\to 0$, we have 
$\bar{r}_{\varepsilon,i}\to 0$ in 
$L^q([0,T];L^2(\Omega\times M))$, for any $q\in [1,\infty)$.

Now we are now going to analyze the key terms ($i=1,\ldots,N$)
\begin{equation}\label{eq:divh2-tmp0}
	\begin{split}
		R_{\varepsilon,i}(\omega,t,x) & :=
		\Divh^{2}\left(\rho(t) \hat{a}_i\right)_\varepsilon(x) 
		-\Divh^{2}\left(\rho_\varepsilon(t) \hat{a}_i\right)(x)
		\\ & \qquad 
		+\Divh V_{i,\varepsilon}(t,x) - \Divh \bar{V}_{i,\varepsilon}(t,x),
	\end{split}
\end{equation}
where $\rho_\varepsilon(t) \hat{a}_i 
:= \sum\limits_{\kappa\in\mathcal{A}}(\rho_\kappa)_\varepsilon(t) 
\hat{a}_i \in\text{SS}\left(\mathcal{T}^0_2(M)\right)$. By definition,
\begin{align*}
	R_{\varepsilon,i}(\omega,t,x)
	& =\sum_{\kappa\in \mathcal{A}}
	\Big\{ \Divh^{2}\mathcal{L}_\kappa 
	\left(\rho_\kappa(t)\hat{a}_i\right)_\varepsilon(x) 
	- \Divh^{2}\left((\rho_\kappa)_\varepsilon(t) \hat{a}_i\right)(x)  
	\\ & \qquad \qquad\qquad
	+ \Divh \mathcal{L}_\kappa V_{\kappa,i,\varepsilon}(t,x) 
	- \Divh\mathcal{L}_\kappa \bar{V}_{\kappa,i,\varepsilon}(t,x)\Big\}
	\\ & =:\sum_{\kappa\in \mathcal{A}} R_{\kappa,\varepsilon,i}(\omega,t,x).
\end{align*}
Fix $\kappa\in\mathcal{A}$, cf.~Lemma \ref{lem:det-metric-const}. 
The quantity $R_{\kappa,\varepsilon,i}(t,\cdot)$ 
is supported in $X_\kappa\subset M$, and there we are going to use 
the coordinates induced by $\kappa$, $z=\kappa(x)$. 
From the definition of $\Divh^{2}$, cf.~\eqref{eq:def-divh2}, and 
by means of formulas \eqref{eq:div-square-our-coord} 
and \eqref{eq:div-V-barV-our-coord}, we deduce
\begin{align*}
	R_{\kappa,\varepsilon,i}(t,z)
	&=\partial_{ml}\left(\rho_\kappa(t) \hat{a}_i^{ml}\right)_\varepsilon(z) 
	+ \partial_l \left[\Gamma_{mj}^l \left(\rho_\kappa(t)\hat{a}^{mj}_i\right)_\varepsilon \right](z)
	\\ & \qquad 
	-\partial_{ml}\left((\rho_\kappa)_\varepsilon(t) \hat{a}_i^{ml}\right)(z) 
	- \partial_l\left[\Gamma_{mj}^l
	(\rho_\kappa)_\varepsilon(t)\hat{a}^{mj}_i\right](z)
	\\ & \qquad\qquad
	+\partial_l \left[
	\left(\rho_\kappa(t)\Gamma^l_{mj}\hat{a}^{mj}_i\right)_\varepsilon
	\right](z) 
	-\partial_l\left[\Gamma_{mj}^l
	\left(\rho_\kappa(t)\hat{a}^{mj}_i\right)_\varepsilon
	\right](z)
	\\ & =\partial_{ml}\left(\rho_\kappa(t) \hat{a}_i^{ml}\right)_\varepsilon(z) 
	-\partial_{ml}\left((\rho_\kappa)_\varepsilon(t) \hat{a}_i^{ml}\right)(z) 
	+ \bar{r}_{\kappa,\varepsilon,i}(t,z),
\end{align*}
where the reminder $\bar{r}_{\kappa,\varepsilon,i}$ is defined 
in \eqref{eq:def-rbar-kappa}

\begin{remark}
Be mindful of the fact that we have computed $R_{\varepsilon,i}$ in the 
(convenient) coordinates provided by Lemma \ref{lem:det-metric-const}. 
With a different choice of coordinates, we would need   
to handle some additional terms involving the Christoffel symbols $\Gamma_{ab}^b$, 
further complicating the analysis.
\end{remark}

By expanding $-\partial_{ml}\left((\rho_\kappa)_\varepsilon(t) \hat{a}_i^{ml}\right)(z)$, 
\begin{align*}
\partial_{ml} \left((\rho_\kappa)_\varepsilon \hat{a}_i^{ml}\right)
& = \partial_{ml} (\rho_\kappa)_\varepsilon \hat{a}_i^{ml} + 
2\partial_{l} (\rho_\kappa)_\varepsilon \partial_m a_i^{m}a^l_i \\
&\qquad+ 
2\partial_{l} (\rho_\kappa)_\varepsilon \partial_m a_i^{l}a^m_i +
(\rho_\kappa)_\varepsilon \partial_{ml}\left( a_i^{m}a^l_i \right)\,,  
\end{align*}
and making use of Lemma \ref{lem:com-term-1}, we arrive at
\begin{align*}
	R_{\kappa,\varepsilon,i}(t,z)
	& = 2\mathcal{C}_\varepsilon\left[\rho_\kappa(t),a_i\right](z) 
	+ 2a_i(r_{\kappa,\varepsilon, i}(t,z))
	+ 2 (\rho_\kappa)_\varepsilon(t) a^m_i\partial_{ml}a_i^l(z) 
	\\ & \qquad \qquad
	- (\rho_\kappa)_\varepsilon(t)\partial_{ml}\left(a_i^ma_i^l\right)(z) 
	+ \bar{r}_{\kappa,\varepsilon,i}(t,z),
\end{align*}
where
\begin{equation}\label{eq:def-Ceps}
	\mathcal{C}_{\varepsilon}\left[\rho_\kappa(t),a_i\right]
	:= \frac12\partial_{ml}\left(\rho_\kappa\hat{a}^{ml}_i\right)_\varepsilon 
	-a^m_i\partial_{ml}\left(\rho_\kappa a^l_i\right)_\varepsilon
	+ \frac12\partial_{ml}(\rho_\kappa)_\varepsilon\hat{a}^{ml}_i. 
\end{equation}
We recognize $\mathcal{C}_{\varepsilon}\left[\rho_\kappa(t),a_i\right]$ 
as a ``second order'' commutator, first identified 
in \cite{Punshon-Smith:2018aa} (cf.~the appendix herein 
for more details). Using again the product rule, 
\begin{equation}\label{eq:C2-eps-pops-up}
	\begin{split}
		& R_{\kappa,\varepsilon,i}(t,z)
		\\ & \qquad = 
		2\left\{\mathcal{C}_\varepsilon\left[\rho_\kappa(t),a_i\right](z) 
		-\frac12 (\rho_\kappa)_\varepsilon\left(\partial_m a^m_i\right)^2
		-\frac12(\rho_\kappa)_\varepsilon 
		\partial_l a^m_i\partial_m a^l_i\right\}
		\\ &\qquad\qquad\qquad
		+ 2a_i\left(r_{\kappa,\varepsilon, i}(t,z)\right) 
		+ \bar{r}_{\kappa,\varepsilon,i}(t,z),
	\end{split}
\end{equation}
where $r_{\kappa,\varepsilon,i}$ is defined in \eqref{eq:def-r-kappa}. 
For convenience, let us set $g_{\kappa,i}:=\partial_l a^m_i \partial_m a^l_i$. 
In our coordinates (cf.~Lemma \ref{lem:det-metric-const}), 
$\Divh a_i=\partial_m a^m_i$. By compactness of the supports 
of the involved functions, the quantities $(\rho_\kappa)_\varepsilon\partial_l a^m_i\partial_ma^l_i$ and 
$\mathcal{C}_{\varepsilon}\left[\rho_\kappa(t),a_i\right]$ 
can be thought of as globally defined (on $M$). In other words, \eqref{eq:C2-eps-pops-up} 
may be seen as a global identity on $M$, that is, for $x\in M$,
\begin{equation}\label{eq:divh2-tmp1}
	 R_{\kappa,\varepsilon,i}(t,x)
	= 2 G_{\kappa,\varepsilon,i}(t,x)	
	+ 2a_i \left(r_{\kappa,\varepsilon, i}(t,x)\right) 
	+ \bar{r}_{\kappa,\varepsilon,i}(t,x),
\end{equation}
where
\begin{equation}\label{eq:def-Gkei}
	G_{\kappa,\varepsilon,i}(\omega,t,x)
	:=\mathcal{C}_\varepsilon\left[\rho_\kappa(t),a_i\right](x) 
	-\frac12 (\rho_\kappa)_\varepsilon \left(\Divh a_i(x)\right)^2 
	-\frac12 (\rho_\kappa)_\varepsilon g_{\kappa,i}(x).
\end{equation}
If $x\notin X_\kappa$ for some $\kappa$, then \eqref{eq:divh2-tmp1} 
reduces to the trivial statement ``$0=0$". 
Referring to the appendix, a simple application of 
Lemma \ref{lem:Punshon-Smith-gen} shows that 
$G_{\kappa,\varepsilon,i}\overset{\varepsilon\downarrow0}{\longrightarrow} 0$
in $L^q\left([0,T]; L^2(\Omega\times M)\right)$, for any $q\in [1,\infty)$.

Let us summarize our findings.
\begin{lemma}[Second order commutator; ``$\Divh^2 (\rho \hat{a})$"]
\label{lem:Reps(2)-F}
For $(\omega,t,x)\in \Omega\times [0,T]\times M$ 
and, cf.~\eqref{eq:def-ek-enull}, $\varepsilon<\varepsilon_0$, the remainder 
$R_{\varepsilon,i}$ defined in \eqref{eq:divh2-tmp0} 
takes the form
\begin{equation}\label{eq:Reps(2)-F}
	R_{\varepsilon,i}(\omega,t,x)
	= 2a_i\left(r_{\varepsilon, i}(\omega,t,x)\right) 
	+ \bar{r}_{\varepsilon,i}(\omega,t,x)
	+2G_{\varepsilon,i}(\omega,t,x),
\end{equation}
where $\bar{r}_{\varepsilon,i}$ is defined in \eqref{eq:def-rbar-kappa} 
and $G_{\varepsilon,i}:=\sum\limits_{\kappa\in\mathcal{A}} 
G_{\kappa,\varepsilon,i}$ with $G_{\kappa,\varepsilon,i}$ 
defined in \eqref{eq:def-Gkei}. 

Furthermore,
$$
\bar{r}_{\varepsilon,i}, \, G_{\varepsilon,i}
\overset{\varepsilon\downarrow0}{\longrightarrow} 0 
\quad \text{in $L^q\left([0,T]; L^2(\Omega\times M)\right)$, 
for any $q\in [1,\infty)$}.
$$
\end{lemma}

\begin{remark}
Regarding the error term $G_{\varepsilon,i}$, in general we 
cannot ``sum away'' $\kappa$ due to the nonlinearity of the domain 
$M$, which is manifested in the local nature of the 
commutator $\mathcal{C}_\varepsilon\left[\rho_\kappa(t),a_i\right]$, 
cf.~\eqref{eq:def-Ceps}, and its dependence on different mollifiers! 
\end{remark}

\begin{remark}
In view of \eqref{eq:Reps(2)-F}, we do not expect the 
remainder term
$$
R_{\varepsilon,i}=
\Divh^{2} \Bigl( \left(\rho(t) \hat{a}_i\right)_\varepsilon
-\rho_\varepsilon(t) \hat{a}_i\Bigr) 
+\Divh \Bigl(V_{i,\varepsilon}(t) 
-\bar{V}_{i,\varepsilon}(t)\Bigr)
$$
to converge to zero as $\varepsilon\to 0$, although 
$\bar{r}_{\varepsilon,i}$ and $G_{\varepsilon,i}$ do!  
Indeed, there is no reason to expect 
$a_i\left(r_{\varepsilon, i}\right)$ to 
have a limit as $\varepsilon\to 0$. 
By good fortune, it turns out that this 
quantity is going to cancel out with a term 
that appears when applying the It\^o formula, 
see the upcoming equation \eqref{eq:the-remainder-2}. 
This cancellation is the reason why the renormalization property holds for 
weak $L^2$ solutions without having to assume some kind of 
``parabolic" regularity (cf.~the discussion in Section \ref{sec:intro}).
\end{remark}

\subsubsection{Localization \& smoothing of the vector field $u$}

For $\omega\in\Omega$, $t\in[0,T]$, 
and $z\in\tilde{X}_\kappa$, consider the object 
$\seq{\rho_\kappa(\omega,t,z) u^l(t,z)}_l \in 
\text{RS}\left(\mathcal{T}^0_1(\tilde{X}_\kappa)\right)$, 
that is compactly supported in $\kappa\left(\supp\uk\right)\subset \tilde{X}_\kappa$, 
uniformly in $\omega\in\Omega,s\in[0,T]$.  
As before, we regularize $\{\rho_\kappa(\omega,t,z)u^l(t,z)\}_l$ 
componentwise via the mollifier $\phi_\varepsilon$. 
The result is an object in $\text{SS}(\mathcal{T}^0_1(\tilde{X}_\kappa))$ 
that is compactly supported, uniformly in $\varepsilon$. 
We denote the regularized object by 
$\left(\rho_\kappa(t)u(t)\right)_\varepsilon$. 
Applying the pullback-extension operator $\mathcal{L}_\kappa$, 
$$
\mathcal{L}_\kappa 
\left(\rho_\kappa(t)u(t)\right)_\varepsilon
\in \text{SS}\left(\mathcal{T}^0_1(M)\right).
$$
Finally, we define the companion vector 
field $(\rho_\kappa)_\varepsilon(t)u(t)$, which for 
a.e.~$t$ belongs to $W^{1,2}\left(\mathcal{T}^0_1(M)\right)$ since 
by assumption $u\in L^1([0,T]; W^{1,2}\left(\mathcal{T}^0_1(M)\right))$.

\begin{lemma}[DiPerna-Lions commutator; ``$\Divh(\rho u)$"]\label{lem:com-u}
Fix $\kappa\in\mathcal{A}$, cf.~Lemma \ref{lem:det-metric-const}, 
and define for $(\omega,t,x)\in\Omega\times [0,T]\times M$ the 
smooth (in $x$) function 
$$
r_{\kappa,\varepsilon,u}(\omega,t,x)
:=\Divh\mathcal{L}_\kappa \bigl(\rho_\kappa(t) u(t)\bigr)_\varepsilon(x)
-\Divh \bigl((\rho_\kappa)_\varepsilon(t) u(t)\bigr)(x).
$$
Then $r_{\kappa,\varepsilon,u}
\overset{\varepsilon\downarrow0}{\longrightarrow} 0$ in 
$L^1\left([0,T];L^2\left(\Omega; L^1(M)\right)\right)$. 

Besides, in $X_\kappa$ (the only relevant case) in the 
coordinates induced by $\kappa$,
$$
\Divh \mathcal{L}_\kappa 
\bigl(\rho_\kappa(t)u(t)\bigr)_\varepsilon
=\partial_l \bigl(\rho_\kappa(t)u^l(t)\bigr)_\varepsilon.
$$
\end{lemma}

\begin{proof}
At this point, only the convergence claim needs some explanation. 
Since $u$ is deterministic, cf.~\eqref{eq:u-ass-1},
$$
u\in L^1\left([0,T]; L^\infty\left(\Omega; 
W^{1,2}\left(\mathcal{T}^0_1(M)\right)\right)\right).
$$
Moreover,  $\rho\in L^\infty([0,T];L^2(\Omega\times M))$. Fixing 
$t$, we apply Lemma \ref{lemma:gen-diperna-lions} with $Z=\Omega$, 
$G=\tilde{X}_\kappa$, $p_1=p_2=q_1=2$, $q_2=\infty$ to obtain 
$r_{\kappa,\varepsilon,u}(t) \overset{\varepsilon\downarrow0}{\longrightarrow} 0$ 
in $L^2\left(\Omega;L^1(M)\right)$. Utilizing the dominated convergence 
theorem in $t$ and the bounds provided by Lemma \ref{lemma:gen-diperna-lions}, we
conclude $r_{\kappa,\varepsilon,u} \overset{\varepsilon\downarrow0}{\longrightarrow} 0$ 
in $L^1\left([0,T];L^2\left(\Omega; L^1(M)\right)\right)$.
\end{proof}

For $(\omega,t,x)\in \Omega \times [0,T]\times M$, we define 
the global remainder function
\begin{equation}\label{eq:commutator-u}
	r_{\varepsilon,u}(\omega,t,x)
	:=\sum_{\kappa\in\mathcal{A}} r_{\kappa,\varepsilon,u}(\omega,t,x),
\end{equation}
which belongs to $L^1\left([0,T];L^2\left(\Omega; L^1(M)\right)\right)$ 
and is smooth in $x$. Likewise, we define the global smooth vector field
$$
\bigl(\rho(t)u(t) \bigr)_\varepsilon(x)
:=\sum_\kappa\mathcal{L}_\kappa 
\bigl(\rho_\kappa(t)u(t)\bigr)_\varepsilon(x).
$$
Clearly, $r_{\varepsilon,u}\overset{\varepsilon\downarrow0}{\longrightarrow} 0$ 
in $L^1\left([0,T];L^2\left(\Omega; L^1(M)\right)\right)$.

\subsubsection{Localization \& smoothing of partition of unity terms}
For reasons that will become apparent later, we need to apply 
the machinery developed so far to some additional terms that 
are related to the partition of unity $\{\uk\}_\kappa$ and 
its derivatives. These terms are linked to the nonlinear
geometry of the manifold $M$.

For $\omega\in\Omega$, $t\in[0,T]$, 
and $z\in\tilde{X}_\kappa$, we introduce the functions
\begin{equation}\label{eq:def-Aik-1-2-3}
	\begin{split}
		& A_{\kappa,i}^1(\omega,t,z):=
		\rho(\omega,t,z)a_i(\uk)(z),
		\\ 
		& A_{\kappa,i}^2(\omega,t,z):= 
		\rho(\omega,t,z) \nabla^2\uk(a_i,a_i)(z), 
		\\  
		& A_{\kappa,i}^3(\omega,t,z):= 
		\rho(\omega,t,z) (\nabla_{a_i}a_i)(\uk)(z),
	\end{split}
\end{equation}
cautioning the reader that the superscripts do not mean exponentiation.
Observe that these functions have their supports contained 
in $\supp\uk$, uniformly in $\omega,t$. Besides, recalling the 
local expressions for $\nabla^2\uk$ and $(\nabla_{a_i}a_i)$,  
\begin{align*}
	&A_{\kappa,i}^2(t,z) 
	=\rho(t,z)\left[\partial_{ml}\uk(z)a_i^m(z)a_i^l(z) 
	-\Gamma^j_{ml}(z)\partial_j\uk(z) a_i^m(z)a_i^l(z)\right],
	\\ 
	& A_{\kappa,i}^3(t,z)
	= \rho(t,z)\left[\partial_l\uk(z)\partial_m a_i^l(z)a_i^m(z)
	+\Gamma^j_{ml}(z)\partial_j\uk(z) a_i^m(z)a_i^l(z)\right].
\end{align*}
We regularize these functions using the mollifier $\phi_\varepsilon$ 
and then apply the pullback-extension operator \eqref{eq:extension-op}. 
The next lemma is analogous to Lemma \ref{lem:prop-rho-kappa-eps}, 
with the proof being evident at this stage.

\begin{lemma}\label{lem:propA-kappa-eps-j-1-2-3}
Fix $\kappa\in\mathcal{A}$, cf.~Lemma \ref{lem:det-metric-const}. 
Then
\begin{enumerate}

\item $\left(A_{\kappa,i}^1\right)_\varepsilon(\omega,t,\cdot)$, 
$\left(A_{\kappa,i}^2\right)_\varepsilon(\omega,t,\cdot)$, and
$\left(A_{\kappa,i}^3\right)_\varepsilon(\omega,t,\cdot)$ 
belong to $C^\infty(\R^d)$, for each  
fixed $(\omega,t)\in\Omega\times [0,T]$.

\item For $\varepsilon < \varepsilon_\kappa$, cf.~\eqref{eq:def-ek-enull}, 
and for any $(\omega,t)\in\Omega\times [0,T]$, the 
supports of the functions in (1) are contained in 
$$
\kappa\left(\supp\uk\right) +\overline{B_\varepsilon(0)}
\subset\subset \tilde{X}_\kappa. 
$$ 
This implies in particular that the functions 
$\left(A_{\kappa,i}^1\right)_\varepsilon(\omega,t,\cdot)$, 
$\left(A_{\kappa,i}^2\right)_\varepsilon(\omega,t,\cdot)$, 
and
$\left(A_{\kappa,i}^3\right)_\varepsilon(\omega,t,\cdot)$ 
can be seen as elements of $C^\infty(M)$, provided 
that we set them equal to zero outside of $X_\kappa$, 
for each fixed $(\omega,t)\in \Omega\times [0,T]$.

\item For any $p\in [1,2]$ and $(\omega,t)\in\Omega \times [0,T]$,
\begin{align*}
	&\left (A_{\kappa,i}^1\right)_\varepsilon(\omega,t,\cdot)
	\overset{\varepsilon\downarrow 0}{\longrightarrow} 
	\rho(\omega,t,\cdot) a_i(\uk)(\cdot),
	\\ & 
	\left(A_{\kappa,i}^2\right)_\varepsilon(\omega,t,\cdot)
	\overset{\varepsilon\downarrow 0}{\longrightarrow} 
	\rho(\omega,t,\cdot) \nabla^2\uk(a_i,a_i)(\cdot),
	\\ & 
	\left(A_{\kappa,i}^3\right)_\varepsilon(\omega,t,\cdot)
	\overset{\varepsilon\downarrow 0}{\longrightarrow} 
	\rho(\omega,t,\cdot) (\nabla_{a_i}a_i)(\uk)(\cdot),
\end{align*}
where the convergences taking place in $L^p(M)$, and
\begin{align*}
	 & \left(A_{\kappa,i}^1\right)_\varepsilon
	 \overset{\varepsilon\downarrow 0}{\longrightarrow} \rho\, a_i(\uk),
	 \quad
	\left(A_{\kappa,i}^2\right)_\varepsilon
	\overset{\varepsilon\downarrow 0}{\longrightarrow} 
	\rho\, \nabla^2\uk(a_i,a_i),
	\\ &
	\left(A_{\kappa,i}^3\right)_\varepsilon
	\overset{\varepsilon\downarrow 0}{\longrightarrow} 
	\rho\, (\nabla_{a_i}a_i)(\uk),
\end{align*}
in $L^q\left([0,T]; L^2(\Omega\times M)\right)$, 
for any $q\in [1,\infty)$.
\end{enumerate}
\end{lemma}

Define for $(\omega,t,x)\in\Omega\times [0,T] \times M$, 
the global functions
\begin{equation}\label{eq:Ai-def}
	A_{i,\varepsilon}^j(\omega,t,x):=
	\sum_\kappa\mathcal{L}_\kappa
	\left(A_{\kappa,i}^j\right)_\varepsilon(\omega,t,x),
	\qquad j=1,\ldots,3,
\end{equation}
where the pullback-extension operator $\mathcal{L}_\kappa$ 
is defined in \eqref{eq:extension-op}.

Finally, for $\omega\in\Omega$, $t\in[0,T]$, 
and $z\in \tilde{X}_\kappa$, consider the vector field
\begin{equation}\label{eq:def-Aik-4}
	A_{\kappa,i}^4(\omega,t,z)
	:= \seq{\rho(\omega,t,z)a_i(\uk)a_i(z)^l}_l,
\end{equation}
which belongs to $\text{RS}\left(\mathcal{T}^0_1(\tilde{X}_\kappa)\right)$ 
and is compactly supported in $\kappa\left(\supp\uk\right)\subset \tilde{X}_\kappa$, 
uniformly in $\omega,t$. Following our (by now) standard procedure, 
we regularize $A_{\kappa,i}^4$ componentwise via the 
mollifier $\phi_\varepsilon$, cf.~\eqref{eq:mollifier}. 
We denote the resulting object by $\left(A_{\kappa,i}^4\right)_\varepsilon$, 
and observe that, by definition, 
$\left( \left( A_{\kappa,i}^4\right)_\varepsilon\right)^l
=\left(\rho\,a_i(\uk)a_i^l\right)_\varepsilon$. 
We apply the pullback-extension operator $\mathcal{L}_\kappa$ 
to obtain the compactly supported vector field
$$
\mathcal{L}_\kappa 
\left(A_{\kappa,i}^4\right)_\varepsilon
\in \text{SS}\left(\mathcal{T}^0_1(M)\right).
$$
Summing over $\kappa\in\mathcal{A}$, we 
obtain the global object
$$
A_{i,\varepsilon}^4(\omega,t,x)
:=\sum_{\kappa\in\mathcal{A}} 
\mathcal{L}_\kappa \left(A_{\kappa,i}^4\right)_\varepsilon(\omega,t,x).
$$

From the very definitions of $A_{i,\varepsilon}^1$ and 
$\mathcal{L}_\kappa(A_{\kappa,i}^4)_\varepsilon$, 
we have, for $x\in M$,
\[
A_{i,\varepsilon}^1(t,x)\,a_i(x) 
= \sum_{\kappa\in\mathcal{A}}
\left(A_{\kappa,i}^1\right)_\varepsilon(t,x)\, a_i(x)
=\sum_{\kappa\in\mathcal{A}}
\left(\rho(t)\,a_i(\mathcal{U}_\kappa)\right)_\varepsilon(x)
\, a_i(x), 
\]
and
$$
\mathcal{L}_\kappa\left(A_{\kappa,i}^4\right)_\varepsilon(t,x)
=\mathcal{L}_\kappa\left(A_{\kappa,i}^1(t)\,a_i\right)_\varepsilon(x).
$$

We define for $(\omega,t,x)\in\Omega\times [0,T] \times M$ 
the smooth (in $x$) remainder function
\begin{align*}
	r^\ast_{\kappa,\varepsilon,i}(\omega,t,x) 
	& := \Divh \mathcal{L}_\kappa 
	\left(A_{\kappa,i}^4\right)_\varepsilon(t,x)
	-\Divh \left ( 
	\left(\rho(t)\,a_i(\mathcal{U}_\kappa)\right)_\varepsilon\, a_i
	\right)(x)
	\\ & = \Divh \mathcal{L}_\kappa 
	\left(A_{\kappa,i}^1(t)\,a_i\right)_\varepsilon(x) 
	-\Divh \left( \left(A_{\kappa,i}^1\right)_\varepsilon(t,x)\,a_i(x)\right).
\end{align*}
Observe that for $x\in X_\kappa$ in the coordinates 
given by $\kappa$, $z=\kappa(x)$, we have the representation
$$
r^\ast_{\kappa,\varepsilon,i}(t,z)=
\partial_l\left( \rho(t) a_i(\mathcal{U}_\kappa) \, a^l_i\right)_\varepsilon(z)- 
 \partial_l \left( \left(\rho(t) a_i(\mathcal{U}_\kappa)\right)_\varepsilon 
 a^l_i\right)(z).
$$

By now the following lemma should be easy to prove.
\begin{lemma}[DiPerna-Lions commutator; ``$\Divh((\rho a_i(\uk))\, a_i)$"]\label{lem:com-term-4}
For any $q\in [1,\infty)$, $r^\ast_{\kappa,\varepsilon,i}
\overset{\varepsilon\downarrow0}{\longrightarrow} 0$ 
in $L^q([0,T];L^2(\Omega\times M))$.
\end{lemma}

We define a global remainder function 
by summing over $\kappa$:
\begin{equation}\label{eq:commutator-ast}
	r^\ast_{\varepsilon,i}(\omega,t,x):=
	\sum_{\kappa\in\mathcal{A}} 
	r^\ast_{\kappa,\varepsilon,i}(\omega,t,x),
	\qquad (\omega,t,x)\in \Omega\times [0,T]\times M.
\end{equation}
Clearly, $r^\ast_{\varepsilon,i}
\overset{\varepsilon\downarrow0}{\longrightarrow} 0$ 
in $L^q([0,T];L^2(\Omega\times M))$, for 
all $q\in [1,\infty)$.

Finally, we introduce the function
\begin{equation}\label{eq:def-Ak-u}
	A_{\kappa,u}(\omega,t,z):=\rho(\omega,t,z)u(t,z)(\uk),
	\quad 
	(\omega,t,z)\in \Omega\times [0,T]\times \tilde{X}_\kappa,
\end{equation}
which has support contained in $\supp \uk$, 
uniformly in $\omega,t$. We regularize $A_{\kappa,u}$ 
in $z$ using the mollifier $\phi_\varepsilon$, and then 
apply the pullback-extension operator \eqref{eq:extension-op} to 
produce a function defined on $M$. 
We state the following lemma (without proof), noting that it is 
related to Lemmas \ref{lem:prop-rho-kappa-eps} 
and \ref{lem:propA-kappa-eps-j-1-2-3}.

\begin{lemma}\label{lem:prop-A-kappa-eps-u}
Fix $\kappa\in\mathcal{A}$, cf.~Lemma \ref{lem:det-metric-const}. Then
\begin{enumerate}

\item $\left(A_{\kappa,u}\right)_\varepsilon(\omega,t,\cdot)
\in C^\infty(\R^d)$, for $(\omega,t)\in\Omega \times [0,T]$.

\item For $\varepsilon<\varepsilon_\kappa$, cf.~\eqref{eq:def-ek-enull}, 
and $(\omega,t)\in\Omega \times[0,T]$, the support 
of $\left(A_{\kappa,u}\right)_\varepsilon(\omega,t,\cdot)$ 
is contained in 
$$
\kappa\left(\supp\uk\right) +\overline{B_\varepsilon(0)}
\subset\subset \tilde{X}_\kappa. 
$$ 
This implies in particular that 
$(A_{\kappa,u})_\varepsilon(\omega,t,\cdot)$ can 
be seen as a function in $C^\infty(M)$, provided 
that we set it equal to zero outside of $X_\kappa$.

\item $\left(A_{\kappa,u}\right)_\varepsilon 
\overset{\varepsilon\downarrow 0}{\longrightarrow} \rho\, u(\uk)$ 
in $L^1\left([0,T]; L^1(\Omega\times M)\right)$.
\end{enumerate}
\end{lemma}

We also introduce the globally defined function
$$
A_{u,\varepsilon}(\omega,t,x):=
\sum_\kappa(A_{\kappa,u})_\varepsilon(\omega,t,x),
\quad 
(\omega,t,x) \in \Omega\times [0,T] \times M.
$$
Clearly, $ A_{u,\varepsilon}\overset{\varepsilon\downarrow 0}{\longrightarrow} 
\rho\, u(1)=0$ in $L^1\left([0,T]; L^1(\Omega\times M)\right)$.


\subsection{Localized SPDEs and regularization}
Fix $\kappa\in\mathcal{A}$, $\kappa:X_\kappa\subset M\to \tilde{X}$, 
$ z=\kappa(x)$, cf.~Lemma \ref{lem:det-metric-const}, and recall 
that $\seq{\mathcal{U}_\kappa}_{\kappa\in\mathcal{A}}$ denotes
the partition of the unity introduced at the beginning of 
Subsection \ref{sec:local-smoothing}. By inserting into \eqref{eq:L2-weak-Ito} 
the test function  $\psi\, \mathcal{U}_\kappa$, $\psi\in C^\infty(M)$, 
we obtain, $\P$-a.s., for any $t\in[0,T]$,
\begin{equation}\label{eq:L2-weak-Ito-splitted}
	\begin{split}
		\int_M & \rho(t)\,\uk\psi\, dV_h = \int_M \rho_{0}\,\uk\psi\, dV_h 
		+ \sum_{i=1}^N\int_0^t\int_M \rho(s)\,\uk a_i(\psi) \, dV_h\,dW^i(s)
		\\& +\frac12\sum_{i=1}^N\int_0^t \int_M \rho(s)\,\uk
		\Bigl[ (\nabla^2\psi)(a_i,a_i) + (\nabla_{a_i}a_i)(\psi) \Bigr]\, dV_h \,ds
		\\ &+  \sum_{i=1}^N\int_0^t\int_M \rho(s)\,\psi\, a_i(\mathcal{U}_\kappa) \, dV_h\, dW^i(s)
		\\ &+\frac12\sum_{i=1}^N\int_0^t \int_M \rho(s)
		\Bigl[ \psi(\nabla^2\mathcal{U}_\kappa)(a_i,a_i) 
		\\ & \qquad \qquad \qquad \qquad \qquad 
		+ \psi(\nabla_{a_i}a_i)(\mathcal{U}_\kappa)
		+2d\psi\otimes d\mathcal{U}_\kappa(a_i,a_i) \Bigr ]\, dV_h \,ds
		\\ & +\int_0^t\int_M \Bigl[ \rho(s)\,\uk \,u(s)(\psi) 
		+ \rho(s)\psi \,u(s)(\uk)\Bigr]\, dV_h\,ds,
	\end{split}
\end{equation}
where we have used the tensorial identity
$$
\nabla^2 (fg) = g(\nabla^2f) + f(\nabla^2g) 
+ df\otimes dg + dg\otimes df, 
\quad f,g\in C^2(M). 
$$
By making use of the coordinates induced by $\kappa$, and 
some of quantities introduced previously, the equation 
\eqref{eq:L2-weak-Ito-splitted} amounts to writing
\begin{equation}\label{eq:L2-weak-Ito-splitted-2}
	\begin{split}
		\int_{\tilde{X}_\kappa} & \rho_\kappa(t) \psi\, d\bar{z}
		= \int_{\tilde{X}_\kappa} \rho_{0,\kappa} \psi\, d\bar{z} 
		+ \sum_{i=1}^N\int_0^t \int_{\tilde{X}_\kappa} 
		\rho_\kappa(s) a_i(\psi) \, d\bar{z} \,dW^i(s)
		\\ &  
		+\frac12\sum_{i=1}^N \int_0^t \int_{\tilde{X}_\kappa} 
		\rho_\kappa(s)  \Bigl[(\nabla^2\psi)(a_i,a_i) 
		+ (\nabla_{a_i}a_i)(\psi)\Bigr]\, d\bar{z} \,ds
		\\ & 
		+  \sum_{i=1}^N\int_0^t\int_{\tilde{X}_\kappa} 
		A_{\kappa,i}^1(s) \,\psi \, d\bar{z}\, dW^i(s)
		\\ & 
		+\frac12\sum_{i=1}^N\int_0^t \int_{\tilde{X}_\kappa} 
		\Bigl[\psi A_{\kappa,i}^2(s) + \psi A_{\kappa,i}^3(s)
		+ 2A_{\kappa,i}^4(s)(\psi)\Bigr]
		\, d\bar{z} \,ds
		\\ & 
		+\int_0^t\int_{\tilde{X}_\kappa} 
		\Bigl[ \rho_\kappa(s) \,u(s)(\psi) + A_{\kappa,u}(s)\psi \Bigr] \, d\bar{z}\,ds,
	\end{split}
\end{equation}
where $\rho_\kappa, \rho_{0,\kappa}, A_{\kappa,i}^1, A_{\kappa,i}^2, 
A_{\kappa,i}^3, A_{\kappa,i}^4, A_{\kappa,u}$ are defined 
in \eqref{eq:def-rhok-rho0k}, \eqref{eq:def-Aik-1-2-3}, \eqref{eq:def-Aik-4}, \eqref{eq:def-Ak-u}.

For convenience, set $\widetilde{B}_\kappa :=\kappa(\supp \uk) 
+B_{\varepsilon_\kappa/2}(0)$. From now on, we consider only 
$\varepsilon<\varepsilon_\kappa/4$, cf.~\eqref{eq:def-ek-enull}. 
Let us introduce the following family of test functions 
$\psi_{z,\varepsilon}$ (parametrized by $z\in\tilde{X}_\kappa$ 
and $\varepsilon<\varepsilon_\kappa/4$):
$$
\psi_{z,\varepsilon}(\cdot)
:=
\begin{cases}
	\phi_\varepsilon(z-\cdot),  
	& \mbox{if $B_\varepsilon(z)\cap \partial\tilde{X}_\kappa =\emptyset$}
	\\ 0,  & \mbox{otherwise} 
\end{cases}.
$$
We observe that these functions may be seen as elements of 
$C^\infty_c(X_\kappa)$ and that, for fixed $(\omega,t)\in\Omega \times [0,T]$,
$\supp \left((\ldots)_\varepsilon(\omega,t,\cdot)\right)
\subset \kappa(\supp \uk) +\overline{B_\varepsilon(0)}
\subset\subset \widetilde{B}_\kappa$, where $(\cdots)$ denotes 
any one of the objects defined previously. Moreover, for $\omega\in\Omega$, $t\in[0,T]$, 
and $z\in\tilde{X}_\kappa \setminus \widetilde{B}_\kappa$, we have that 
$(\cdots)_\varepsilon(\omega,t,z)=0$ and $(\cdots)_\varepsilon(\omega,t,z)$ 
coincides with the action of $(\cdots)(\omega,t)$ on the function $\psi_{z,\varepsilon}$.

We make use of $\psi_{z,\varepsilon}$ as test function 
in \eqref{eq:L2-weak-Ito-splitted-2}, which results in 
\begin{equation}\label{eq:reg-split-eqn}
	\begin{split}
		(\rho_\kappa & )_\varepsilon(t)(z) -(\rho_{0,\kappa})_\varepsilon(z) 
		= - \sum_{i=1}^N\int_0^t\int_{\tilde{X}_\kappa} \rho_\kappa(s) a_i^l(\bar{z})
		\left(\partial_l\phi_\varepsilon\right)(z-\bar{z}) \,d\bar{z} \, dW^i(s)
		\\ &
		+ \frac12\sum_{i=1}^N\int_0^t \int_{\tilde{X}_\kappa} \rho_\kappa(s)
		\left(\partial_{lm}\phi_\varepsilon\right)(z-\bar{z}) a_i^m(\bar{z}) a_i^l(\bar{z}) 
		\, d\bar{z} \,ds
		\\ & -\frac12\sum_{i=1}^N\int_0^t \int_{\tilde{X}_\kappa} \rho_\kappa(s)
		a_i^m(\bar{z})\partial_m a_i^l(\bar{z}) 
		\left(\partial_l\phi_\varepsilon\right)(z-\bar{z})\, d\bar{z} \,ds
		\\ & 
		+  \sum_{i=1}^N\int_0^t\int_{\tilde{X}_\kappa} \rho(s)\, 
		\phi_\varepsilon(z-\bar{z})\, a_i^l(\bar{z})\partial_l\uk(\bar{z}) \, d\bar{z} \, dW^i(s)
		\\ &
		+\frac12\sum_{i=1}^N\int_0^t \int_{\tilde{X}_\kappa} \rho(s)
		\phi_\varepsilon(z-\bar{z})\partial_{ml} \uk(\bar{z}) a_i^m(\bar{z}) a_i^l(\bar{z}) \,d\bar{z} \,ds
		\\ &
		+\frac12\sum_{i=1}^N\int_0^t \int_{\tilde{X}_\kappa} \rho(s)
		\phi_\varepsilon(z-\bar{z})a_i^l(\bar{z})\partial_l a_i^m(\bar{z}) \partial_m\uk(\bar{z})\,d\bar{z}\,ds
		\\ &
		-\sum_{i=1}^N\int_0^t\int_{\tilde{X}_\kappa}\rho(s)
		\left(\partial_l\phi_\varepsilon\right)(z-\bar{z})\partial_m\uk(\bar{z}) 
		a_i^l(\bar{z}) a_i^m(\bar{z}) \, d\bar{z} \,ds
		\\ & 
		-\int_0^t\int_{\tilde{X}_\kappa} \rho_\kappa(s) \,u^l(s,\bar{z})
		\left(\partial_l\phi_\varepsilon\right)(z-\bar{z})\, d\bar{z} \,ds
		\\ &
		+\int_0^t\int_{\tilde{X}_\kappa} \rho(s) \phi_\varepsilon(z-\bar{z}) 
		\,u^l(s,\bar{z})\partial_l\uk(\bar{z}) \, d\bar{z} \,ds,
	\end{split}
\end{equation}
valid for $z\in\tilde{X}_\kappa$ and $\varepsilon<\varepsilon_\kappa/4$. 
Note that for the terms involving the covariant Hessian we have used 
the geometric identities appearing in Lemma \ref{lem:ai-ai}. 
We can rewrite \eqref{eq:reg-split-eqn} in the following (pointwise) form:
\begin{equation}\label{eq:reg-split-eqn2}
	\begin{split}
		(\rho_\kappa &)_\varepsilon(t,z) -(\rho_{0,\kappa})_\varepsilon(z) 
		= - \sum_{i=1}^N\int_0^t \Divh \mathcal{L}_\kappa 
		\bigl(\rho_\kappa(s) a_i\bigr)_\varepsilon(z) \, dW^i(s) 
		\\ &
		+\frac12\sum_{i=1}^N\int_0^t 
		\Bigl[ 
		\Divh^{2}\mathcal{L}_\kappa \bigl(\rho_\kappa(s) \hat{a}_i\bigr)_\varepsilon(z) 
		- \Divh\mathcal{L}_\kappa \bigl(\rho_\kappa(s)\nabla_{a_i}a_i\bigr)_\varepsilon(z)
		\Bigr] \, ds 
		\\ & +\frac12\sum_{i=1}^N\int_0^t \Bigl[ 
		\Divh\mathcal{L}_\kappa V_{\kappa,i,\varepsilon}(s,z) 
		- \Divh\mathcal{L}_\kappa \bar{V}_{\kappa,i,\varepsilon}(s,z)
		\Bigr]\, ds 
		\\ &
		+ \sum_{i=1}^N\int_0^t \left(A_{\kappa,i}^1\right)_\varepsilon(s,z)\, dW^i(s)
		+\frac12\sum_{i=1}^N\int_0^t \Bigl[
		\left(A_{\kappa,i}^2\right)_\varepsilon(s,z)
		+ \left(A_{\kappa,i}^3\right)_\varepsilon(s,z)
		\Bigr] \,ds 
		\\ &
		-\sum_{i=1}^N\int_0^t\Divh \mathcal{L}_\kappa 
		\left(A_{\kappa,i}^4\right)_\varepsilon(s,z) \,ds
		-\int_0^t\Divh \mathcal{L}_\kappa \left(\rho_\kappa(s)u(s)\right)_\varepsilon(z)\, ds
		\\ & 
		+\int_0^t(A_{\kappa,u})_\varepsilon(s,z)\, ds,
	\end{split}
\end{equation}
where $V_{\kappa,i,\varepsilon}, \bar{V}_{\kappa,i,\varepsilon}$ 
are respectively the regularized versions of $V_{\kappa,i}$, cf.~\eqref{eq:def-Vk}, 
and $ \bar{V}_{\kappa,i}$, cf.~\eqref{eq:def-barVk}.  
Next, we turn these regularized local SPDEs, one equation 
for each chart $\kappa\in \mathcal{A}$, into a globally defined SPDE.


\subsection{Global SPDE for $\rho_\varepsilon$}
\label{subsec: Global regularized equation}
Summing the local equation \eqref{eq:reg-split-eqn2} over 
$\kappa\in\mathcal{A}$, we arrive at the global equation
\begin{align}
	\rho_\varepsilon&(t,x) -\rho_{0,\varepsilon}(x) 
	= - \sum_{i=1}^N\int_0^t\Divh \left(\rho(s)a_i\right)_\varepsilon(x) \, dW^i(s)
	\notag
	\\ & +\frac12\sum_{i=1}^N \int_0^t 
	\Bigl[ \Divh^{2} \left(\rho(s) \hat{a}_i\right)_\varepsilon(x) 
	- \Divh \left(\rho(s)\nabla_{a_i}a_i\right)_\varepsilon(x)\Bigr]\, ds
	\notag
	\\ &
	+\frac12\sum_{i=1}^N\int_0^t 
	\Bigl[ 
	\Divh V_{i,\varepsilon}(s,x) - \Divh \bar{V}_{i,\varepsilon}(s,x)
	\Bigr]\, ds
	\label{eq:global-reg-eqn}
	\\ &
	+\sum_{i=1}^N \int_0^t A_{i,\varepsilon}^1(s,x)\, dW^i(s)
	+\frac12\sum_{i=1}^N \int_0^t 
	\Bigl[
	A_{i,\varepsilon}^2(s,x)+A_{i,\varepsilon}^3(s,x)
	\Bigr] \,ds
	\notag
	\\ &
	-\sum_{i=1}^N\int_0^t \Divh A_{i,\varepsilon}^4(s,x) \,ds
	-\int_0^t\Divh \left(\rho(s)u(s)\right)_\varepsilon(x) \, ds
	+ \int_0^t A_{u,\varepsilon}(s,x)\, ds,
	\notag
\end{align}
which holds $\P$-a.s., for all $(t,x)\in [0,T]\times M$, and 
any $\varepsilon<\varepsilon_0$, cf.~\eqref{eq:def-ek-enull}.

\subsection{Global SPDE for $F(\rho_\varepsilon)$; $u\equiv 0$}
For simplicity of presentation, let us assume that the vector 
field $u(t,\cdot)$ is the zero section for all $t\in[0,T]$, and 
then derive the equation satisfied by the stochastic process 
$(\omega,t)\mapsto F(\rho_\varepsilon(\cdot,\cdot,x))$, $x\in M$ fixed. 
The general case ($u\not \equiv 0$) will be handled later.

Let us apply It\`o's formula to $F(\rho_\varepsilon(\cdot,x))$, where 
$F\in C^2(\R)$ with $F,F',F''$ bounded on $\R$. 
In view of \eqref{eq:global-reg-eqn}, we obtain
\begin{equation}\label{eq:Ito-reg-eqn}
	\begin{split}
		F(\rho_\varepsilon&(t,x)) -F(\rho_{0,\varepsilon}(x))
		 = - \sum_{i=1}^N\int_0^tF'(\rho_\varepsilon(s,x))
		 \Divh \left(\rho(s)a_i\right)_\varepsilon(x)\, dW^i(s)
		 \\ & 
		 +\frac12\sum_{i=1}^N\int_0^t F'(\rho_\varepsilon(s,x))
		 \Bigl[ \Divh^{2} \left(\rho(s) \hat{a}_i\right)_\varepsilon(x) 
		 - \Divh \left(\rho(s)\nabla_{a_i}a_i\right)_\varepsilon(x)
		 \Bigr]\, ds
		 \\ & 
		 +\frac12\sum_{i=1}^N\int_0^t F'(\rho_\varepsilon(s,x))
		 \Bigl[ \Divh V_{i,\varepsilon}(s,x) - \Divh \bar{V}_{i,\varepsilon}(s,x)\Bigr]\, ds
		 \\& 
		 +  \sum_{i=1}^N\int_0^t F'(\rho_\varepsilon(s,x))A_{i,\varepsilon}^1(s,x)\, dW^i(s)
		 \\& 
		 +\frac12\sum_{i=1}^N\int_0^t F'(\rho_\varepsilon(s,x))
		 \Bigl[A_{i,\varepsilon}^2(s,x)+A_{i,\varepsilon}^3(s,x)\Bigr]\,ds
		 \\ & 
		 -\sum_{i=1}^N\int_0^t F'(\rho_\varepsilon(s,x))
		 \Divh A_{i,\varepsilon}^4(s,x) \,ds
		 \\& 
		 +\frac12\sum_{i=1}^N\int_0^t F''(\rho_\varepsilon(s,x))
		 \bigl(\Divh (\rho(s)a_i)_\varepsilon(x)\bigr)^2\,ds
		 \\& 
		 +\frac12  \sum_{i=1}^N\int_0^tF''(\rho_\varepsilon(s,x))
		 \left(A_{i,\varepsilon}^1(s,x)\right)^2\, ds
		 \\ & 
		 -  \sum_{i=1}^N\int_0^tF''(\rho_\varepsilon(s,x)) 
		 \,\Divh \left(\rho(s)a_i\right)_\varepsilon(x) \,A_{i,\varepsilon}^1(s,x)\, ds, 
	\end{split}
\end{equation}
valid $\P$-a.s., for $(t,x)\in [0,T]\times M$, and any $\varepsilon<\varepsilon_0$.

In a nutshell, to prove Theorem \ref{thm:main-result}, i.e., 
the renormalized equation \eqref{eq:main-result-weak-form}, we need to 
send $\varepsilon\to 0$ (after integrating in $x$). The task is rather nontrivial, 
and, before we can accomplish that, we need 
several intermediate results, which will involve 
crucial cancellations among some of the terms in \eqref{eq:Ito-reg-eqn}.

First of all, to bring \eqref{eq:Ito-reg-eqn} into the form of a stochastic 
continuity equation for $F(\rho_\varepsilon(\cdot,x))$, 
we ``add and subtract'' the two terms
$\frac12 \sum_i \int_0^t\Lambda_i\left(F(\rho_\varepsilon)\right)\,ds$ 
and $\sum_i\int_0^t \Divh \left(F(\rho_\varepsilon)\,a_i\right)\, dW^i(s)$.
The resulting equation is
\begin{equation}\label{eq:Ito-reg-eqn-2}
	F(\rho_\varepsilon(t,x)) -F(\rho_{0,\varepsilon}(x)) 
	= \sum_{\ell=1}^{13} \mathcal{I}_\ell(\omega,t,x;\varepsilon),
\end{equation}
where $ \mathcal{I}_\ell= \mathcal{I}_\ell(\omega,t,x;\varepsilon)$, 
$\ell=1,\ldots, 13$, are defined by
\begin{equation}\label{eq:Ito-reg-eqn-2-defs}
	\begin{split}
 		& \mathcal{I}_1= -\sum_{i=1}^N\int_0^t 
		\Divh \left(F(\rho_\varepsilon(s,x))\,a_i\right)\, dW^i(s),
		\\ & 
		\mathcal{I}_2 = \frac12 \sum_{i=1}^N\int_0^t
		\Lambda_i \left(F(\rho_\varepsilon(s,x))\right) \,ds,
		\\ &
		\mathcal{I}_3 = - \sum_{i=1}^N \int_0^t 
		F'(\rho_\varepsilon(s,x)) \Divh \left(\rho(s)a_i\right)_\varepsilon(x)\, dW^i(s),
		\\ &
		\mathcal{I}_4 = \sum_{i=1}^N\int_0^t
		\Divh \left(F(\rho_\varepsilon(s,x))\,a_i\right)\, dW^i(s),
		\\  & 
		\mathcal{I}_5 =\frac12\sum_{i=1}^N\int_0^t F'(\rho_\varepsilon(s,x))
		\Bigl[ \Divh^{2}\left(\rho(s) \hat{a}_i\right)_\varepsilon(x) 
		- \Divh \left(\rho(s)\nabla_{a_i}a_i\right)_\varepsilon(x)\Bigr]\, ds,
		\\ & 
		\mathcal{I}_6 = -\frac12 \sum_{i=1}^N\int_0^t
		\Lambda_i\left(F(\rho_\varepsilon(s,x))\right)\,ds,
		\\ & 
		\mathcal{I}_7 =\frac12\sum_{i=1}^N\int_0^t F'(\rho_\varepsilon(s,x))
		\Bigl[ \Divh V_{i\,\varepsilon}(s,x) 
		- \Divh \bar{V}_{i\,\varepsilon}(s,x)\Bigr]\, ds,
		\\ & 
		\mathcal{I}_8= \sum_{i=1}^N\int_0^t F'(\rho_\varepsilon(s,x))\, 
		A_{i,\varepsilon}^1(s,x)\, dW^i(s),
		\\ & 
		\mathcal{I}_9 = \frac12\sum_{i=1}^N\int_0^t F'(\rho_\varepsilon(s,x))
		\Bigl[A_{i,\varepsilon}^2(s,x)+A_{i,\varepsilon}^3(s,x)\Bigr]\,ds,
		\\ &
		\mathcal{I}_{10} =-\sum_{i=1}^N\int_0^t F'(\rho_\varepsilon(s,x))
		\Divh A_{i,\varepsilon}^4(s,x) \,ds,
		\\ &
		\mathcal{I}_{11} =\frac12\sum_{i=1}^N\int_0^tF''(\rho_\varepsilon(s,x))
		\bigl(\Divh \left(\rho(s)a_i\right)_\varepsilon(x)\bigr)^2\,ds,
		\\ &
		\mathcal{I}_{12} =\frac12  \sum_{i=1}^N\int_0^tF''(\rho_\varepsilon(s,x))
		\left(A_{i,\varepsilon}^1(s,x)\right)^2\, ds,
		\\ & \mathcal{I}_{13} =-  \sum_{i=1}^N\int_0^t F''(\rho_\varepsilon(s,x))
		\, \Divh \left(\rho(s)a_i\right)_\varepsilon(x) 
		\,A_{i,\varepsilon}^1(s,x)\, ds.
	\end{split}
\end{equation}

\begin{lemma}[``$\mathcal{I}_3+\mathcal{I}_4$'']\label{lem:cor-to-commutator1}
With $\mathcal{I}_3$ and $\mathcal{I}_4$ defined in 
\eqref{eq:Ito-reg-eqn-2-defs},
\begin{align*}
	&\mathcal{I}_3+\mathcal{I}_4=\mathcal{I}_{3+4,1}+\mathcal{I}_{3+4,2},
	\\ & \mathcal{I}_{3+4,1} := 
	- \sum_{i=1}^N\int_0^tG_F(\rho_\varepsilon(s,x)) \Divh a_i(x)\, dW^i(s),
	\\ & 
	\mathcal{I}_{3+4,2} :=
	- \sum_{i=1}^N\int_0^t
	F'(\rho_\varepsilon(s,x)) \,r_{\varepsilon,i}(s,x)\, dW^i(s),
\end{align*}
where $G_F$ is defined in \eqref{eq:GF-def} and $r_{\varepsilon,i}$ is 
defined in \eqref{eq:commutator-r}. 
\end{lemma}

\begin{proof}
By definition,  in a coordinate-free notation,
\begin{align*}
	& \Divh \left( \rho(s) a_i\right)_\varepsilon(x) 
	= \sum_{\kappa \in \mathcal{A}}
	\Divh \mathcal{L}_\kappa \left(\rho_\kappa(s)a_i \right)_\varepsilon(x)
	 \\ & \quad 
	 = \sum_{\kappa \in \mathcal{A}}
	 \Bigl \{ 
	 \Divh \bigl( (\rho_\kappa)_\varepsilon(s) a_i \bigr)(x) 
	 + r_{\kappa,\varepsilon,i}(s,x)
	 \Bigr\}
	 =\Divh (\rho_\varepsilon(s) a_i)(x) + r_{\varepsilon,i}(s,x).
\end{align*}
Therefore, by the product and chain rules, 
recalling \eqref{eq:GF-def}, the lemma follows.
\end{proof}

\begin{lemma}[``$\mathcal{I}_5+\mathcal{I}_6$'']\label{lem:cor-to-commutator2}
With $\mathcal{I}_5$, $\mathcal{I}_6$ defined in \eqref{eq:Ito-reg-eqn-2-defs}, 
$\mathcal{I}_5+\mathcal{I}_6
= \sum\limits_{\ell=1}^4 \mathcal{I}_{5+6,\ell}$,
\begin{align*}
	& \mathcal{I}_{5+6,1} := 
	\frac12\sum_{i=1}^N \int_0^t F'(\rho_\varepsilon(s,x))
	\Bigl[ \Divh^{2}\left(\rho(s) \hat{a}_i\right)_\varepsilon(x) 
	- \tilde{r}_{\varepsilon,i} \Bigr] \, ds,
	\\ & 
	\mathcal{I}_{5+6,2} :=
	-\frac12\sum_{i=1}^N\int_0^t F''(\rho_\varepsilon(s,x)) 
	\bigl(a_i(\rho_\varepsilon(s,x))\bigr)^2\,ds,
	\\ 
	& \mathcal{I}_{5+6,3} := 
	\frac12\sum_{i=1}^N\int_0^t G_F(\rho_\varepsilon(s,x)) \Lambda_i(1)\,ds,
	\\ 
	& \mathcal{I}_{5+6,4} := 
	-\frac12\sum_{i=1}^N\int_0^t F'(\rho_\varepsilon(s,x)) 
	\Divh^{2}\left(\rho_\varepsilon(s) \hat{a}_i \right)(x) \, ds,
\end{align*}
where $\tilde{r}_{\varepsilon,i}$, $G_F$, $\Lambda_i(1)$ are defined 
respectively in \eqref{eq:commutator-tilde}, \eqref{eq:GF-def}, 
\eqref{eq:bar-ai-Lambda-i(1)}.
\end{lemma}

\begin{proof}
By \eqref{eq:expand-Lambda_i(F)} and arguing 
exactly as we have done several times before (i.e., expanding the sum 
over $\kappa$ that defines the global objects, working locally relative 
to a chart $\kappa$, and in the end reaggregate), we arrive at 
$\mathcal{I}_5+\mathcal{I}_6 = \sum_{\ell=1}^8 \mathcal{J}_{\ell}$, 
where $\mathcal{J}_{\ell}=\mathcal{J}_{\ell}(\omega,t,x,\varepsilon)$, 
$\ell=1,\ldots, 8$, are defined by
\begin{align*}
	& \mathcal{J}_{1} = I_{5+6,1}, \quad
	\mathcal{J}_{2} = -\frac12\sum_{i=1}^N\int_0^t F'(\rho_\varepsilon(s,x))
	\Divh \left(\rho_\varepsilon(s) \nabla_{a_i}a_i \right)(x) \, ds,
	\\ & 
	\mathcal{J}_{3} 
	=-\frac12\sum_{i=1}^N\int_0^t F'(\rho_\varepsilon(s,x)) 
	\left(\Divh(\hat{a}_i)\right)(\rho_\varepsilon(s,x)) \,ds,
	\\ &
	\mathcal{J}_{4} 
	=-\frac12\sum_{i=1}^N\int_0^t F(\rho_\varepsilon(s,x)) 
	\Divh^{2} \left(\hat{a}_i\right)(x)\,ds,
	\\ & 
	\mathcal{J}_{5} 
	= -\frac12\sum_{i=1}^N\int_0^t F'(\rho_\varepsilon(s,x)) 
	\Divh \left(\hat{a}_i(d\rho_\varepsilon(s),\cdot)\right)(x)\,ds,
	\\ & 
	\mathcal{J}_{6}=\mathcal{I}_{5+6,2},
	\quad \mathcal{J}_{7} 
	= \frac12\sum_{i=1}^N\int_0^t F'(\rho_\varepsilon(s,x)) 
	(\nabla_{a_i}a_i)(\rho_\varepsilon(s,x))\,ds, 
	\\ & 
	\mathcal{J}_{8} 
	= \frac12\sum_{i=1}^N\int_0^t F(\rho_\varepsilon(s,x)) 
	\Divh(\nabla_{a_i}a_i)(x)\,ds.
\end{align*}
By expanding $\Divh \left(\rho_\varepsilon(s) \nabla_{a_i}a_i \right)=
\rho_\varepsilon(s) \Divh(\nabla_{a_i}a_i) + (\nabla_{a_i}a_i)(\rho_\varepsilon(s))$ 
and recalling the definition of $G_F$, we obtain
\begin{align*}
	\mathcal{I}_5+\mathcal{I}_6 & = 
	\mathcal{J}_{1} + \mathcal{J}_{3}
	+ \mathcal{J}_{4}+ \mathcal{J}_{5}
	+ \mathcal{J}_{6}
	-\frac12\sum_{i=1}^N\int_0^t 
	G_F(\rho_\varepsilon(s,x)) \Divh (\nabla_{a_i}a_i)(x) \,ds.
\end{align*}
Considering \eqref{eq:expand-Lambda_i(F)} 
with $F(u)=u$ and $u\in C^2(M)$, 
$$
\Divh^{2} \left(u\hat{a}_i\right) 
= \Divh \left(\hat{a}_i\right)(u) 
+ u\,\Divh^{2}\left(\hat{a}_i\right) 
+ \Divh \left(\hat{a}_i(du,\cdot) \right).
$$
Using this identity and recalling once more \eqref{eq:GF-def}, 
\begin{align*}
	\mathcal{I}_5+\mathcal{I}_6 
	& = \mathcal{J}_{1}+\mathcal{J}_{6} 
	+\frac12\sum_{i=1}^N\int_0^t G_F(\rho_\varepsilon(s,x)) 
	\Bigl[ \Divh^{2}(\hat{a}_i) -\Divh(\nabla_{a_i}a_i) \Bigr] \,ds
	+\mathcal{I}_{5+6,4}.
\end{align*}
Recalling \eqref{eq:bar-ai-Lambda-i(1)}, the third term 
on the right-hand side of the equality sign is $\mathcal{I}_{5+6,3}$. 
Since $\mathcal{J}_{1} = I_{5+6,1}$ and $\mathcal{J}_{6}=\mathcal{I}_{5+6,2}$, 
this concludes the proof.
\end{proof}

\begin{lemma}[``$\mathcal{I}_{11}$'']\label{lem:quad-term}
With $\mathcal{I}_{11}$ defined in \eqref{eq:Ito-reg-eqn-2-defs}, 
$\mathcal{I}_{11}= \sum\limits_{\ell=1}^4 \mathcal{I}_{11,\ell}$,
\begin{align*}
	&\mathcal{I}_{11,1}
	:= \frac12\sum_{i=1}^N\int_0^tF''(\rho_\varepsilon(s,x))
	\bigl(a_i(\rho_\varepsilon(s,x))\bigr)^2  \,ds,
	\\ & 
	\mathcal{I}_{11,2}
	:= \frac12\sum_{i=1}^N\int_0^tF''(\rho_\varepsilon(s,x))
	\bigl(\rho_\varepsilon(s,x)\Divh a_i\bigr)^2 \,ds,
	\\ &
	\mathcal{I}_{11,3}
	:= \frac12\sum_{i=1}^N \int_0^t F''(\rho_\varepsilon(s,x))
	\bigl(r_{\varepsilon,i}(s,x)\bigr)^2\,ds,
	\\ &
	\mathcal{I}_{11,4}
	:=\sum_{i=1}^N\int_0^tF''(\rho_\varepsilon(s,x))
	\Divh \left(\rho_\varepsilon(s,x)a_i\right) 
	\,r_{\varepsilon,i}(s,x)\,ds,
	\\ &
	\mathcal{I}_{11,5}
	:=\sum_{i=1}^N\int_0^t 
	\bar{a}_i\left(G_F(\rho_\varepsilon(s,x))\right)\,ds,
\end{align*}
where $r_{\varepsilon,i}$ is defined in \eqref{eq:commutator-r}, 
$\bar{a}_i = (\Divh a_i) \,a_i$, and $G_F$ is defined in \eqref{eq:GF-def}.
\end{lemma}

\begin{proof}
Writing 
\begin{equation}\label{eq:div+error}
	\Divh \left(\rho(s)a_i\right)_\varepsilon(x) 
	=\Divh \left(\rho_\varepsilon(s)a_i\right)(x) 
	+ r_{\varepsilon,i}(s,x)
\end{equation}
and expanding
$\bigl(\Divh \left(\rho_\varepsilon(s)a_i\right)(x) 
+ r_{\varepsilon,i}(s,x)\bigr)^2$, we obtain 
$\mathcal{I}_{11}=\mathcal{J}_{1}
+\mathcal{J}_{2}+\mathcal{J}_{3}$, where
\begin{align*}
	&\mathcal{J}_{1}
	= \frac12\sum_{i=1}^N\int_0^t F''(\rho_\varepsilon(s,x))
	\bigl(\Divh \left(\rho_\varepsilon(s)a_i\right)(x)\bigr)^2 \,ds,
	\quad 
	\mathcal{J}_{2}=\mathcal{I}_{11,2},
	\quad 
	\mathcal{J}_{3}=\mathcal{I}_{11,3}.
\end{align*}
By applying the product and chain rules to 
$\Divh \left(\rho_\varepsilon a_i\right)$ and then expanding the square 
$\bigl(\Divh \left(\rho_\varepsilon(s) a_i\right)\bigr)^2$
into $\bigl(a_i(\rho_\varepsilon(s))\bigr)^2 + (\rho_\varepsilon(s)\Divh a_i)^2 +
2\rho_\varepsilon(s)\bar{a}_i(\rho_\varepsilon(s))$,
\begin{align*}
	\mathcal{J}_{1} &=\mathcal{I}_{11,1}+\mathcal{I}_{11,2}
	+ \sum_{i=1}^N\int_0^t \rho_\varepsilon(s,x)
	\,F''(\rho_\varepsilon(s,x))\, \bar{a}_i\left(\rho_\varepsilon(s,x)\right)\,ds.
\end{align*}
Because $G_F'(\xi)=\xi\,F''(\xi)$ and 
$G_F'(\rho_\varepsilon)\bar{a}_i\left(\rho_\varepsilon\right)
=\bar{a}_i(G_F(\rho_\varepsilon))$, the last term becomes 
$\mathcal{I}_{11,4}$. This concludes the proof of the lemma.
\end{proof}

In view of Lemmas \ref{lem:cor-to-commutator1}, \ref{lem:cor-to-commutator2}, \ref{lem:quad-term} 
and noting the cancellation $\mathcal{I}_{5+6,2}+\mathcal{I}_{11,1}=0$, the 
equation \eqref{eq:Ito-reg-eqn-2} becomes
\begin{equation}\label{eq:Ito-reg-eqn-3}
	\begin{split}
		&F(\rho_\varepsilon(t,x)) -F(\rho_{0,\varepsilon}(x))
		\\ & \quad
		=\mathcal{I}_{1}+\mathcal{I}_{2}+\mathcal{I}_{3+4,1}
		+\mathcal{I}_{3+4,2}
		+\mathcal{I}_{5+6,1}+\mathcal{I}_{5+6,3}+\mathcal{I}_{5+6,4}
		\\ & \qquad\quad
		+\mathcal{I}_{11,2}+\mathcal{I}_{11,3}+\mathcal{I}_{11,4}
		+\mathcal{I}_{11,5}
		+\mathcal{I}_{7}+\mathcal{I}_{8}+\mathcal{I}_{9}+\mathcal{I}_{10}
		+\mathcal{I}_{12}+ \mathcal{I}_{13}.
	\end{split}
\end{equation}
Keeping an eye on the end result \eqref{eq:main-result} ($u\equiv 0$) 
while inspecting \eqref{eq:Ito-reg-eqn-3} as well 
as reorganizing and relabeling some of the terms, we 
rewrite \eqref{eq:Ito-reg-eqn-3} in the form
\begin{equation}\label{eq:Ito-reg-eqn-4}
	\begin{split}
		&F(\rho_\varepsilon(t,x)) -F(\rho_{0,\varepsilon}(x))
		+\sum_{i=1}^N\int_0^t
		\Divh \left(F(\rho_\varepsilon(s,x))\,a_i\right)\, dW^i(s)
		\\ & \quad 
		+ \sum_{i=1}^N\int_0^tG_F(\rho_\varepsilon(s,x))\Divh a_i\, dW^i(s)
		=\frac12 \sum_{i=1}^N\int_0^t \Lambda_i(F(\rho_\varepsilon(s,x)))\,ds
		\\ & \quad \quad 
		+\frac12\sum_{i=1}^N\int_0^t G_F(\rho_\varepsilon(s,x))
		\,\Lambda_i(1)\, ds
		+\frac12\sum_{i=1}^N\int_0^t F''(\rho_\varepsilon(s,x))
		\, \bigl(\rho_\varepsilon(s,x)\Divh a_i\bigr)^2\, ds
		\\ &\quad\quad\quad
		+\sum_{i=1}^N\int_0^t \bar{a}_i(G_F(\rho_\varepsilon(s,x)))\, ds
		+\mathcal{R}_\varepsilon(\omega,t,x),
	\end{split}
\end{equation}
where the third, fourth, fifth, sixth, seventh, 
and eighth terms correspond to $\mathcal{I}_1$, $\mathcal{I}_2$, 
$\mathcal{I}_{3+4,1}$, $\mathcal{I}_{5+6,3}$, 
$\mathcal{I}_{11,2}$, $\mathcal{I}_{11,5}$, respectively. The 
remaining terms from \eqref{eq:Ito-reg-eqn-3} 
have been transferred to the remainder $\mathcal{R}_\varepsilon$. 
Modulo the ``$\varepsilon$-subscripts'' and the remainder 
term, we recognize \eqref{eq:Ito-reg-eqn-4} as 
the sought-after renormalized equation \eqref{eq:main-result}, cf.~also 
Remark \ref{rem:modification-of-some-terms}. 
One of the remaining tasks is to demonstrate that 
$\mathcal{R}_\varepsilon\to 0$ as $\varepsilon\to 0$, 
weakly in $x$ and strongly in $(\omega,t)$. 
The remainder term $\mathcal{R}_\varepsilon
=\mathcal{R}_\varepsilon(\omega,t,x)$ is 
\begin{equation}\label{eq:the-remainder}
	\begin{split}
		& \mathcal{R}_\varepsilon=\sum_{\ell=1}^{11} \mathcal{H}_\ell, 
		\quad \text{where} \quad
		\mathcal{H}_1:= \mathcal{I}_{3+4,2},
		\\ &
		\mathcal{H}_2 := \frac12\sum_{i=1}^N\int_0^t
		F'(\rho_\varepsilon(s,x))
		\Bigl[ \Divh^{2} \left(\rho(s)\hat{a}_i\right)_\varepsilon(x)
		-\Divh^{2}\left( \rho_\varepsilon(s) \hat{a}_i\right)(x) 
		\Bigr]\, ds,
		\\ & 
		\mathcal{H}_3 :=-\frac12\sum_{i=1}^N\int_0^t 
		F'(\rho_\varepsilon(s,x))\, 
		\tilde{r}_{\varepsilon,i}(s,x)\,ds,
		\\ &
		\mathcal{H}_4 :=\mathcal{I}_{11,3},
		\quad 
		\mathcal{H}_5:= \mathcal{I}_{11,4},
		\quad
		\mathcal{H}_6 :=\mathcal{I}_{7},
		\quad
		\mathcal{H}_7:=\mathcal{I}_{8},
		\\ &
		\mathcal{H}_8 :=\mathcal{I}_{9},
		\quad
		\mathcal{H}_9 := \mathcal{I}_{10}, 
		\quad
		\mathcal{H}_{10} := \mathcal{I}_{12},
		\quad 
		\mathcal{H}_{11} := \mathcal{I}_{13}.
	\end{split}
\end{equation}

Some crucial cancellations will occur also in the remainder 
term $\mathcal{R}_\varepsilon$. This is the subject matter 
of the remaining lemmas in this subsection, involving 
the terms $\mathcal{H}_2$, $\mathcal{H}_5$, $\mathcal{H}_6$, 
and $\mathcal{H}_9$.

\begin{lemma}[``$\mathcal{H}_5$'']\label{lem:related-to-Reps-1}
With $\mathcal{H}_5=\mathcal{I}_{11,4}$ defined in 
Lemma \ref{lem:quad-term},
\begin{align*}
	&\mathcal{H}_5 =\mathcal{H}_{5,1}+\mathcal{H}_{5,2},
	\quad \text{where}
	\\ & 
	\mathcal{H}_{5,1} :=
	\sum_{i=1}^N\int_0^t 
	G_F'(\rho_\varepsilon(s,x))\, \Divh a_i 
	\,r_{\varepsilon,i}(s,x) \,ds,
	\\ & 
	\mathcal{H}_{5,2}:=
	-\sum_{i=1}^N\int_0^t 
	F'(\rho_\varepsilon(s,x)) 
	\, a_i\bigl(r_{\varepsilon,i}(s,x)\bigr)\,ds,
	\\ & \mathcal{H}_{5,3}:=
	\sum_{i=1}^N\int_0^t a_i\bigl(F'(\rho_\varepsilon(s))
	\,r_{\varepsilon,i}(s)\bigr)(x) \,ds, 
\end{align*}
where $r_{\varepsilon,i}$ is defined in \eqref{eq:commutator-r} 
and $G_F$ is defined in \eqref{eq:GF-def}.
\end{lemma}

\begin{proof}
In $\mathcal{I}_{11,4}$, expand $\Divh\left(\rho_\varepsilon(s) a_i\right)$ 
and then use the product rule for $a_i$, finally noting 
that $\xi F''(\xi)=G_F'(\xi)$.
\end{proof}

\begin{lemma}[``$\mathcal{H}_2+\mathcal{H}_6$"]\label{lem:related-to-Reps-2}
Consider $\mathcal{H}_2$ defined in \eqref{eq:the-remainder} 
and $\mathcal{H}_6=\mathcal{I}_7$ with $\mathcal{I}_7$ 
defined in \eqref{eq:Ito-reg-eqn-2-defs}. Then
\begin{align*}
	&\mathcal{H}_2+\mathcal{H}_6 
	= \sum_{\ell=1}^3\mathcal{H}_{2+6,\ell}, 
	\quad \text{where}
	\\ &
	\mathcal{H}_{2+6,1}:=\sum_{i=1}^N\int_0^t 
	F'(\rho_\varepsilon(s,x))
	a_i\bigl(r_{\varepsilon,i}(s,x)\bigr)\,ds,
	\\ &
	\mathcal{H}_{2+6,2}:=\frac12\sum_{i=1}^N\int_0^t 
	F'(\rho_\varepsilon(s,x))
	\bar{r}_{\varepsilon,i}(s,x)\,ds,
	\\ &
	\mathcal{H}_{2+6,3}:=
	\sum_{i=1}^N\int_0^t F'(\rho_\varepsilon(s,x))
	\sum_{\kappa\in\mathcal{A}} G_{\kappa,\varepsilon,i}(s,x)\,ds,
\end{align*}
where $r_{\varepsilon,i}$, $\bar{r}_{\varepsilon,i}$, 
$G_{\kappa,\varepsilon,i}$ are defined 
respectively in \eqref{eq:commutator-r}, 
\eqref{eq:commutator-bar}, \eqref{eq:def-Gkei}.
\end{lemma}

\begin{proof}
By inspecting the integrands of $\mathcal{H}_2$ and 
$\mathcal{H}_6$, we recognize that the claim follows from the 
``second order" commutator Lemma \ref{lem:Reps(2)-F}. 
\end{proof}

\begin{lemma}[``$\mathcal{H}_9$'']\label{lem:related-to-Reps-3}
With $\mathcal{H}_9=\mathcal{I}_{10}$ 
defined in \eqref{eq:Ito-reg-eqn-2-defs},
\begin{align*}
	& \mathcal{H}_9=\mathcal{H}_{9,1}+\mathcal{H}_{9,2},
	\quad \text{where}
	\\ &
	\mathcal{H}_{9,1}
	:=-\sum_{i=1}^N\int_0^tF'(\rho_\varepsilon(s,x))
	\, r^\ast_{\varepsilon,i}(s,x)\,ds,
	\\ &
	\mathcal{H}_{9,2}:=-\sum_{i=1}^N\int_0^tF'(\rho_\varepsilon(s,x))
	\Divh \left( A_{i,\varepsilon}^1(s,x)\,a_i \right)\,ds,
\end{align*}
where $r^\ast_{\varepsilon,i}$, $A_{i,\varepsilon}^1$ are 
defined respectively in \eqref{eq:commutator-ast}, \eqref{eq:Ai-def}.
\end{lemma}

\begin{proof}
This follows from the very definition of $r^\ast_{\varepsilon,i}$.
\end{proof}

\begin{lemma}[``$\mathcal{H}_{11}$'']\label{lem:related-to-Reps-4}
With $\mathcal{H}_9=\mathcal{I}_{13}$ 
defined in \eqref{eq:Ito-reg-eqn-2-defs}
\begin{align*}
	& \mathcal{H}_{11}=\mathcal{H}_{11,1}+\mathcal{H}_{11,2},
	\quad \text{where}	
	\\ & 
	\mathcal{H}_{11,1}:=
	-\sum_{i=1}^N\int_0^t A_{i,\varepsilon}^1(s,x)\,
	a_i\bigl(F'(\rho_\varepsilon(s,x))\bigr)\,ds,
	\\ &
	\mathcal{H}_{11,2}:=
	-\sum_{i=1}^N\int_0^t F''(\rho_\varepsilon(s,x))
	A_{i,\varepsilon}^1(s,x)
	\Bigl( r_{\varepsilon,i}(s,x) 
	+ \rho_\varepsilon(s,x)\Divh a_i\Bigr) \,ds,
\end{align*}
where $r_{\varepsilon,i}$, $A_{i,\varepsilon}^1$ are 
defined respectively in \eqref{eq:commutator-r}, \eqref{eq:Ai-def}.
\end{lemma}

\begin{proof}
Making use of \eqref{eq:div+error} and the 
product rule, 
$$
\Divh \left(\rho(s)a_i\right)_\varepsilon
=r_{\varepsilon,i}(s) 
+\rho_\varepsilon(s)\Divh a_i + 
a_i\left(\rho_\varepsilon(s)\right).
$$ 
Thus the claim follows by noting that
$F''(\rho_\varepsilon(s))\, 
a_i(\rho_\varepsilon(s)) 
=a_i\bigl(F'(\rho_\varepsilon(s))\bigr)$.
\end{proof}

\begin{lemma}[``$\mathcal{H}_9+\mathcal{H}_{11}$'']\label{lem:related-to-Reps-5}
Consider $\mathcal{H}_9=\mathcal{I}_{10}$ and 
$\mathcal{H}_{11}=\mathcal{I}_{13}$ 
with $\mathcal{I}_{10}$ and $\mathcal{I}_{13}$ defined 
in \eqref{eq:Ito-reg-eqn-2-defs}. Then
\begin{align*}
	& \mathcal{H}_{9}+\mathcal{H}_{11}
	=\sum_{\ell=1}^4 \mathcal{H}_{9+11,\ell},
	\quad \text{where}
	\\ &
	\mathcal{H}_{9+11,1} 
	:= -\sum_{i=1}^N\int_0^tF'(\rho_\varepsilon(s,x))
	\, r^\ast_{\varepsilon,i}(s,x)\,ds,
	\\ &
	\mathcal{H}_{9+11,2}
	:=-\sum_{i=1}^N\int_0^tF'(\rho_\varepsilon(s,x))
	\, A_{i,\varepsilon}^1(s,x)\, \Divh a_i\,ds,
	\\ &
	\mathcal{H}_{9+11,3}
	:=-\sum_{i=1}^N\int_0^t
	a_i\bigl(F'(\rho_\varepsilon(s,x))A_{i,\varepsilon}^1(s,x)\bigr)\,ds,
	\\ &
	\mathcal{H}_{9+11,4}
	:=-\sum_{i=1}^N\int_0^t F''(\rho_\varepsilon(s,x))
	A_{i,\varepsilon}^1(s,x)
	\bigl (r_{\varepsilon,i}(s,x) 
	+\rho_\varepsilon(s,x)\, \Divh a_i\bigr) \,ds,
\end{align*}
where $r^\ast_{\varepsilon,i}$, $A_{i,\varepsilon}^1$, 
$r_{\varepsilon,i}$ are defined respectively 
in \eqref{eq:commutator-ast}, \eqref{eq:Ai-def}, 
\eqref{eq:commutator-r}.
\end{lemma}

\begin{proof}
Note that $\mathcal{H}_{9+11,1}=\mathcal{H}_{9,1}$, 
cf.~Lemma \ref{lem:related-to-Reps-3},  
and $\mathcal{H}_{9+11,4}=\mathcal{H}_{11,2}$, 
cf.~Lemma \ref{lem:related-to-Reps-4}. 
Moreover, applying the product rule to 
$\Divh \left( A_{i,\varepsilon}^1(s,x)\,a_i \right)$, 
we find that $\mathcal{H}_{9+11,2}+\mathcal{H}_{9+11,3}
=\mathcal{H}_{9,2}+\mathcal{H}_{11,1}$. 
\end{proof}

Combining \eqref{eq:the-remainder} and 
Lemmas \ref{lem:related-to-Reps-1}, \ref{lem:related-to-Reps-2}, 
and \ref{lem:related-to-Reps-5}, we arrive at the following expression 
for the remainder term $\mathcal{R}_\varepsilon$, which is a function 
of $(\omega,t,x)\in\Omega\times [0,T] \times M$:
\begin{equation}\label{eq:the-remainder-2}
	\begin{split}
		\mathcal{R}_\varepsilon & = \mathcal{H}_1
		+\mathcal{H}_3+\mathcal{H}_4+\mathcal{H}_7
		+\mathcal{H}_8+\mathcal{H}_{10}
		+\mathcal{H}_{5,1}+\mathcal{H}_{5,3}
		\\ & \qquad +\mathcal{H}_{2+6,2}+\mathcal{H}_{2+6,3}
		+\mathcal{H}_{9+11,1}+\mathcal{H}_{9+11,2}
		+\mathcal{H}_{9+11,3}+\mathcal{H}_{9+11,4},
	\end{split}
\end{equation}
where the difficult terms involving $a_i\bigl(r_{\varepsilon,i}\bigr)$ 
cancel out: $\mathcal{H}_{2+6,1}+\mathcal{H}_{5,2}=0$.


\subsection{Passing to the limit in SPDE for $F(\rho_\varepsilon)$; $u\equiv 0$}
We wish to send $\varepsilon\to 0$ in the $x$-weak 
form of \eqref{eq:Ito-reg-eqn-4}, analyzing the limiting 
behavior of the remainder $\mathcal{R}_\varepsilon$, 
which is going to vanish, separately from the other terms 
in \eqref{eq:Ito-reg-eqn-4}, which are going to 
converge to their respective terms in \eqref{eq:main-result-weak-form}. 
Denote by $\langle\cdot,\cdot\rangle$ the canonical 
pairing between functions on $M$. In the following, 
we will make repeated (unannounced) use of Lemma \ref{lem:classical-result} 
and the (stochastic) Fubini theorem.

We begin with the convergence analysis of the remainder term.

\begin{proposition}[convergence of the remainder $\mathcal{R}_\varepsilon$]
\label{prop:the-remainder-vanishes}
For any $\psi\in C^\infty(M)$, 
$\bigl\langle \mathcal{R}_\varepsilon,\psi\bigr\rangle
\to 0$ in $L^1(\Omega \times [0,T])$ as $\varepsilon\to 0$.
\end{proposition}

\begin{proof}
We recall the following convergences, 
which are consequences of Lemmas \ref{lem:prop-rho-kappa-eps}, 
\ref{lem:com-term-1}, \ref{lem:com-term-2}, \ref{lem:com-term-3}, 
\ref{lem:Reps(2)-F}, \ref{lem:propA-kappa-eps-j-1-2-3}, 
and \ref{lem:com-term-4}:
\begin{equation}\label{eq:main-convergences}
	\begin{split}
		& \rho_\varepsilon 
		\overset{\varepsilon\downarrow 0}{\longrightarrow} \rho, 
		\quad
		\rho_{0,\varepsilon} 
		\overset{\varepsilon\downarrow 0}{\longrightarrow} \rho_0,
		\quad 
		r_{\varepsilon,i}
		\overset{\varepsilon\downarrow 0}{\longrightarrow} 0, 
		\quad 
		\tilde{r}_{\varepsilon,i}
		\overset{\varepsilon\downarrow 0}{\longrightarrow} 0, 
		\\ & 
		\bar{r}_{\varepsilon,i}
		\overset{\varepsilon\downarrow 0}{\longrightarrow} 0,
		\quad
		r^\ast_{\varepsilon,i}
		\overset{\varepsilon\downarrow 0}{\longrightarrow} 0,
		\quad 
		G_{\kappa,\varepsilon,i}
		\overset{\varepsilon\downarrow 0}{\longrightarrow} 0,
		\quad  
		A_{i,\varepsilon}^1
		\overset{\varepsilon\downarrow 0}{\longrightarrow} 
		\rho\, a_i(1)=0,
		\quad 
		\\ & A_{i,\varepsilon}^2 
		\overset{\varepsilon\downarrow 0}{\longrightarrow} 
		\rho\, \nabla^2 1(a_i,a_i)=0, 
		\quad
		A_{i,\varepsilon}^3
		\overset{\varepsilon\downarrow 0}{\longrightarrow} 
		\rho\, (\nabla_{a_i}a_i)(1)=0,
	\end{split}
\end{equation}
which take place in $L^q\left([0,T]; L^2(\Omega\times M)\right)$, 
for any $q\in [1,\infty)$.

We multiply the remainder identity \eqref{eq:the-remainder-2} by $\psi$ 
and integrate over $M$, and then analyze the convergence (in $\omega,t$)
of the resulting terms separately. Consider first 
the term $\action{\mathcal{H}_1(\omega,t)}{\psi}$, 
$\mathcal{H}_1=- \sum_i\int_0^tF'(\rho_\varepsilon) 
\,r_{\varepsilon,i}\, dW^i(s)$. 
Since $F'$ is bounded,
$$
\abs{
\bigl\langle 
F'(\rho_\varepsilon(\omega,s)) 
\,r_{\varepsilon,i}(\omega,s),\psi
\bigr\rangle}
\leq \norm{\psi}_{L^2(M)}\norm{F'}_\infty
\norm{r_{\varepsilon,i}(\omega,s)}_{L^2(M)}.
$$
Therefore, in view of \eqref{eq:main-convergences} with $q=2$, 
$\bigl\langle F'(\rho_\varepsilon)\,r_{\varepsilon,i},
\psi\bigr \rangle \overset{\varepsilon\downarrow 0}{\longrightarrow} 0$ 
in $L^2(\Omega \times [0,T])$. 
Consequently, by the It\^o isometry, $\sup\limits_{t\in [0,T]}
\Bigl(\EE \action{\mathcal{H}_1}{\psi}^2\Big)^{1/2}
\overset{\varepsilon\downarrow 0}{\longrightarrow} 0$; 
whence
$$
\action{\mathcal{H}_1}{\psi} 
\overset{\varepsilon\downarrow 0}{\longrightarrow} 
0 \quad \text{in $L^2(\Omega \times [0,T])$}.
$$
The term $\action{\mathcal{H}_{7}(\omega,t)}{\psi}$, 
where $\mathcal{H}_7=\sum_i\int_0^t F'(\rho_\varepsilon)
\, A_{i,\varepsilon}^1 \, dW^i(s)$, is treated in 
the same way, yielding the convergence  
$$
\action{\mathcal{H}_7}{\psi} 
\overset{\varepsilon\downarrow 0}{\longrightarrow} 
0 \quad \text{in $L^2(\Omega \times [0,T])$}.
$$

Again thanks to \eqref{eq:main-convergences}, 
$\bigl\langle F'(\rho_\varepsilon)
\,\tilde{r}_{\varepsilon,i},\psi\bigr\rangle 
\overset{\varepsilon\downarrow 0}{\longrightarrow} 0$ 
in $L^2(\Omega \times [0,T])$. It is 
therefore straightforward to deduce
$$
\action{\mathcal{H}_3}{\psi} 
\overset{\varepsilon\downarrow 0}{\longrightarrow} 
0 \quad \text{in $L^2(\Omega \times [0,T])$},
$$
recalling that $\mathcal{H}_3 :=-\frac12\sum_i\int_0^t 
F'(\rho_\varepsilon)\, \tilde{r}_{\varepsilon,i}\,ds$.

The two terms $\action{\mathcal{H}_{2+6,2}(\omega,t)}{\psi}$, 
and $\action{\mathcal{H}_{9+11,1}(\omega,t)}{\psi}$, 
where we recall that $\mathcal{H}_{2+6,2}=\frac12 \sum_i\int_0^t 
F'(\rho_\varepsilon)\, \bar{r}_{\varepsilon,i}\,ds$ and 
$\mathcal{H}_{9+11,1}=-\sum_i\int_0^t
F'(\rho_\varepsilon)\, r^\ast_{\varepsilon,i}\,ds$, 
are dealt with exactly in the same way, supplying
$$
\action{\mathcal{H}_{2+6,2}}{\psi},\,
\action{\mathcal{H}_{9+11,1}}{\psi} 
\overset{\varepsilon\downarrow 0}{\longrightarrow} 
0 \quad \text{in $L^2(\Omega \times [0,T])$},
$$

We continue with  $\action{\mathcal{H}_4(\omega,t)}{\psi}$, 
$\mathcal{H}_4=\frac12\sum_i \int_0^t F''(\rho_\varepsilon) 
\, r_{\varepsilon,i}^2 \,ds$. Since $F''$ is bounded, 
$$
\abs{\left\langle F''(\rho_\varepsilon(\omega,s))
\, \left(r_{\varepsilon,i}(\omega,s)\right)^2,\psi \right\rangle}
\leq \norm{\psi}_{L^\infty(M)}\norm{F''}_\infty
\norm{r_{\varepsilon,i}(\omega,s)}^2_{L^2(M)},
$$
and so, once again using \eqref{eq:main-convergences}, 
$\left\langle F''(\rho_\varepsilon) \,r^2_{\varepsilon,i},
\psi\right\rangle \overset{\varepsilon\downarrow 0}{\longrightarrow} 0$ 
in $L^1(\Omega \times [0,T])$. As a result
$$
\action{\mathcal{H}_4}{\psi} 
\overset{\varepsilon\downarrow 0}{\longrightarrow} 
0 \quad \text{in $L^1(\Omega \times [0,T])$}.
$$
Similarly, with 
$\mathcal{H}_{10}=\frac12  \sum_i\int_0^tF''(\rho_\varepsilon)
\left(A_{i,\varepsilon}^1\right)^2\, ds$ 
and \eqref{eq:main-convergences},
$$
\action{\mathcal{H}_{10}}{\psi} 
\overset{\varepsilon\downarrow 0}{\longrightarrow} 
0 \quad \text{in $L^1(\Omega \times [0,T])$}.
$$

Let us analyze $\action{\mathcal{H}_{5,1}(\omega,t)}{\psi}$, 
$\mathcal{H}_{5,1}=\sum_i\int_0^t 
\rho_\varepsilon F''(\rho_\varepsilon)
\, \Divh a_i\,r_{\varepsilon,i} \,ds$. We have
\begin{align*}
	& \abs{\bigl\langle \rho_\varepsilon(\omega,s)
	F''(\rho_\varepsilon(\omega,s))\, 
	\Divh a_i\,r_{\varepsilon,i}(\omega,s),
	\psi \bigr\rangle}
	\\ & \qquad 
	\leq \norm{\psi}_{L^\infty(M)}\norm{F''}_\infty
	\norm{\Div a_i}_{L^\infty(M)}
	\norm{\rho_\varepsilon(\omega,s)}_{L^2(M)}
	\norm{r_{\varepsilon,i}(\omega,s)}_{L^2(M)}.
\end{align*}
Therefore, by the Cauchy-Schwarz inequality,
\begin{align*}
	& \iint\limits_{\Omega \times [0,T]}
	\abs{
	\bigl\langle 
	\rho_\varepsilon F''(\rho_\varepsilon)
	\Divh a_i \,r_{\varepsilon,i},\psi
	\bigr\rangle}\, ds\, d\P
	\\ & \qquad
	\lesssim_{\psi} \norm{\rho_\varepsilon}_{L^2(\Omega \times [0,T]\times M)}
	\norm{r_{\varepsilon,i}}_{L^2(\Omega \times [0,T]\times M)}.
\end{align*}
Accordingly, again making use of 
\eqref{eq:main-convergences} with $q=2$,
$$
\action{\mathcal{H}_{5,1}}{\psi} 
\overset{\varepsilon\downarrow 0}{\longrightarrow} 
0 \quad \text{in $L^1(\Omega \times [0,T])$}.
$$
Likewise, for 
$\mathcal{H}_{9+11,4}
=-\sum_i\int_0^t F''(\rho_\varepsilon)
A_{i,\varepsilon}^1\bigl (r_{\varepsilon,i} 
+\rho_\varepsilon\, \Divh a_i\bigr) \,ds$, we find that
$$
\action{\mathcal{H}_{9+11,4}}{\psi} 
\overset{\varepsilon\downarrow 0}{\longrightarrow} 
0 \quad \text{in $L^1(\Omega \times [0,T])$}.
$$

Next we deal with the term 
$\action{\mathcal{H}_{5,3}(\omega,t)}{\psi}$, where 
$\mathcal{H}_{5,3}=\sum_i\int_0^t 
a_i\bigl(F'(\rho_\varepsilon) \,r_{\varepsilon,i}\bigr) \,ds$. 
Integration by parts yields 
$$
\action{\mathcal{H}_{5,3}(\omega,t)}{\psi}
=-\sum_{i=1}^N\int_0^t \int_M F'(\rho_\varepsilon(s,x))
\,r_{\varepsilon,i}(s,x) \Div \left(\psi \, a_i\right)
\,dV_h(x) \,ds.
$$
Since $\Div\left(\psi \,a_i\right)\in L^\infty(M)$, 
we conclude as before that 
$$
\action{\mathcal{H}_{5,3}}{\psi} 
\overset{\varepsilon\downarrow 0}{\longrightarrow} 
0 \quad \text{in $L^2(\Omega \times [0,T])$}.
$$
We treat $\action{\mathcal{H}_{9+11,3}(\omega,t)}{\psi}$, 
where $\mathcal{H}_{9+11,3}=-\sum_i\int_0^t 
a_i\bigl(F'(\rho_\varepsilon)
A_{i,\varepsilon}^1\bigr)\,ds$, in a similar way, writing
$$
\action{\mathcal{H}_{9+11,3}(\omega,t)}{\psi}
= \sum_{i=1}^N\int_0^t \int_MF'(\rho_\varepsilon(s,x))A_{i,\varepsilon}^1(s,x)
\Div(\psi\,a_i)\,dV_h(x)\,ds
$$
and also in this case hammering out the convergence
$$
\action{\mathcal{H}_{9+11,3}}{\psi} 
\overset{\varepsilon\downarrow 0}{\longrightarrow} 
0 \quad \text{in $L^2(\Omega \times [0,T])$}.
$$

Putting into service once more the boundedness of $F'$ 
and the convergences \eqref{eq:main-convergences}, we infer 
$$
\action{\mathcal{H}_8}{\psi}, \, 
\action{\mathcal{H}_{9+11,2}}{\psi}  
\overset{\varepsilon\downarrow 0}{\longrightarrow} 
0 \quad \text{in $L^2(\Omega \times [0,T])$},
$$
where $\mathcal{H}_8 = \frac12\sum_i
\int_0^t F'(\rho_\varepsilon)
\bigl[A_{i,\varepsilon}^2+A_{i,\varepsilon}^3\bigr]\,ds$, 
$\mathcal{H}_{9+11,2}=-\sum_i\int_0^tF'(\rho_\varepsilon)
\, A_{i,\varepsilon}^1 \Divh a_i\,ds$.

Finally, we deal with $\action{\mathcal{H}_{9+11,3}(\omega,t)}{\psi}$, 
where $\mathcal{H}_{9+11,3}=\sum_i\int_0^t 
F'(\rho_\varepsilon)\sum_{\kappa} G_{\kappa,\varepsilon,i}\,ds$. 
We clearly have
$$
\action{\mathcal{H}_{9+11,3}(\omega,t)}{\psi}
=\sum_{i=1}^N\sum_{\kappa\in\mathcal{A}}\int_0^t
\bigl \langle F'(\rho_\varepsilon(s))
\, G_{\kappa,\varepsilon,i}(s),
\psi \bigr\rangle \,ds,
$$
along with the following bounds on the integrands:
$$
\abs{\bigl\langle F'(\rho_\varepsilon(\omega,s))\, 
G_{\kappa,\varepsilon,i}(\omega,s),\psi \bigr\rangle}
\leq \norm{F'}_\infty\norm{\psi}_{L^2(M)}
\norm{G_{\kappa,\varepsilon,i}(\omega,s)}_{L^2(M)}.
$$
Recalling that $G_{\kappa,\varepsilon,i}\to 0$ in 
$L^2$, cf.~\eqref{eq:main-convergences}, we obtain
$$
\action{\mathcal{H}_{9+11,3}}{\psi}  
\overset{\varepsilon\downarrow 0}{\longrightarrow} 
0 \quad \text{in $L^2(\Omega \times [0,T])$}.
$$

Summarizing our findings, 
$\bigl\langle \mathcal{R}_\varepsilon,\psi\bigr\rangle
\overset{\varepsilon\downarrow 0}{\longrightarrow} 0$ 
in $L^1(\Omega \times [0,T])$.
\end{proof}

\begin{proposition}[limit SPDE, $u\equiv 0$]\label{prop:convergence-SPDE}
The function $F(\rho)$ satisfies the weak (in $x$) formulation 
\eqref{eq:main-result-weak-form}, $\P$-a.s., for all $t\in[0,T]$, 
for each $\psi\in C^\infty(M)$. 
\end{proposition}

\begin{proof}
We multiply \eqref{eq:Ito-reg-eqn-4} by $\psi\in C^\infty(M)$ 
and integrate over $M$. Let us write the resulting 
identity symbolically as 
$$
\mathcal{J}_{1,\varepsilon}
-\mathcal{J}_{2,\varepsilon}
+\mathcal{J}_{3,\varepsilon}
+\mathcal{J}_{4,\varepsilon}
=\mathcal{J}_{5,\varepsilon}
+\mathcal{J}_{6,\varepsilon}
+\mathcal{J}_{7,\varepsilon}
+\mathcal{J}_{8,\varepsilon}
+\bigl\langle \mathcal{R}_\varepsilon,\psi\bigr\rangle.
$$ 
In what follows, we analyze separately the terms $\mathcal{J}_{\ell,\varepsilon}
=\mathcal{J}_{\ell,\varepsilon}(\omega,t)$, for $\ell=1,\ldots,7$, 
referring to Proposition \ref{prop:the-remainder-vanishes} 
for $\bigl\langle \mathcal{R}_\varepsilon,\psi\bigr\rangle$. 

The term $\mathcal{J}_{1,\varepsilon}(\omega,t) 
=\bigl\langle F(\rho_\varepsilon(t)),\psi\bigr\rangle$ 
is easily handled. Indeed, we have
$$
\abs{\bigl\langle F(\rho_\varepsilon(\omega,t))
-F(\rho(\omega,t)),\psi\bigr\rangle}
\leq \norm{\psi}_{L^\infty(M)}\norm{F'}_\infty 
\norm{\rho_\varepsilon(\omega,t) - \rho(\omega,t)}_{L^2(M)},
$$
and hence, by \eqref{eq:main-convergences},
$\mathcal{J}_{1,\varepsilon}
\overset{\varepsilon\downarrow 0}{\longrightarrow} 
\langle F(\rho),\psi\rangle$ in $L^2(\Omega \times [0,T])$. 
Similarly, we have $\mathcal{J}_{2,\varepsilon}(\omega)
=\bigl\langle F(\rho_{0,\varepsilon}),\psi\bigr\rangle
\overset{\varepsilon\downarrow 0}{\longrightarrow} 
\bigl \langle F(\rho_0),\psi\bigr\rangle$ 
in $L^2(\Omega \times [0,T])$. 

Integration by parts yields
$$
\mathcal{J}_{5,\varepsilon}(\omega,t)
=\action{\frac12 \sum_{i=1}^N\int_0^t
\Lambda_i\left(F(\rho_\varepsilon(s))\right)\,ds}{\psi} 
=\frac12 \sum_{i=1}^N\int_0^t
\bigl\langle 
F(\rho_\varepsilon(s)),a_i(a_i(\psi))
\bigr\rangle \,ds.
$$
Noting that
\begin{align*}
	&\abs{
	\bigl\langle F(\rho_\varepsilon(\omega,s))-F(\rho(\omega,s)),
	a_i(a_i(\psi))\bigr\rangle}\\ & \qquad 
	\leq \norm{a_i(a_i(\psi))}_{L^2(M)}\norm{F'}_\infty
	\norm{\rho_\varepsilon(\omega,s)-\rho(\omega,s)}_{L^2(M)},
\end{align*}
we use again \eqref{eq:main-convergences} to infer  
$$
\mathcal{J}_{5,\varepsilon}
\overset{\varepsilon\downarrow 0}{\longrightarrow} 
\frac12 \sum_{i=1}^N\int_0^\cdot 
\bigl\langle F(\rho(s)), 
a_i(a_i(\psi))\bigr\rangle \,ds 
\quad \text{in $L^2(\Omega \times [0,T])$}.
$$

Let us analyze the stochastic integral
$$
\mathcal{J}_{3,\varepsilon}(\omega,t)
=-\sum_{i=1}^N\int_0^t
\bigl\langle F(\rho_\varepsilon(s)),
a_i(\psi)\bigr\rangle\, dW^i(s),
$$
where integration by parts was used to 
obtain the right-hand side. Making use of the 
estimates ($I=1,\ldots,N$)
\begin{align*}
	&\abs{\bigl\langle F(\rho_\varepsilon(\omega,s))
	-F(\rho(\omega,s)),a_i(\psi)\bigr\rangle}
	\\ & \qquad \leq \norm{a_i(\psi)}_{L^2(M)}\norm{F'}_\infty 
	\norm{\rho_\varepsilon(\omega,s)-\rho(\omega,s)}_{L^2(M)},
\end{align*}
the It\^o isometry, and \eqref{eq:main-convergences}, we obtain
$$
\mathcal{J}_{3,\varepsilon}
\overset{\varepsilon\downarrow 0}{\longrightarrow} 
\sum_{i=1}^N\int_0^\cdot 
\bigl\langle F(\rho(s)),a_i(\psi)\bigr\rangle\, dW^i(s)
\quad \text{in $L^2(\Omega \times [0,T])$}.
$$
The other term involving a stochastic integral is dealt with 
in a similar fashion. Indeed, recalling \eqref{eq:GF-def}, 
$\mathcal{J}_{4,\varepsilon}=\mathcal{J}_{4_1,\varepsilon}
+\mathcal{J}_{4_2,\varepsilon}$, where 
\begin{align*}
	\mathcal{J}_{4_1,\varepsilon}(\omega,t)
	& :=\sum_{i=1}^N\int_0^t
	\bigl\langle \rho_\varepsilon(s) F'(\rho_\varepsilon(s)) 
	\Divh a_i, \psi \bigr\rangle \, dW^i(s), 
	\\ \mathcal{J}_{4_2,\varepsilon}(\omega,t)
	& :=- \sum_{i=1}^N\int_0^t
	\langle F(\rho_\varepsilon(s)) \Divh a_i,\psi\rangle\, dW^i(s).
\end{align*}
As before, 
$$
\mathcal{J}_{6_4,\varepsilon}
\overset{\varepsilon\downarrow 0}{\longrightarrow} 
\sum_{i=1}^N\int_0^\cdot 
\bigl \langle F(\rho(s))\Divh a_i,\psi\bigr\rangle\, dW^i(s)
\quad \text{in $L^2(\Omega \times [0,T])$}.
$$
Regarding $\mathcal{J}_{4_1,\varepsilon}$, note that 
\begin{align*}
	&\abs{\bigl\langle \rho_\varepsilon(\omega,s) 
	F'(\rho_\varepsilon(\omega,s)) 
	- \rho(\omega,s)F'(\rho(\omega,s)), 
	\Divh a_i\,\psi \bigr\rangle}
	\\ & \qquad 
	\leq \norm{\psi \Div a_i}_{L^2(M)} 
	\norm{\rho_\varepsilon(\omega,s)F'(\rho_\varepsilon(\omega,s))
	-\rho(\omega,s)F'(\rho(\omega,s))}_{L^2(M)}
\end{align*}
and, invoking Lemma \ref{lem:product-limit} (appendix),
$$
\norm{\rho_\varepsilon F'(\rho_\varepsilon)-\rho F'(\rho)}_{L^2(M)}
\overset{\varepsilon\downarrow 0}{\longrightarrow} 
\quad \text{in $L^2(\Omega \times [0,T])$}.
$$
Hence, appealing once more to the 
It\^o isometry,
$$
\mathcal{J}_{4_1,\varepsilon}
\overset{\varepsilon\downarrow 0}{\longrightarrow} 
\sum_{i=1}^N\int_0^\cdot \bigl\langle \rho(s) F'(\rho(s)) \Divh a_i,
\psi \bigr\rangle\, dW^i(s)
\quad \text{in $L^2(\Omega \times [0,T])$.}
$$

By a similar reasoning process, we compute easily the limits
\begin{align*}
	\mathcal{J}_{6,\varepsilon}(\omega,t)
	&=\action{ \frac12\sum_{i=1}^N\int_0^\cdot 
	G_F(\rho_\varepsilon(s))\,\Lambda_i(1)\, ds}{\psi} \nonumber
	\\ & \overset{\varepsilon\downarrow 0}{\longrightarrow}
	\frac12\sum_{i=1}^N\int_0^\cdot \langle G_F(\rho(s))
	\,\Lambda_i(1),\psi\rangle\, ds
	\;\;\; \text{in } L^2(\Omega \times [0,T]),
\end{align*}
and, after an integration by parts, 
\begin{align*}
	\mathcal{J}_{8,\varepsilon}(\omega,t)
	& =\action{\sum_{i=1}^N\int_0^t 
	\bar{a}_i\bigl(G_F(\rho_\varepsilon(s))\bigr)\, ds}{\psi} \nonumber
	\\ & \overset{\varepsilon\downarrow 0}{\longrightarrow}
	-\sum_{i=1}^N\int_0^\cdot
	\bigl \langle 
	G_F(\rho(s)), \Divh\left(\psi\bar{a}_i\right)
	\bigr\rangle \, ds
	\quad \text{in $L^2(\Omega \times [0,T])$}.
\end{align*}

It remains to deal with $\mathcal{J}_{7,\varepsilon}(\omega,t)=
\action{\frac12\sum_i \int_0^t F''(\rho_\varepsilon)
\,\bigl(\rho_\varepsilon \Divh a_i\bigr)^2\, ds}{\psi}$. 
Paying attention to the estimates ($i=1,\ldots,N$)
\begin{align*}
	&\abs{\left\langle 
	\left(\rho(\omega,s)\right)^2 
	F''(\rho_\varepsilon(\omega,s)) 
	- \left(\rho(\omega,s)\right)^2 F''(\rho(\omega,s)),
	\left(\Divh a_i\right)^2\,\psi\right\rangle}
	\\ & \quad 
	\leq \norm{\psi (\Divh a_i)^2}_{L^\infty(M)} 
	\norm{\left(\rho_\varepsilon(\omega,s)\right)^2
	F''(\rho_\varepsilon(\omega,s)) 
	- \left(\rho^2(\omega,s)\right)^2 
	F''(\rho(\omega,s))}_{L^1(M)}.
\end{align*}
Since $\rho^2_\varepsilon  
\overset{\varepsilon\downarrow 0}{\longrightarrow}
\rho^2$ in $L^1(\Omega\times [0,T]\times M)$, cf.~\eqref{eq:main-convergences}, 
and $F''\in C_b(\R)$, we can again invoke 
Lemma \ref{lem:product-limit} to arrive at 
$$
\mathcal{J}_{7,\varepsilon}
\overset{\varepsilon\downarrow 0}{\longrightarrow}
\frac12\sum_{i=1}^N\int_0^\cdot 
\left\langle F''(\rho(s)) \bigl(\rho(s)\Divh a_i \bigr)^2, \psi \right \rangle \, ds
\quad \text{in $L^1(\Omega \times [0,T])$.}
$$

In view of the established convergences, it is clear that $F(\rho)$ 
satisfies the weak formulation \eqref{eq:main-result-weak-form} 
(with $u\equiv 0$), $\P$-a.s., for a.e.~$t\in [0,T]$. To improve this to 
\textit{all} times $t\in [0,T]$ note that the right-hand side of 
\eqref{eq:main-result-weak-form} defines a continuous stochastic process. 
Therefore, $(\omega,t)\mapsto \int_M F(\rho(t))\psi \,dV_h$ 
admits a continuous modification. This concludes the proof of the proposition.
\end{proof}

As of now, we have proved our main result 
(Theorem \ref{thm:main-result})
under the additional assumption that $u\equiv 0$.


\subsection{The general case $u\not\equiv 0$}
Let us adapt the prior proof to the general 
case. First, \eqref{eq:Ito-reg-eqn} 
continues to hold provided we add to 
the right-hand side the terms
$$
-\int_0^t F'(\rho_\varepsilon(s,x))
\Divh \bigl(\rho(s) u(s)\bigr)_\varepsilon(x)\, ds, 
\quad 
\widetilde{\mathcal{J}_{A}}
:=\int_0^t F'(\rho_\varepsilon(s,x))\, 
A_{u,\varepsilon}(s,x)\, ds,
$$
where the first term is (by now) easily 
seen to be equal to 
\begin{align*}
	\widetilde{\mathcal{J}_u} & :=-\int_0^t 
	\Divh \bigl (F(\rho_\varepsilon(s,x))\, u(s) \bigr)\,ds
	-\int_0^t G_F(\rho_\varepsilon(s,x)) \Divh u(s)\, ds
	\\ & \qquad 
	-\int_0^t F'(\rho_\varepsilon(s,x)) 
	\, r_{\varepsilon,u}(s,x)\, ds,	
\end{align*}
where $G_F$ is defined in \eqref{eq:GF-def} and 
the remainder $r_{\varepsilon,u}$ is 
defined in \eqref{eq:commutator-u}. In other words, 
the equation \eqref{eq:Ito-reg-eqn-4} for $F(\rho_\varepsilon)$ 
continues to hold with $-\widetilde{\mathcal{J}_u}$ and 
$-\widetilde{\mathcal{J}_A}$ added to the left-hand side of the equality sign.  

To conclude proof of Theorem \ref{thm:main-result} (in the general case, 
$u\not \equiv 0$), we need strong convergence results for the following 
terms related to $\widetilde{\mathcal{J}_u}$ 
and $\widetilde{\mathcal{J}_A}$:
\begin{align*}
	& \mathcal{J}_{1,\varepsilon}(\omega,t)
	:=\action{\int_0^t 
	\Divh \bigl (F(\rho_\varepsilon(s))\, u(s)\bigr)\,ds}{\psi},
	\\ 
	& \mathcal{J}_{2,\varepsilon}(\omega,t)
	:=\action{\int_0^t G_F(\rho_\varepsilon(s))\Divh u(s)\, ds}{\psi},
	\\ 
	& \mathcal{J}_{3,\varepsilon}(\omega,t)
	:= \action{ \int_0^t F'(\rho_\varepsilon(s))
	\, A_{u,\varepsilon}(s)\, ds}{\psi},
	\\ 
	& \mathcal{J}_{4,\varepsilon}(\omega,t)
	:= \action{\int_0^t F'(\rho_\varepsilon(s)) 
	\,r_{\varepsilon,u}(s)\, ds}{\psi},
\end{align*}
for any $\psi\in C^\infty(M)$ and 
$(\omega,t)\in \Omega\times [0,T]$.

\begin{lemma}[convergence of terms related to $u$]
Fix $\psi\in C^\infty(M)$. Then
\begin{align}
	& \mathcal{J}_{1,\varepsilon}
	\overset{\varepsilon\downarrow 0}{\longrightarrow}
	-\int_0^\cdot \bigl\langle F(\rho(s)),u(s)(\psi)
	\bigr\rangle \, ds 
	\quad \text{in $L^1(\Omega \times [0,T])$},
	\label{eq:u-claim1}
	\\ & 
	\mathcal{J}_{2,\varepsilon}
	\overset{\varepsilon\downarrow 0}{\longrightarrow}
	\int_0^\cdot \bigl \langle G_F(\rho(s)) 
	\Divh u(s),\psi\bigr \rangle\, ds
	\quad \text{in $L^1(\Omega \times [0,T])$},
	\label{eq:u-claim2}
	\\ & 
	\mathcal{J}_{3,\varepsilon}
	\overset{\varepsilon\downarrow 0}{\longrightarrow} 0
	\quad \text{in $L^1(\Omega \times [0,T])$},
	\quad
	\mathcal{J}_{4,\varepsilon}
	\overset{\varepsilon\downarrow 0}{\longrightarrow}0
	\quad \text{in $L^1(\Omega \times [0,T])$}.
	\label{eq:u-claim34}
\end{align}
\end{lemma}

\begin{proof}
Clearly,
$\mathcal{J}_{1,\varepsilon}
=-\int_0^t \bigl\langle 
F(\rho_\varepsilon(s)),u(s)(\psi) \bigr\rangle\, ds$ 
and
\begin{align*}
	& \abs{\bigl \langle F(\rho_\varepsilon(\omega,s)) 
	- F(\rho(\omega,s)), u(s)(\psi) \bigr\rangle}
	\\ & \quad 
	\lesssim_\psi  
	\norm{F'}_\infty\norm{u(s)}_{L^2\left(\mathcal{T}^0_1(M)\right)}
	\norm{\rho_\varepsilon(\omega,s) - \rho(\omega,s)}_{L^2(M)}.
\end{align*}
Recalling \eqref{eq:main-convergences}, the latter estimate 
implies that $\abs{\bigl\langle F(\rho_\varepsilon)-F(\rho),u(\psi)\bigr\rangle}
\overset{\varepsilon\downarrow 0}{\longrightarrow}0$ 
in $L^1(\Omega \times [0,T])$. Thus, invoking 
Lemma \ref{lem:classical-result}, the claim 	\eqref{eq:u-claim1} follows.

Similarly, looking back on \eqref{eq:GF-def},  
\begin{align*}
	I_1(\omega,s) & :=\abs{\bigl\langle \rho_\varepsilon(\omega,s) 
	F'(\rho_\varepsilon(\omega,s)) 
	- \rho(\omega,s)F'(\rho(\omega,s)),\psi\Divh u(s)\bigr\rangle}
	\\ & 
	\lesssim_{\psi} \norm{\Divh u(s)}_{L^2(M)}
	\norm{\rho_\varepsilon(\omega,s) F'(\rho_\varepsilon(\omega,s)) 
	-\rho(\omega,s)F(\rho(\omega,s))}_{L^2(M)},
	\\ 
	I_2(\omega,s) & 
	:=\abs{\bigl\langle F(\rho_\varepsilon(\omega,s))
	-F(\rho(\omega,s)),\psi\Divh u(s)\bigr\rangle}
	\\ & 
	\lesssim_{\psi} \norm{\Divh u(s)}_{L^2(M)}
	\norm{\rho_\varepsilon(\omega,s)-\rho(\omega,s)}_{L^2(M)}.
\end{align*}
Recalling that $\rho_\varepsilon \in L^\infty_tL^2_{\omega,x}$ 
and $\Divh u\in L^1_tL^2_x$, cf.~\eqref{eq:u-ass-1}, we thus obtain
\begin{align*}
	& \iint\limits_{\Omega \times [0,T]} 
	I_1(\omega,s)\, ds\,d\P 
	\lesssim_{\psi} \int_0^T \tilde{I}_1(s) \, ds,
	\\ & \tilde{I}_1(s):=
	\norm{\rho_\varepsilon(s) F'(\rho_\varepsilon(s))
	-\rho(s)F'(\rho(s))}_{L^2(\Omega\times M)}
	\norm{\Divh u(s)}_{L^2(M)},
\end{align*}
and
\begin{align*} 
	& \iint\limits_{\Omega \times [0,T]} 
	I_2(\omega,s)\, ds\,d\P
	\lesssim_{\psi} \int_0^T \tilde{I}_2(s) \, ds,
	\\ & \tilde{I}_2(s):=\norm{ F(\rho_\varepsilon(s))
	-F(\rho(s))}_{L^2(\Omega\times M)}
	\norm{\Divh u(s)}_{L^2(M)}.
\end{align*}
The functions $\norm{\rho_\varepsilon F'(\rho_\varepsilon)
-\rho F'(\rho)}_{L^2(\Omega\times M)}$ and $\norm{F(\rho_\varepsilon)
-F(\rho)}_{L^2(\Omega\times M)}$ converge to zero in 
$L^q([0,T])$ for any $q\in [1,\infty)$, by Lemma \ref{lem:product-limit-2}, and 
also a.e.~on $[0,T]$ (up to subsequences). Moreover, for a.e. $s\in [0,T]$,
\begin{align*}
	& \tilde{I}_1(s),\, \tilde{I}_2(s)
	\lesssim \norm{\rho }_{L^2(\Omega\times M)}
	\norm{\Divh u}_{L^2(M)}\in L^1([0,T]).
\end{align*}
Hence, by the dominated convergence theorem, 
$\int_0^T\int_{\Omega} I_1(\omega,s)\, ds\,d\P 
\overset{\varepsilon\downarrow 0}{\longrightarrow} 0$ and 
$\int_0^T\int_{\Omega} I_2(\omega,s)\, ds\,d\P 
\overset{\varepsilon\downarrow 0}{\longrightarrow} 0$.
In combination with Lemma \ref{lem:classical-result}, 
this gives \eqref{eq:u-claim2}.

Finally, since $A_{u,\varepsilon}
\overset{\varepsilon\downarrow 0}{\longrightarrow}0$ 
in $L^1(\Omega\times [0,T]\times M)$, cf.~Lemma \ref{lem:prop-A-kappa-eps-u}, 
and $r_{\varepsilon,u}\overset{\varepsilon\downarrow 0}{\longrightarrow}0$ 
in $L^1\left([0,T];L^2(\Omega; L^1(M))\right)$, 
cf.~Lemma \ref{lem:com-u}, we easily arrive at \eqref{eq:u-claim34}.
\end{proof}

\section{Uniqueness and a priori estimate}\label{sec:uniqueness}


\subsection{Uniqueness, proof of Corollary \ref{cor:uniq-result}}
The aim is to prove Corollary \ref{cor:uniq-result}, 
relying on the renormalization property of $L^2$ 
weak solutions (Theorem \ref{thm:main-result}). 
The renormalization property holds for 
bounded nonlinearities $F:\R\to\R$. To handle $F(\xi)=\xi^2$ 
we employ an approximation (truncation) procedure. 

To this end, pick any increasing function 
$\chi\in C^\infty\bigl([0,\infty)\bigr)$ such that
$\chi(\xi)=\xi$ for $\xi\in [0,1]$, $\chi(\xi)=2$ 
for $\xi\geq 2$, $\chi(\xi)\in (1,2)$ for $\xi\in (1,2)$, 
and $A_0:=\sup_{\xi\geq 0}\chi'(\xi)>1$. 
Set $A_1:=\sup_{\xi\geq 0} \abs{\chi''(\xi)}$. 
We need the rescaled function 
$\chi_\mu(\xi)=\mu \,\chi(\xi/\mu)$, for $\mu>0$. 
The relevant approximation of $F(\xi)=\xi^2$ is
$$
F_\mu(\xi):=\chi_\mu\left(\xi^2\right), 
\quad \xi\in\R, \quad \mu>0.
$$
Some tedious computations will reveal that
\begin{equation}\label{eq:Fmu-bounds}
	\begin{split}
		& F_\mu\in C^\infty(\R), 
		\quad 
		\lim_{\mu\to\infty} F_\mu(\xi)=\xi^2,
		\quad
		\sup_{\xi\in\R} F_\mu(\xi)\leq 2\mu, 
		\quad \sup_{\mu>0} F_\mu(\xi)\leq 2\xi^2,
		\\ & 
		\sup_{\xi\in\R}\abs{F_\mu'(\xi)}\leq 2\sqrt{2}A_0\sqrt{\mu}, 
		\quad 
		\sup_{\mu>0} \abs{F_\mu'(\xi)}\leq 2\sqrt{2}A_0\abs{\xi},
		\quad 
		\lim_{\mu\to\infty} F_\mu'(\xi)=2\xi,
		\\ &
		\lim_{\mu\to\infty} F_\mu''(\xi)=2,
		\quad
		\abs{F_\mu''(\xi)}\leq 8A_1 + 2A_0, 
		\quad \text{for $\xi\in \R$, $\mu>0$}.
	\end{split}
\end{equation}
Furthermore, the function $G_{F_\mu}(\xi)=\xi F_\mu'(\xi)-F_\mu(\xi)$, 
cf.~\eqref{eq:GF-def}, satisfies
\begin{equation}\label{eq:Gmu-bounds}
	\begin{split}
		& \sup_{\xi\in\R}\abs{G_{F_\mu}(\xi)}
		\leq (4A_0 + 2)\mu, 
		\quad \sup_{\mu>0}\abs{G_{F_\mu}(\xi)}
		\leq 2(\sqrt{2}A_0 + 1) \xi^2,
		\\ & \qquad 
		\text{and} \quad 
		\lim_{\mu\to\infty}G_{F_\mu}(\xi)=\xi^2,
		\quad \text{for $\xi\in \R$, $\mu>0$}.
	\end{split}
\end{equation}
Finally, to prove Corollary \ref{cor:apriori-est}, 
we will also make use of the bounds
\begin{equation}\label{eq:Fmu-Gmu-prop}
	\abs{G_{F_{\mu}}(\xi)} \leq C_\chi F_{\mu}(\xi), 
	\quad
	\abs{\xi^2F_\mu''(\xi)} \leq C_\chi 
	\begin{cases}
		F_{\mu}(\xi), & \abs{\xi} \leq \mu\\
		\xi^2, & \abs{\xi} \in \left[\sqrt{\mu},\sqrt{2\mu}\right]\\
		0, & \abs{\xi} > \sqrt{2\mu},
	\end{cases}
\end{equation} 
for some constant $C_\chi$ that does not depend on $\mu$.

Consider weak $L^2$-solution $\rho$ of \eqref{eq:target} 
with initial datum $\rho_0\in L^2(M)$. Referring to
Theorem \ref{thm:main-result}, taking $F=F_\mu$ and $\psi\equiv 1$ 
in \eqref{eq:main-result-weak-form} supplies the equation
\begin{equation}\label{eq:weak-form-Fmu}
	\begin{split}
		&\int_M F_\mu(\rho(t))\, dV_h 
		= \int_M F_\mu(\rho_0)\, dV_h 
		-\int_0^t\int_M G_{F_\mu}(\rho(s))\Divh u(s) \, dV_h\, ds
		\\ &\qquad 
		-\sum_{i=1}^N\int_0^t\int_M 
		G_{F_\mu}(\rho(s)) \Divh a_i \,dV_h \, dW^i(s)
		\\ & \qquad\qquad 
		-\frac12 \sum_{i=1}^N\int_0^t\int_M 
		\Lambda_i(1)\,G_{F_\mu}(\rho(s))\,dV_h\,ds
		\\ & \qquad\qquad\qquad
		+\frac12\sum_{i=1}^N\int_0^t\int_M 
		F_\mu''(\rho(s)) \bigl(\rho(s)\Divh a_i\bigr)^2\,dV_h\,ds,
	\end{split}
\end{equation}
which holds $\P$-a.s., for all $t\in [0,T]$, and for any finite $\mu>0$. 
Recall that $\Lambda_i(1)$ equals $\Divh \bar{a}_i$ 
and $\bar{a}_i=(\Divh a_i) \, a_i$.

In view of the bounds on $G_{F_\mu}$, cf.~\eqref{eq:Gmu-bounds}, it 
is clear that the stochastic integral in \eqref{eq:weak-form-Fmu}  
is a zero-mean martingale, and taking the expectation leads then to
\begin{equation}\label{eq:weak-form-Fmu-EE}
	\begin{split}
		&\EE\int_M F_\mu(\rho(t))\, dV_h 
		= \EE\int_M F_\mu(\rho_0)\, dV_h 
		-\EE\int_0^t\int_M G_{F_\mu}(\rho(s)) \Divh u(s) \, dV_h\, ds
		\\ & \qquad
		-\frac12 \sum_{i=1}^N\EE\int_0^t\int_M \Lambda_i(1)
		\,G_{F_\mu}(\rho(s))\,dV_h\,ds
		\\ & \qquad \qquad 
		+\frac12\sum_{i=1}^N \EE\int_0^t\int_M 
		F_\mu''(\rho(s))\bigl(\rho(s)\Divh a_i\bigr)^2\,dV_h\,ds,
	\end{split}
\end{equation}
for all $t\in [0,T]$ and any $\mu>0$. In view of the properties 
of $F_\mu$ and $G_{F_\mu}$, cf.~\eqref{eq:Fmu-bounds} and \eqref{eq:Gmu-bounds}, 
the assumption $\Divh u\in L^1([0,T]; L^\infty(M))$, cf.~\eqref{eq:u-ass-2}, 
and $\rho\in L^\infty\left([0,T];L^2(\Omega\times M)\right)$, it is 
straightforward to use the dominated convergence theorem to 
send $\mu\to \infty$ in \eqref{eq:weak-form-Fmu-EE}, 
eventually concluding that
\begin{align*}
	&\EE \norm{\rho(t)}_{L^2(M)}^2 = \EE \norm{\rho_0}_{L^2(M)}^2 
	-\int_0^t \EE\int_M \left(\rho(s)\right)^2\Divh u(s) \,dV_h\,ds
	\\ &\quad
	-\frac12 \sum_{i=1}^N\int_0^t\EE\int_M 
	\Lambda_i(1)\, \left(\rho(s)\right)^2\,dV_h\,ds
	+\sum_{i=1}^N\int_0^t \EE\int_M 
	\left(\rho(s)\right)^2 \bigl(\Divh a_i\bigr)^2\,dV_h\,ds.
\end{align*}
Setting $f(t):=\EE \norm{\rho(t)}_{L^2(M)}^2$ for $t>0$ 
and $f(0):=\norm{\rho_0}_{L^2(M)}^2$, this identity implies
$$
f(t) \leq f(0) + \int_0^t \norm{\Divh u(s)}_{L^\infty(M)} f(s)\, ds 
+\bar C\int_0^t f(s)\,ds, \quad t\in [0,T],
$$
where $\bar C=\sum_{i=1}^N\Bigl( \frac12
\norm{\Lambda_i(1)}_{L^\infty(M)} + \norm{\bigl(\Divh a_i\bigr)^2}_{L^\infty(M)}\Bigr)$.  
By Gr{\"o}nwall's inequality, there is a constant $C$ 
depending on $\bar C$, $T$, and $\norm{\Divh u}_{L^1_tL^\infty_x}$ such that
$$
\EE \norm{\rho(t)}_{L^2(M)}^2 
\leq C\,\EE \norm{\rho_0}_{L^2(M)}^2,
\qquad t\in [0,T].
$$
This, in combination with the linearity of the SPDE \eqref{eq:target},
implies Corollary \ref{cor:uniq-result}.

 
\subsection{A priori estimate, proof of Corollary \ref{cor:apriori-est}}

Define $f_\mu:[0,T]\to \R$ by 
$$
f_\mu(t)=\EE \esssup_{r\in [0,t]} \int_M F_\mu(\rho(r))\, dV_h,  
\quad \text{for $t>0$},
$$
and $f_\mu(0)=\int_M F_\mu(\rho_0)\, dV_h$. 
By the boundedness of $F_\mu$, note that $f_\mu \in L^\infty$ 
(for a forthcoming application of Gr{\"o}nwall's inequality, 
we simply need $f_\mu\in L^1$). 
Set
$$
M_\mu(t):=\sum_{i=1}^N\int_0^t \int_M 
G_{F_\mu}(\rho(s))\Divh a_i\,dV_h \, dW^i(s), 
\quad t\in [0,T].
$$
Kicking off from \eqref{eq:weak-form-Fmu} and utilizing 
\eqref{eq:Fmu-Gmu-prop}, it is not difficult deduce
\begin{align*}
	f_\mu(t) & \lesssim f_\mu(0)
	+\int_0^t  \norm{\Divh u(s)}_{L^\infty(M)}
	\EE  \int_M F_\mu(\rho(s))\, dV_h, \, ds
	\\ & \qquad 
	+ \EE \sup_{r\in [0,t]}\abs{M_\mu(r)}
	+\int_0^t \EE \int_M F_\mu(\rho(\tau))\, dV_h,\,ds 
	+ o(1/\mu),
	\quad t\in [0,T],
\end{align*}
where we have taken advantage of the assumption 
$\rho \in L^\infty_t L^2_{\omega,x}$ to conclude that 
$\rho^2 \in L^1(\Omega\times [0,T]\times M)$ and thus
$$
\iiint\limits_{\seq{\rho^2> \mu}}\rho^2\,dV_h \,ds \, d\P 
=o(1/\mu) \overset{\mu\uparrow \infty}{\longrightarrow} 0.
$$
The constant hidden in ``$\lesssim$" depends 
on  $\max_i\norm{a_i}_{C^2}$ and $\chi$. 

By the Burkholder-Davis-Gundy inequality \eqref{eq:BDG},
\begin{align*}
	& \EE \sup_{r\in [0,t]}\abs{M_\mu(r)}
	\leq 3 \EE \left( \sum_{i=1}^N\int_0^t
	\left( \int_M G_{F_\mu}(\rho(s))\Divh a_i \,dV_h \right)^2
	\, ds \right)^{\frac12}
	\\ & \qquad \overset{\eqref{eq:Fmu-Gmu-prop}}{\leq} C_1
	\EE \left( \int_0^t
	\left(\int_M F_\mu(\rho(s))\,dV_h \right)^2 
	\, ds \right)^{\frac12}
	\\ & \qquad \leq C_1\EE  \left(  
	\esssup_{\tau\in [0,t]} \int_M F_\mu(\rho(\tau))\,dV_h
	\int_0^t \int_M F_\mu(\rho(s))\,dV_h \,ds \right)^{\frac12}
	\\ & \qquad 
	\leq \frac12  \EE  \esssup_{r\in [0,t]} 
	\int_M F_\mu(\rho(r))\,dV_h 
	+ \frac{C_1}{2}\int_0^t \EE  \int_M F_\mu(\rho(s)) \,dV_h \,ds,
\end{align*}
where the constant $C_1$ is independent of $\mu, t$ (but it 
depends on $\max_i\norm{a_i}_{C^1}$). On that account, we obtain
\begin{align*}
	f_\mu(t) & \lesssim f_\mu(0)
	+ \int_0^t  \left(1+\norm{\Divh u(s)}_{L^\infty(M)}\right)f_\mu(s) 
	\, ds+o(1/\mu), \quad t\in [0,T],
\end{align*}
which, in combination with the Gr{\"o}nwall inequality, implies 
$$
f_\mu(t) \leq \exp(C t) f_\mu(0) + o(1/\mu)
\overset{\eqref{eq:Fmu-Gmu-prop}}{\leq}  
2 \exp(C t) \norm{\rho_0}_{L^2(M)}^2 +o(1/\mu),
$$
for some $\mu$-independent constant $C$. Relying 
on the Fatou lemma, the a priori estimate 
\eqref{eq:apriori-est} emerges after sending $\mu\to \infty$.

\appendix


\section{Some technical results}\label{sec:appendix}

We collect here several results that have been used 
throughout the paper (often unannounced), starting with 
a minor generalization of a well-known commutator 
estimate, see \cite[Lemma II.1]{DL89} 
or \cite[Lemma 2.3]{Lions:NSI}. 

We fix a standard Friedrichs mollifier $\phi_\varepsilon$ 
($=\varepsilon^{-d} \phi(x/\varepsilon)$) 
on $\R^d$. In what follows, we will consider functions and vector fields 
defined on an open (bounded or unbounded) 
subset of the Euclidean space $\R^d$.

We say that a triple $(\alpha_1,\alpha_2,\beta)$ is 
\textit{(1)-admissible} if $\alpha_1,\alpha_2\in [1,\infty]$, 
$\frac{1}{\alpha_1} + \frac{1}{\alpha_2} \leq 1$, 
$\frac{1}{\beta}= \frac{1}{\alpha_1} + \frac{1}{\alpha_2}$ if 
$\alpha_1<\infty$ or $\alpha_2<\infty$, and 
$\beta\in [1,\infty)$ is arbitrary if $\alpha_1=\alpha_2=\infty$.

\begin{lemma}[DiPerna-Lions]\label{lemma:gen-diperna-lions}
Let $(Z,\mu)$ be a finite measure space. Suppose
$$
g\in L^{q_1}\left(Z;L^{p_1}_\mathrm{loc}(G)\right), 
\qquad 
V\in L^{q_2}\left(Z;W^{1,p_2}_\mathrm{loc}(G;\R^d)\right),
$$ 
for some (1)-admissible triples $(p_1,p_2,p)$, $(q_1,q_2,q)$. 
Then, for any compact $K\subset G$,
\begin{equation}\label{eq:genDP1}
	\begin{split}
		&\norm{\dive\left(g V\right)_\varepsilon
		-\dive\left(g_\varepsilon V\right)}_{L^q\left(Z;L^p(K)\right)}
		\\ & \quad
		\leq  C \norm{g}_{L^{q_1}\left(Z;L^{p_1}(K)\right)} 
		\norm{V}_{L^{q_2}\left(Z;W^{1,p_2}(K;\R^d)\right)},
	\end{split}
\end{equation}
for some constant $C$ that does not depend on 
$\varepsilon, p, g, V$. Furthermore,
\begin{equation}\label{eq:genDP2}
	\norm{\dive\left(gV\right)_\varepsilon
	-\dive\left(g_\varepsilon V\right)}_{L^q\left(Z;L^p(K)\right)}
	\overset{\varepsilon\downarrow 0}{\longrightarrow}0.
\end{equation}
\end{lemma}

\begin{proof}
For brevity, let us write
$c_\varepsilon(z,x) := \dive\left(gV\right)_\varepsilon(z,x) 
- \dive\left(g_\varepsilon V\right)(z,x)$, for $z\in Z$ and $x\in G$. 
By the classical DiPerna-Lions theory 
(cf.~\cite[Lemma II.1]{DL89} or \cite[Lemma 2.3]{Lions:NSI}), 
$c_\varepsilon(z,\cdot)\overset{\varepsilon\downarrow 0}{\longrightarrow} 0$ 
in $L^p(K)$ for $\mu$-a.e.~$z\in Z$. Besides,
$$
\norm{c_\varepsilon(z,\cdot)}_{L^p(K)}\lesssim
\norm{g(z,\cdot)}_{L^{p_1}(K)} 
\norm{V(z,\cdot)}_{W^{1,p_2}(K;\R^d)},
$$
for $\mu$-a.e.~$z\in \Z$. Suppose $q_1<\infty$ or $q_2<\infty$. We raise to the power 
$q$ this inequality and apply the generalized H\"older inequality
to demonstrate that the resulting right-hand side is an integrable function (i.e., 
a $\mu$-dominant integrable function). Therefore, by the dominated convergence 
theorem, we obtain the desired convergence result \eqref{eq:genDP2} 
as well the bound \eqref{eq:genDP1}. The case $q_1=q_2=\infty$ is treated 
analogously. Indeed, for any $q\in [1,\infty)$,
$$
\norm{c_\varepsilon(z,\cdot)}^q_{L^p(K)}
\lesssim \norm{g}^q_{L^{\infty}\left(Z;L^{p_1}(K)\right)} 
\norm{V}^q_{L^{\infty}\left(Z;W^{1,p_2}(K;\R^d)\right)},
$$
for $\mu$-a.e.~$z\in \Z$, and once again we have obtained 
a $\mu$-dominant integrable function and conclude by dominated convergence.
\end{proof}

\begin{remark}
In this paper, we apply Lemma \ref{lemma:gen-diperna-lions} with 
the finite measure space $(Z,\mu)$ equal to $\left(\Omega,\P\right)$, $\left([0,T],dt\right)$, 
or $\left(\Omega\times [0,T],\P\otimes dt\right)$.
\end{remark}

Our next lemma is about the convergence of 
a ``second order'' commutator. The lemma is taken from 
Punshon-Smith's preprint \cite{Punshon-Smith:2017aa}.

We say that a triple $(\alpha_1,\alpha_2,\beta)$ is 
\textit{(2)-admissible} if $\alpha_1,\alpha_2\in [1,\infty]$, 
$\frac{1}{\alpha_1} + \frac{2}{\alpha_2} \leq 1$, 
$\frac{1}{\beta}= \frac{1}{\alpha_1} + \frac{2}{\alpha_2}$ if 
$\alpha_1<\infty$ or $\alpha_2<\infty$, and 
$\beta\in [1,\infty)$ is arbitrary if $\alpha_1=\alpha_2=\infty$.

\begin{lemma}[Punshon-Smith]\label{lem:Punshon-Smith}
Suppose 
$$
g\in L^{p_1}_\mathrm{loc}(G), 
\qquad 
V\in W^{1,p_2}_\mathrm{loc}(G;\R^d),
$$
for some (2)-admissible triple $(p_1,p_2,p)$, and define
$$
\mathcal{C}_{\varepsilon}\left[g ,V\right]
:=\frac12\partial_{ij}\left(V^iV^j \, g\right)_\varepsilon
-V^i\partial_{ij}\left(V^j\, g\right)_\varepsilon
+ \frac12 V^i V^j\partial_{ij} g_\varepsilon. 
$$
For any compact subset $K\subset G$,
$$
\mathcal{C}_{\varepsilon}\left[g ,V\right]
-\frac12  \Bigl( \left(\dive V\right)^2
+\partial_iV^j\partial_jV^i\Bigr)g_\varepsilon
\overset{\varepsilon\downarrow 0}{\longrightarrow}0
\quad \text{in $L^p(K)$}.
$$
Furthermore, there is a constant 
$C$ independent of $\varepsilon,p,g,V$ such that
\begin{align*}
	&\norm{\mathcal{C}_{\varepsilon}\left[g ,V\right]
	-\frac12 \Bigl(\left(\dive V\right)^2
	+ \partial_i V^j \partial_j V^i \Bigr) g_\varepsilon}_{L^p(K)} 
	\\ & \qquad
	\leq C \norm{V}_{W^{1,p_2}(K;\R^d)}^2 \norm{g}_{L^{p_1}(K)} .
\end{align*}

\end{lemma}
\begin{proof}
By \cite[Lemma 3.2]{Punshon-Smith:2017aa},
$$
\mathcal{C}_{\varepsilon}\left[g ,V\right]
- \frac12 \Bigl( \left(\dive V\right)^2+ \partial_iV^j\partial_jV^i\Bigr)g
\overset{\varepsilon\downarrow 0}{\longrightarrow}0
\quad \text{in $L^p(K)$},
$$
and 
\begin{align*}
	&\norm{\mathcal{C}_{\varepsilon}\left[g ,V\right]
	-\frac12 \Bigl( \left(\dive V\right)^2+\partial_i V^j \partial_j V^i \Bigr) g}_{L^p(K)} 
	\\ & \qquad
	\leq C \norm{V}_{W^{1,p_2}(K;\R^d)}^2\norm{g}_{L^{p_1}(K)},
\end{align*}
for some constant $C$ independent of $\varepsilon,p,g,V$. 
The lemma follows from this, the triangle inequality, 
and the bound ($\frac{1}{p}=\frac{1}{p_1}+\frac{1}{p_2/2}$)
\begin{align*}
	& \norm{\Bigl( \left(\dive V\right)^2+\partial_i V^j \partial_j V^i \Bigr) 
	\bigl(g-g_\varepsilon\bigr)}_{L^p(K)} 
	\\ & \quad \leq 
	\norm{\left(\dive V\right)^2+\partial_i V^j \partial_j V^i}_{L^{p_2/2}(K)}
	\norm{g-g_\varepsilon}_{L^{p_1}(K)}
	\lesssim \norm{V}_{W^{1,p_2}(K)}^2
	\norm{g}_{L^{p_1}(K)}.
\end{align*}
\end{proof}

Let us also state the following generalization 
of Lemma \ref{lem:Punshon-Smith}, which is 
analogous to Lemma \ref{lemma:gen-diperna-lions} (the 
proof is also the same).

\begin{lemma}\label{lem:Punshon-Smith-gen}
Let $(Z,\mu)$ be a finite measure space. Suppose
$$
g\in L^{q_1}\left(Z;L^{p_1}_\mathrm{loc}(G)\right), 
\qquad 
V\in L^{q_2}\left(Z;W^{1,p_2}_\mathrm{loc}(G;\R^d)\right),
$$ 
for some (2)-admissible triples $(p_1,p_2,p)$, $(q_1,q_2,q)$. 
Then, for any compact $K\subset G$,
\begin{align*}
&\norm{2\mathcal{C}_{\varepsilon}\left[g,V\right] 
-g_\varepsilon\left(\dive V\right)^2
	-g_\varepsilon \partial_i V^j \partial_j V^i
}_{L^q\left(Z;L^p(K)\right)}
\\ & \qquad
\leq  C \norm{g}_{L^{q_1}\left(Z;L^{p_1}(K)\right)} 
\norm{V}_{L^{q_2}\left(Z;W^{1,p_2}(K;\R^d)\right)}^2,
\end{align*}
for some constant $C$ that does not depend 
on $\varepsilon,p,g,V$. Furthermore, 
\[
\norm{2\mathcal{C}_{\varepsilon}\left[g,V\right]
-g_\varepsilon\left(\dive V\right)^2
-g_\varepsilon \partial_i V^j \partial_j V^i}_{L^q\left(Z;L^p(K)\right)}
\overset{\varepsilon\downarrow 0}{\longrightarrow}0.
\]
\end{lemma}

On several occasions we  use the following 
basic convergence lemma:

\begin{lemma}\label{lem:product-limit}
Fix $r\in [1,\infty]$ and $H\in C_b(\R)$. Let $\seq{f_j}_{j\ge 1}$ be 
a sequence in $L^r(\Omega \times [0,T]\times M)$ 
converging to $f$ in $L^r(\Omega\times [0,T]\times M)$. 
Then, as $j\to\infty$,
$$
H(f_j)f_j \to H(f)f \quad \text{in $L^r(\Omega\times [0,T]\times M)$}.
$$
\end{lemma}

\begin{proof}
We can assume $r<\infty$, as the result is trivial for $r=\infty$. 
Fix an arbitrary subsequence $\seq{f_{j_n}}_{n\ge 1}
\subset \seq{f_j}_{j\ge 1}$. Then, by the ``inverse dominated 
convergence" theorem, there exists a sub-subsequence 
$\seq{f_{j_{n_k}}}_{k\ge 1}\subset \seq{f_{j_n}}_{n\ge 1}$ which converges 
a.e.~to $f$, and there exists a function $g\in L^r$ that 
dominates $\seq{f_{j_{n_k}}}_{k\ge 1}$, see \cite[Theorem 4.9]{Brezis:2010aa}. 
Clearly, $H(f_{j_{n_k}})\to H(f)$ a.e.~as $k\to \infty$, and
$$
\sup_k \norm{H\left(f_{j_{n_k}}\right)}_{L^\infty}<\infty, 
\qquad 
H(f)\in L^\infty(\Omega\times[0,T]\times M).
$$
Therefore, by the dominated convergence theorem,
$$
H\left(f_{j_{n_k}}\right) f_{j_{n_k}} 
\overset{k\uparrow \infty}{\longrightarrow} 
H(f)f
 \quad 
 \text{in $L^r(\Omega\times [0,T]\times M)$}.
$$
By the arbitrariness of the fixed subsequence 
and the uniqueness of the limit, the original sequence 
must converge as well.
\end{proof}

We will also need an easy variant of the previous lemma.

\begin{lemma}\label{lem:product-limit-2}
Fix $q\in [1,\infty)$. Lemma \ref{lem:product-limit} continues 
to hold with $L^r(\Omega \times [0,T]\times M)$ replaced 
by $L^q\left([0,T]; L^2(\Omega\times M)\right)$.
\end{lemma}

Finally, we recall (without proof) a simple result that has been 
utilized several times when passing to the limit 
in the space-weak formulation of the SPDE. 

\begin{lemma}\label{lem:classical-result}
Fix $r\in [1,\infty]$. Let $\seq{f_j}_{j\ge 1}$ be 
a sequence in $L^r(\Omega \times [0,T])$ 
converging to $f$ in $L^r(\Omega\times [0,T])$.
Consider the functions
$$
F_j(\omega,t):=\int_0^t f_j(\omega,s)\,ds,
\qquad
F(\omega,t) := \int_0^tf(\omega,s)\,ds.
$$
Then $F_j\to F$ in $L^r(\Omega\times [0,T])$ as $j\to\infty$.
\end{lemma}



\begin{thebibliography}{10}

\bibitem{Ambrosio:2004aa}
L.~Ambrosio.
\newblock Transport equation and {C}auchy problem for {$BV$} vector fields.
\newblock {\em Invent. Math.}, 158(2):227--260, 2004.

\bibitem{Attanasio:2011fj}
S.~Attanasio and F.~Flandoli.
\newblock Renormalized solutions for stochastic transport equations and the
  regularization by bilinear multiplication noise.
\newblock {\em Comm. Partial Differential Equations}, 36(8):1455--1474, 2011.

\bibitem{Amorim:2005aa}
P.~Amorim, M.~Ben-Artzi, and P.~G. LeFloch.
\newblock Hyperbolic conservation laws on manifolds: total variation estimates
  and the finite volume method.
\newblock {\em Methods Appl. Anal.}, 12(3):291--323, 2005.

\bibitem{Aubin}
T.~Aubin.
\newblock {\em Some nonlinear problems in {R}iemannian geometry}.
\newblock Springer Monographs in Mathematics. Springer-Verlag, Berlin, 1998.

\bibitem{Banyaga}
A.~Banyaga.
\newblock Formes-volume sur les vari\'{e}t\'{e}s \`a bord.
\newblock {\em Enseign. Math. (2)}, 20:127--131, 1974.

\bibitem{Beck:2014aa}
L.~Beck, F.~Flandoli, M.~Gubinelli, and M.~Maurelli.
\newblock {Stochastic ODEs and stochastic linear PDEs with critical drift:
  regularity, duality and uniqueness}.
\newblock arXiv 1401.1530.

\bibitem{Ben-Artzi:2007aa}
M.~Ben-Artzi and P.~G. LeFloch.
\newblock Well-posedness theory for geometry-compatible hyperbolic conservation
  laws on manifolds.
\newblock {\em Ann. Inst. H. Poincar\'e Anal. Non Lin\'eaire}, 24(6):989--1008,
  2007.

\bibitem{Brezis:2010aa}
H.~Brezis.
\newblock {\em Functional Analysis, Sobolev Spaces and Partial Differential
  Equations}.
\newblock Universitext. Springer New York, 2010.

\bibitem{Chow:2015aa}
P.-L. Chow.
\newblock {\em Stochastic partial differential equations}.
\newblock Advances in Applied Mathematics. CRC Press, Boca Raton, FL, second
  edition, 2015.

\bibitem{DL89}
R.~J. DiPerna and P.-L. Lions.
\newblock Ordinary differential equations, transport theory and {S}obolev
  spaces.
\newblock {\em Invent. Math.}, 98(3):511--547, 1989.

\bibitem{Duboscq:2016aa}
R.~Duboscq and A.~R\'{e}veillac.
\newblock Stochastic regularization effects of semi-martingales on random
  functions.
\newblock {\em J. Math. Pures Appl. (9)}, 106(6):1141--1173, 2016.

\bibitem{Dumas:1994aa}
H.~S. Dumas, F.~Golse, and P.~Lochak.
\newblock Multiphase averaging for generalized flows on manifolds.
\newblock {\em Ergodic Theory Dynam. Systems}, 14(1):53--67, 1994.

\bibitem{Elliott:2012aa}
C.~M. Elliott, M.~Hairer, and M.~R. Scott.
\newblock Stochastic partial differential equations on evolving surfaces and
  evolving {Riemannian} manifolds.
\newblock arXiv:1208.5958.

\bibitem{Fang:2011aa}
S.~Fang, H.~Li, and D.~Luo.
\newblock Heat semi-group and generalized flows on complete {R}iemannian
  manifolds.
\newblock {\em Bull. Sci. Math.}, 135(6-7):565--600, 2011.

\bibitem{Fedrizzi:2013aa}
E.~Fedrizzi and F.~Flandoli.
\newblock Noise prevents singularities in linear transport equations.
\newblock {\em J. Funct. Anal.}, 264(6):1329--1354, 2013.

\bibitem{Flandoli:2011aa}
F.~Flandoli, M.~Gubinelli, and E.~Priola.
\newblock Full well-posedness of point vortex dynamics corresponding to
  stochastic 2{D} {E}uler equations.
\newblock {\em Stochastic Process. Appl.}, 121(7):1445--1463, 2011.

\bibitem{Flandoli-Gubinelli-Priola}
F.~Flandoli, M.~Gubinelli, E.~Priola.
\newblock Well-posedness of the transport equation by stochastic perturbation.
\newblock {\em Invent. Math.}, 2010 180: 1-53.

\bibitem{Galimberti:2018aa}
L.~Galimberti and K.~H. Karlsen.
\newblock Well-posedness theory for stochastically forced conservation laws on
  {R}iemannian manifolds.
\newblock {\em J. Hyperbolic Differ. Equ.}, 16(3):519--593, 2019.

\bibitem{Galimberti:2021aa}
L.~Galimberti and K.~H. Karlsen.
\newblock Well-posedness of stochastic continuity equations on {R}iemannian
  manifolds.
\newblock Submitted, 2021.

\bibitem{Greene:1979aa}
R.~E. Greene and H.~Wu.
\newblock {$C^{\infty }$} approximations of convex, subharmonic, and
  plurisubharmonic functions.
\newblock {\em Ann. Sci. \'{E}cole Norm. Sup. (4)}, 12(1):47--84, 1979.

\bibitem{Gyongy:1993aa}
I.~Gy\"{o}ngy.
\newblock Stochastic partial differential equations on manifolds. {I}.
\newblock {\em Potential Anal.}, 2(2):101--113, 1993.

\bibitem{Gyongy:1997aa}
I.~Gy\"{o}ngy.
\newblock Stochastic partial differential equations on manifolds. {II}.
  {N}onlinear filtering.
\newblock {\em Potential Anal.}, 6(1):39--56, 1997.

\bibitem{Kunita}
H.~Kunita.
\newblock {\em Stochastic flows and stochastic differential equations}, volume~24 of
  {\em Cambridge studies in advanced mathematics}.
\newblock Cambridge University Press, 1990.

\bibitem{LeeSmooth}
J.~M. Lee.
\newblock {\em Introduction to smooth manifolds}, volume 218 of {\em Graduate
  Texts in Mathematics}.
\newblock Springer-Verlag, New York, 2003.

\bibitem{Le-Bris:2008pb}
C.~Le~Bris and P.-L. Lions.
\newblock Existence and uniqueness of solutions to {F}okker-{P}lanck type
  equations with irregular coefficients.
\newblock {\em Comm. Partial Differential Equations}, 33(7-9):1272--1317, 2008.

\bibitem{Lions:NSI}
P.-L. Lions.
\newblock {\em Mathematical topics in fluid mechanics. {V}ol. 1: Incompressible
  models}.
\newblock Oxford University Press, New York, 1996.

\bibitem{Nash:1956aa}
J.~Nash.
\newblock The imbedding problem for {R}iemannian manifolds.
\newblock {\em Ann. of Math. (2)}, 63:20--63, 1956.

\bibitem{Neves:2015aa}
W.~Neves and C.~Olivera.
\newblock Wellposedness for stochastic continuity equations with
  {L}adyzhenskaya-{P}rodi-{S}errin condition.
\newblock {\em NoDEA Nonlinear Differential Equations Appl.}, 22(5):1247--1258,
  2015.

\bibitem{Neves:2016aa}
W.~Neves and C.~Olivera.
\newblock Stochastic continuity equations--a general uniqueness result.
\newblock {\em Bulletin of the Brazilian Mathematical Society, New Series},
  47:631--639, 2016.

\bibitem{Mohammed:2015aa}
S.-E.~A. Mohammed, T.~K. Nilssen, and F.~N. Proske.
\newblock Sobolev differentiable stochastic flows for {SDE}s with singular
  coefficients: {A}pplications to the transport equation.
\newblock {\em Ann. Probab.}, 43(3):1535--1576, 2015.

\bibitem{Protter}
P.~Protter.
\newblock {\em Stochastic integration and differential equations}, volume~21 of
  {\em Applications of Mathematics (New York)}.
\newblock Springer-Verlag, Berlin, 1990.

\bibitem{Punshon-Smith:2017aa}
S.~Punshon-Smith.
\newblock Renormalized solutions to stochastic continuity equations with rough
  coefficients.
\newblock arXiv 1710.06041.

\bibitem{Punshon-Smith:2018aa}
S.~Punshon-Smith and S.~Smith.
\newblock On the {B}oltzmann equation with stochastic kinetic transport: global
  existence of renormalized martingale solutions.
\newblock {\em Arch. Ration. Mech. Anal.}, 229(2):627--708, 2018.

\bibitem{Revuz:1999aa}
D.~Revuz and M.~Yor.
\newblock {\em Continuous martingales and {B}rownian motion}, volume 293 of
  {\em Grundlehren der Mathematischen Wissenschaften [Fundamental Principles of
  Mathematical Sciences]}.
\newblock Springer-Verlag, Berlin, third edition, 1999.

\bibitem{Rossmanith:2004aa}
J.~A. Rossmanith, D.~S. Bale, and R.~J. LeVeque.
\newblock A wave propagation algorithm for hyperbolic systems on curved
  manifolds.
\newblock {\em J. Comput. Phys.}, 199(2):631--662, 2004.

\bibitem{Zhang:2010aa}
X.~Zhang.
\newblock Stochastic flows of {SDE}s with irregular coefficients and stochastic
  transport equations.
\newblock {\em Bull. Sci. Math.}, 134(4):340--378, 2010.

\end{thebibliography}
\end{document}